\RequirePackage{rotating}
\documentclass[a4paper,twopage,reqno,12pt]{amsart}
\usepackage{eurosym}
\usepackage[top=30mm,right=30mm,bottom=30mm,left=30mm]{geometry}
\usepackage{tabularx}
\usepackage{graphicx}
\usepackage{longtable}
\usepackage{makecell}
\usepackage{booktabs}
\usepackage[flushleft]{threeparttable}
\usepackage{array}
\usepackage{amsmath}
\usepackage{amssymb}
\usepackage{amsfonts}
\usepackage{amsthm}
\usepackage{amstext}
\usepackage{amsbsy}
\usepackage{amsopn}
\usepackage{amscd}
\usepackage{enumerate}
\usepackage[caption=false]{subfig}
\usepackage{listings}
\usepackage{seqsplit}
\usepackage{color}
\usepackage{cancel}
\usepackage{xcolor}
\usepackage[colorlinks]{hyperref}
\usepackage[hyperpageref]{backref}
\usepackage{multirow}
\usepackage{longtable}
\usepackage{float}

\usepackage{lscape}
\usepackage{pdflscape}
\usepackage{url}
\usepackage{adjustbox}
\usepackage{rotating}
\usepackage{etoolbox}
\usepackage[dvipsnames] {xcolor}
\usepackage{comment}
\usepackage{booktabs}
\newtheorem{theorem}{Theorem}[section]
\newtheorem{corollary}[theorem]{Corollary}
\newtheorem{lemma}[theorem]{Lemma}
\newtheorem{proposition}[theorem]{Proposition}

\AtEndEnvironment{proof}{\setcounter{claim}{0}}
\theoremstyle{definition}
\newtheorem{construction}[theorem]{Construction}

\newtheorem{example}[theorem]{Example}

\numberwithin{equation}{section}

\newcommand{\GL}{\mathrm{GL}}

\newcommand{\PSL}{\mathrm{PSL}}

\newcommand{\PGL}{\mathrm{PGL}}
\newcommand{\PGaL}{\mathrm{P\Gamma L}}

\newcommand{\PSiL}{\mathrm{P \Sigma L}}

\newcommand{\PSigL}{\mathrm{P\Sigma L}}

\newcommand{\A}{\mathrm{A}}

\renewcommand{\S}{\mathrm{S}}

\newcommand{\C}{\mathrm{C}}
\newcommand{\D}{\mathrm{D}}

\newcommand{\Aut}{\mathrm{Aut}}

\newcommand{\Out}{\mathrm{Out}}

\newcommand{\PG}{\mathrm{PG}}

\newcommand{\Dmc}{\mathcal{D}}

\renewcommand{\leq}{\leqslant}
\renewcommand{\geq}{\geqslant}

\begin{document}
\title[Flag-transitive $2$-designs] {$2$-designs admitting a flag-transitive automorphism group with socle $\PSL(2,q)$}
\author[Liang]{Hongxue Liang}
\address{School of Mathematics, Foshan University, Foshan 528000, P.R. China}
\email{hongxueliang@fosu.edu.cn}
\author[Galici]{Mario Galici}
\address{Dipartimento di Matematica e Fisica “E. De Giorgi”, University of Salento, Lecce, Italy}
\email{mario.galici@unisalento.it}
\author[Liu]{Zhihui Liu}
\address{School of Mathematics, Foshan University, Foshan 528000, P.R. China}
\email{zhihuiliu889@gmail.com}
\author[Martinovi\'{c}]{Filip Martinovi\'{c}}
\address{University of Zagreb, Faculty of Electrical Engineering and Computing}
\email{filip.martinovic3@fer.unizg.hr}
\author[Montinaro]{Alessandro Montinaro}
\thanks{Corresponding author: Alessandro Montinaro}
\address{Dipartimento di Matematica e Fisica “E. De Giorgi”, University of Salento, Lecce, Italy}
\email{alessandro.montinaro@unisalento.it}
\author[Romano]{Eleonora Romano}
\address{Dipartimento di Matematica e Fisica “E. De Giorgi”, University of Salento, Lecce, Italy}
\email{eleonora.romano@unisalento.it}
\subjclass[MSC 2020:]{05B05; 05B25; 20B25}
\date{\today }

\begin{abstract}
$2$-designs admitting a flag-transitive automorphism group $G$ with socle $\PSL(2,q)$, where $q=p^{f}\geq 4$, are investigated in both the point-primitive and point-imprimitive cases. In the latter case, a complete classification is achieved, and three known examples occur, namely: the complementary designs of $\PG(3,2)$ and $\PG(3,4)$, and the $2$-$(36,8,4)$ design constructed by Devillers and Praeger in \cite{DP}. In the point-primitive case, apart from the Witt-Bose-Shrikhande linear spaces of even order $q$, $48$ sporadic examples are classified. Surprisingly, one of these numerical examples is the linear space with $v=496$ and $k=4$ admitting $\PGaL(2,2^{5})$ as a flag-transitive automorphism group, which was missing in the 1990 classification by Buekenhout et al. \cite{BDDKLS,Saxl,De}.
\end{abstract}

\maketitle

\section{Introduction and Main Result } \label{Intro}
A $2$-$(v,k,\lambda )$ design $\Dmc$ is a pair $(\mathcal{P},%
\mathcal{B})$ with a set $\mathcal{P}$ of $v$ points and a set $\mathcal{B}$
of $b$ blocks such that each block is a $k$-subset of $\mathcal{P}$ and each two distinct points are contained in exactly $\lambda $ blocks. The \emph{replication number} $r$ of $\mathcal{D}$ is the number of blocks containing a given point.
We say $\mathcal{D}$ is \emph{non-trivial} if $2<k<v-1$, and \emph{symmetric} if $v=b$. Unless otherwise specified, differently specified, all $2$-$(v,k,\lambda
)$ designs in this paper are assumed to be non-trivial.

An automorphism of $\mathcal{D}$ is a permutation of the point set which preserves the block
set. The set of all automorphisms of $\mathcal{D}$ with the composition of
permutations forms a group, denoted by $\mathrm{Aut(\mathcal{D})}$. For a
subgroup $G$ of $\mathrm{Aut(\mathcal{D})}$ acting point-transitively on $\mathcal{D}$, $G$ is said to be \emph{%
point-primitive} if $G$ acts primitively on $\mathcal{P}$, and said to be 
\emph{point-imprimitive} otherwise. In this setting, we also say that $%
\mathcal{D}$ is either \emph{point-primitive} or \emph{point-imprimitive}, respectively. A \emph{flag} of $\mathcal{D}$ is a pair $(x,B)$ where $x$ is
a point and $B$ is a block containing $x$. If $G\leq \mathrm{Aut(\mathcal{D})%
}$ acts transitively on the set of flags of $\mathcal{D}$, then we say that $%
G$ is \emph{flag-transitive} and that $\mathcal{D}$ is a \emph{flag-transitive design}.

There is a vast literature focusing on $2$-designs $\mathcal{D}$ admitting a
flag-transitive automorphism group $G$, with satisfactory classification results achieved 
in certain cases (for instance, see \cite{BDDKLS,United,DP,LM,MS,Mo,GM}).

Our paper is a contribution to the aforementioned research topic. We focus on $2$-designs admitting a flag-transitive automorphism group $G$ with socle isomorphic to $\PSL(2,q)$, where $q=p^{f}\geq 4$, in both the point-primitive and point-imprimitive cases. 
The motivation for analyzing such specific family of groups is manifold, as briefly outlined below:

\noindent \textbf{(1).} $\PSL(2,q) \leq G \leq \PGaL(2,q)$ is a \emph{spring} of examples of flag-transitive designs, most of which arise from remarkable geometric objects such as projective lines, conics, or hyperovals. Indeed, this is confirmed by the large number of non-isomorphic examples listed in this paper (almost fifty). Usually, but not always, most of the examples admitting $\PSL(n,q) \leq G \leq \PGaL(n,q)$ as a flag-transitive automorphism group predominantly occur when $n=2$.

\noindent \textbf{(2).} A surprisingly missing example in the existing classification of flag-transitive linear spaces
\cite{BDDKLS, Saxl,De}. Specifically, we identify a linear space with $v=496$ and $k=4$ admitting $G\cong \PGaL(2,2^{5})$ as a flag-transitive automorphism group, not contained in \cite{BDDKLS, Saxl,De}. Due to its significance, we provide two proofs, one theoretical and the other \texttt{GAP}-based, to demonstrate its existence and uniqueness up to isomorphism (note that this example is not the Witt-Bose-Shrikhande linear space of order $2^{5}$). 

\noindent \textbf{(3).} Our paper is meant to be a natural prosecution of the analysis of
$2$-designs admitting $\PSL(2,q)\trianglelefteq G\leq \mathrm{\PGaL(2,q)}$ as a flag-transitive automorphism group. This analysis was previously carried out by Delandtsheer \cite{De} and Saxl \cite{Saxl} for $\lambda =1$;
by Zhang and Zhou \cite{ZZ} when $\gcd(r,\lambda)=1$;
by Alavi, Bayat and Daneshkhah \cite{AlaviSymPSL2} for any value of $\lambda $ when $\mathcal{D}$ is symmetric; 
by Alavi et al. \cite{ABDT} and by Montinaro et al \cite{MZZZ} for $\lambda =2$;
and more recently by Zhang and Zhou \cite{ZZ2} when $\gcd(r,\lambda)=2$ and
by Chen et al. \cite{CWX} when $G$ is locally transitive and point-primitive with dihedral point-stabilizer.

\smallskip
Our method and result do not rely on any of the results listed in (3). 
In fact, when $\PSL(2,q)\trianglelefteq G\leq \mathrm{\PGaL(2,q)}$, 
all the aforementioned results can be deduced from our theorem by imposing their respective additional constraints on the parameters of $\mathcal{D}$ and/or on $G$. More precisely, we prove the following result.

\begin{theorem}\label{main}
Let $\mathcal{D}$ be a $2$-$(v,k,\lambda)$ design admitting $\PSL(2,q) \unlhd G\leq \PGaL(2,q)$ as a flag-transitive automorphism group. Then one of the following holds:
\begin{enumerate}
    \item $G$ acts point-primitively on $\mathcal{D}$, and one of the following holds:
\begin{enumerate}
    \item $\mathcal{D}$ is a $2$-$(q+1,k,\lambda)$ design, with the projective line $\PG(1,q)$ as a point set, as in Construction \ref{Build}, and $G$ acts point-$2$-transitively on $\mathcal{D}$; 
    \item $\mathcal{D}$ and $G$ are as in Table \ref{tab:alldesigns};
    \item$\left(v,b,k,r\right)=\left(2^{f-1}(2^{f}-1),(2^{2f}-1)\lambda,2^{f-1},(2^{f}+1)\lambda\right)$ with $\lambda \mid 2f$. Moreover, $X_{\alpha}\cong D_{2(2^f+1)}$ and $X_{B}\leq Z^{f}_{2}$, and either $\mathcal{D}$ is the Witt-Bose-Shrirkhande linear space of order $2^{f}$ with $f\geq 3$, or $\lambda>1$ and $f \geq 8$;
    \item $\left(v,b,k,r\right)=\left(\frac{q(q-1)}{2},\frac{q(q+1)\lambda}{4},q-1,\frac{(q+1)\lambda}{2}\right)$ with $q>50$ and $\lambda \mid 4f$. Moreover, $X_{\alpha}\cong \D_{\frac{2(q+1)}{\gcd(2,q-1)}}$ and $X_{B}\leq \D_{\frac{2(q-1)}{\gcd(2,q-1)}}$.
\end{enumerate}
\item $G$ acts point-imprimitively on $\mathcal{D}$, and one of the following holds:
\begin{enumerate}
\item $\mathcal{D}$ is the $2$-$(15,8,4)$ design complementary to $\PG(3,2)$, and $ G\cong
\PGL(2,5)$;
\item $\mathcal{D}$ is the $2$-$(36,8,4)$ design as in \cite[Construction 9]{DP}, and $G\cong \PSigL(2,9)$;
\item $\mathcal{D}$ is the symmetric $2$-$(85,64,48)$ design complementary to $\PG(3,4)$ and $G\cong \PGaL (2,2^4)$.
\end{enumerate}
\end{enumerate}
\end{theorem}

\noindent \textbf{Outline of Theorem \ref{main}}. Theorem \ref{main} provides a complete classification in the point-imprimitive case. In the point-primitive case, two infinite families of parameters occur, namely (1.c) and (1.d), and \emph{currently are the subject of study by the authors via geometric methods}. In case (1.a), $G$ acts point-$2$-transitively on the $2$-$(q+1,k\lambda)$ design $\mathcal{D}$. It is a well known fact that $G$ has a unique $2$-transitive permutation representation of degree $q+1$, namely the one on the projective line $\PG(1,q)$, where $q=p^{f}$ and $f \geq 1$. Hence, we may identify the point set of $\mathcal{D}$ with $\PG(1,q)$. Then $\mathcal{D}$ is obtained as follows:
\begin{construction} \label{Build}
Let $H$ be any subgroup of $G$ of order greater than $2$ and $Z_{q}\neq H$, and hence let $B$ be any $H$-orbit of length $2<k<q+1$. Then $\mathcal{D}=(\PG(1,q),B^{G})$ is a $2$-design since $G$ acts point-$2$-transitively on $\PG(1,q)$. Moreover, it is flag-transitive since it is block-transitive and $H\leq G_{B}$ and $B$ is an $H$-orbit. It is immediate to see that $\mathcal{D}$ arises from this construction, if $v=q+1$ and $G$ acts flag-transitively and point-$2$-transitively on $\mathcal{D}$.      
\end{construction}

All cases listed in Theorem 1.1 occur. Moreover, apart from Lines 3--4 and 5--6, different lines in Table \ref{tab:alldesigns} correspond to non-isomorphic designs. Due to the constraints on the designs parameters, there is no overlapping between $2$-designs in Table \ref{tab:alldesigns} and those in (1.a), (1.c) or (1.d), except for the examples in Lines 2 and 4, since $\PSL(2,2^{2})\cong \PSL(2,5)$ and $\PGaL(2,2^{2})\cong \PGL(2,5)$. 
One can see that, for the cases in Lines 1--35 of Table \ref{tab:alldesigns},
the group acts point-$2$-transitively; hence, these designs are obtained in a 
way similar to Construction \ref{Build},
but without the assumption $v=q+1$. The examples in Lines 36--39 were constructed in \cite{MF}. The designs in Lines40--49 are presumably new.

We provided hypertext links in Table \ref{tab:alldesigns} for the sporadic $\mathcal{D}$ and $G$, so that, clicking on the line of the table, the reader is redirected to the corresponding base block of $\mathcal{D}$ and generators of $G$ presented in Section \ref{Appendix}.

\noindent \textbf{Outline of the proof strategy}. Firstly, we analyze the case where $G$ acts flag-transitively and point-primitively on $\mathcal{D}$. We prove that socle $X\cong \PSL(2,q)$ acts point-primitively on $\mathcal{D}$ by using \cite{gmaxi} and then we use \cite{Di, COT} to determine the structure of $X_{\alpha}$, where $\alpha$ is a point of $\mathcal{D}$, thereby we derive the admissible structure of $G_{\alpha}$. The admissible subdegrees of $G$ are then determined by using \cite{Ka,FI}. We use this information to determine the replication number $r$ of $\mathcal{D}$ since $r$ must divide any integral linear combination of these subdegrees and satisfies $r>(\lambda \cdot v)^{1/2}$. The numerical set of admissible parameters designs are settled by using the package \texttt{Design} \cite{Design} of \texttt{GAP} \cite{GAP}; the remaining cases are analyzed by using some remarkable geometry of the group $G$, namely the fact that $G$ preserves a conic or a hyperoval in $\PG(2,q)$ depending on whether $q$ is odd or even, respectively; hence, the point set of $\mathcal{D}$ can be identified with the set of external lines of the conic or the hyperoval, respectively.

Once the case of flag-transitive and point-primitive $2$-design is settled, we use this result to completely determine $\mathcal{D}$ when the design admits a flag-transitive and point-imprimitive automorphism group $G$. Here, $G$ preserves a non-trivial partition $\Sigma$ of the point set, and we apply the Camina-Zieschang Theorem \cite{CZ} to 'factorize' $\mathcal{D}$ into two flag-transitive $2$-designs: one induced on $\Delta$, where $\Delta \in \Sigma$, say $\mathcal{D}_{0}$, the other on $\Sigma$, say $\mathcal{D}_{1}$. Choosing a maximal $G$-invariant partition with respect to set-theoretic containment, it implies that $\mathcal{D}_{1}$ is flag-transitive and point-primitive, hence it is known thanks to previous investigation. 
Combining this information together with some  constraints between the parameter of $\mathcal{D}_{1}$ and $\mathcal{D}_{0}$, derived from the Camina-Zieschang Theorem and the structure of $G$, we completely determine $\mathcal{D}$ and $G$. 

\noindent \textbf{Outline of the structure of the paper}. The paper consists of six sections, which are briefly described below. Section \ref{Intro} is devoted to the introduction and the establishment of Theorem \ref{main}, our main result. Section \ref{Examples} is dedicated to the description and proof of the linear space with $v=496$ and $k=4$ admitting $G\cong \PGaL(2,2^{5})$ as a flag-transitive automorphism group missed in \cite{BDDKLS, Saxl,De}.
For the other sporadic examples, their base blocks and group generators can be found in Section \ref{Appendix}. In the first part of Section \ref{FT-PP-Section}, we provide some useful tools from Design theory, Group theory and Finite Geometry; in the second part, we prove Theorem \ref{FT-PP-reduction}, a reduction theorem for the case where $G$ acts point-primitively on $\mathcal{D}$. In Section \ref{FTPI}, we state the Camina-Zieschang Theorem and apply it together with Theorem \ref{FT-PP-reduction} to prove Theorem \ref{FTPI}, which provides a classification of $\mathcal{D}$ when $G$ is flag-transitive point-imprimitive on $\mathcal{D}$. Section \ref{FT-PP-Geometry-Section} is dedicated to the completion of the proof of Theorem \ref{main}. In particular, we settle one of the infinite families of $2$-design parameters arising from Theorem \ref{FT-PP-reduction}.
Finally, in Section \ref{Appendix}, the last section, we provide a base block of $\mathcal{D}$ and generators of $G$ for each of the $2$-designs and their corresponding groups $G$ recorded in Table \ref{tab:alldesigns}.

\begin{table} 
	\centering
	\caption{Final table of all non-isomorphic designs}
	\scalebox{0.7}{
		\begin{tabular}{ccccccccccccc}
 Line  & $   q$ & $v$   & $b$    & $r$   & $k$  & $\lambda$ & $G$   &          
			$G_\alpha$         & $G_B$            & $\Aut(\mathcal{D})$ & location \\
            \toprule
 \ref{533} & $5$ & $5$   & $10$   & $6$   & $3$  & $3$       & $\PSL(2,5)$             & 
			$A_{4}$               & $S_{3}$             & $\PGL(2,5)$                & \ref{case4} \\
\ref{632} & $2^2$ & $6$   & $10$   & $5$  & $3$  & $2$       & $\PSL(2,2^2)$  & 
			$D_{10}$               & $S_{3}$ & $\PSL(2,2^{2})$ & \ref{MintX7Cor}\\
 \ref{634_1} & $2^2$ & $6$   & $20$   & $10$  & $3$  & $4$       & $\PGaL(2,2^2)$  & 
			$Z_{5}: Z_{4}$               & $S_{3}$ & $\PSigL(2,3^{2})$                &  \ref{MintX7Cor}\\            
\ref{634_2} & $3 ^2$ &    &    &   &  &       & $\PSL(2,3^{2})$  & 
			$A_{5}$               & $Z_{2}^3 : Z_{2}$ &                 & \ref{case6} \\ 
\ref{646_1} & $2^2$ & $6$   & $15$   & $10$  & $4$  & $6$       & $\PGaL(2,2^2)$  &      
			$Z_{5} : Z_{4}$          & $D_{8}$             & $\PSigL(2,3^{2})$                & \ref{MintX7}, \ref{MinTX8}, \ref{MintX9} \\
 \ref{646_2}& $3^2$ &   &   &   &   &        & $\PSL(2,3^2)$  &      
			$A_{5}$          & $Z_{2}^3:Z_2$             &                 & \ref{case6} \\ \ref{731}& 7& 7& 7& 3 & 3 & 1 & $\PSL(2,7)$ & $S_4$ & $S_4$ & $\PSL(2,7)$ & \ref{case5} \\
 \ref{742}& 7& 7& 7& 4 & 4 & 2 & $\PSL(2,7)$ & $S_4$ & $S_3$ & $\PSL(2,7)$ & \ref{case5} \\ \ref{734}& $7$ & $7$   & $28$   & $12$  & $3$  & $4$       & $\PSL(2,7)$           & 
			$S_{4}$               & $S_{3}$             & $\PSL(2,7)$          & \ref{case5} \\
 \ref{1042}& $2^2$ & $10$  & $15$   & $6$   & $4$  & $2$       & $\PGaL(2,2^2)$       &
			$D_{12}$              & $D_{8}$             & $\PSigL(2,3^{2})$                & \ref{case7}, \ref{MintX456}, \ref{MintX7Cor}, \ref{MinTX8} \\ 
\ref{1133}& $11$ & $11$  & $55$   & $15$  & $3$  & $3$       & $\PSL(2,11)$      &
			$A_{5}$               & $D_{12}$            & $\PSL(2,11)$         & \ref{case6} \\
 \ref{1136}& $11$ & $11$  & $110$  & $30$  & $3$  & $6$       & $\PSL(2,11)$      &
			$A_{5}$               & $S_{3}$             & $\PSL(2,11)$         & \ref{case6} \\
 \ref{1146}& $11$ & $11$  & $55$   & $20$  & $4$  & $6$       & $\PSL(2,11)$    &  
			$A_{5}$               & $A_{4}$             & $\PSL(2,11)$         & \ref{case6} \\
 \ref{1152}& $11$ & $11$  & $11$   & $5$  & $5$  & $2$       & $\PSL(2,11)$      &
			$A_{5}$               & $A_{5}$            & $\PSL(2,11)$         & \ref{case6} \\
 \ref{11512}& $11$ & $11$  & $66$   & $30$  & $5$  & $12$      & $\PSL(2,11)$     & 
			$A_{5}$               & $D_{10}$            & $\PSL(2,11)$         & \ref{case6} \\
 \ref{1163}& $11$ & $11$  & $11$   & $6$   & $6$  & $3$       & $\PSL(2,11)$      & 
			$A_{5}$               & $A_{5}$             & $\PSL(2,11)$         & \ref{case6} \\
 \ref{11615}& $11$ & $11$  & $55$   & $30$  & $6$  & $15$      & $\PSL(2,11)$     & 
			$A_{5}$               & $D_{12}$            & $\PSL(2,11)$         & \ref{case6} \\
 \ref{1584_2}& $3^2$ & $15$  & $15$   & $8$  & $8$  & $4$      & $\PSL(2,3^2)$             & 
			$S_4$               & $S_4$             & $A_{8}$                & \ref{case1} \\
 \ref{2832_1}& $2 ^3$ & $28$  & $252$  & $27$  & $3$  & $2$       & $\PGaL(2,2^{3}) $  &
			$Z_{9} : Z_{6}$          & $Z_{6}$             & $\PGaL(2,2^{3}) $     & \ref{MintX7}, \ref{MintX9}, \ref{Ifin} \\
 \ref{2832_2}& $2 ^3$ & $28$  & $252$  & $27$  & $3$  & $2$       & $\PGaL(2,2^{3}) $  & 
			$Z_{9} : Z_{6}$          & $S_{3}$             & $\PGaL(2,2^{3}) $     & \ref{MintX7}, \ref{MintX9}, \ref{Ifin}\\
 \ref{2834}& $2 ^3$ & $28$  & $504$  & $54$  & $3$  & $4$       & $\PGaL(2,2^{3}) $  &
			$Z_{9} : Z_{6}$          & $Z_{3}$             & $\PGaL(2,2^{3}) $     & \ref{MintX7}, \ref{Ifin} \\
 \ref{2865}& $2 ^3$ & $28$  & $126$  & $27$  & $6$  & $5$       & $\PGaL(2,2^{3}) $  &
			$Z_{9} : Z_{6}$          & $A_{4}$             & $\PGaL(2,2^{3}) $     & \ref{MintX7}, \ref{MinTX8} \\
 \ref{28610_1}& $2 ^3$ & $28$  & $252$  & $54$  & $6$  & $10$      & $\PGaL(2,2^{3}) $  &
			$Z_{9} : Z_{6}$          & $S_{3}$             & $\PGaL(2,2^{3}) $     & \ref{MintX7}, \ref{MinTX8}, \ref{MintX9} \\
 \ref{28610_2} & $2 ^3$ & $28$  & $252$  & $54$  & $6$  & $10$      & $\PGaL(2,2^{3}) $  &
			$Z_{9} : Z_{6}$          & $Z_{6}$             & $\PGaL(2,2^{3}) $     & \ref{MintX7}, \ref{MinTX8}, \ref{MintX9}\\
 \ref{28610_3}& $2 ^3$ & $28$  & $252$  & $54$  & $6$  & $10$      & $\PGaL(2,2^{3}) $  &
			$Z_{9} : Z_{6}$          & $Z_{6}$             & $\PGaL(2,2^{3}) $     & \ref{MintX7}, \ref{MinTX8}, \ref{MintX9} \\
 \ref{28610_4}& $2 ^3$ & $28$  & $252$  & $54$  & $6$  & $10$      & $\PGaL(2,2^{3}) $  &
			$Z_{9} : Z_{6}$          & $Z_{6}$             & $\PGaL(2,2^{3}) $     & \ref{MintX7}, \ref{MinTX8}, \ref{MintX9} \\
 \ref{28610_5}& $2 ^3$ & $28$  & $252$  & $54$  & $6$  & $10$      & $\PGaL(2,2^{3}) $  &
			$Z_{9} : Z_{6}$          & $Z_{6}$             & $\PGaL(2,2^{3}) $     & \ref{MintX7}, \ref{MinTX8}, \ref{MintX9} \\
 \ref{2872}& $2 ^3$ & $28$  & $36$  & $9$  & $7$  & $2$      & $\PSL(2,2^{3}) $  &
			$D_{18}$          & $D_{14}$             & $\PGaL(2,2^{3}) $     & \ref{MintX7Cor}\\     
\ref{2876}& $2 ^3$ & $28$  & $108$  & $27$  & $7$  & $6$      & $\PGaL(2,2^{3}) $  &
			$Z_{9}:Z_{6}$          & $D_{14}$             & $\PGaL(2,2^{3}) $     & \ref{MintX7Cor}\\    
 \ref{2898_1}& $2 ^3$ & $28$  & $84$   & $27$  & $9$  & $8$       & $\PGaL(2,2^{3}) $  &
			$Z_{9} : Z_{6}$          & $D_{18}$            & $\PGaL(2,2^{3}) $     & \ref{MinTX8} \\
 \ref{2898_2}& $2 ^3$ & $28$  & $84$   & $27$  & $9$  & $8$       & $\PGaL(2,2^{3}) $  &
			$Z_{9} : Z_{6}$          & $Z_{3} \times S_{3}$        & $\PGaL(2,2^{3}) $     & \ref{MinTX8} \\
 \ref{28916}& $2 ^3$ & $28$  & $168$  & $54$  & $9$  & $16$      & $\PGaL(2,2^{3}) $  & 
			$Z_{9} : Z_{6}$          & $Z_{9}$             & $\PGaL(2,2^{3}) $     & \ref{MinTX8} \\ \ref{281211_1}& $2 ^3$ & $28$  & $63$   & $27$  & $12$ & $11$      & $\PGaL(2,2^{3}) $  &
			$Z_{9} : Z_{6}$          & $Z_{2} \times A_{4}$        & $Sp(6,2)$            & \ref{MintX9} \\
 \ref{281211_2}& $2 ^3$ & $28$  & $63$   & $27$  & $12$ & $11$      & $\PGaL(2,2^{3}) $  &
			$Z_{9} : Z_{6}$          & $Z_{2} \times A_{4}$        & $\PGaL(2,2^{3}) $     & \ref{MintX9} \\
 \ref{281834}& $2 ^3$ & $28$  & $84$   & $54$  & $18$ & $34$      & $\PGaL(2,2^{3}) $  &
			$Z_{9} : Z_{6}$          & $Z_{3} \times S_{3}$        & $\PGaL(2,2^{3}) $     & \ref{MinTX8} \\
 \ref{3662}& $2 ^3$ & $36$  & $84$   & $14$  & $6$  & $2$       & $\PSL(2,2^{3})$       & 
			$D_{14}$              & $S_{3}$             & $\PGaL(2,2^{3}) $     & \ref{case7} \\
 \ref{3666_1} & $2 ^3$ & $36$  & $252$  & $42$  & $6$  & $6$       & $\PGaL(2,2^{3}) $  & 
			$Z_{7} : Z_{6}$          & $Z_{6}$             & $\PGaL(2,2^{3}) $     & \ref{case7}\\
 \ref{3666_2}& $2 ^3$ & $36$  & $252$  & $42$  & $6$  & $6$       & $\PGaL(2,2^{3}) $  & 
			$Z_{7} : Z_{6}$          & $Z_{6}$             & $\PGaL(2,2^{3}) $     & \ref{case7} \\
 \ref{3666_3}& $2 ^3$ & $36$  & $252$  & $42$  & $6$  & $6$       & $\PGaL(2,2^{3}) $  &
			$Z_{7} : Z_{6}$          & $Z_{6}$             & $\PGaL(2,2^{3}) $     & \ref{case7} \\
 \ref{3688}& $3 ^2$ & $36$  & $180$  & $40$  & $8$  & $8$       & $\PGaL(2,3^{2}) $ &
			$D_{10} : Z_{4}$   & $D_{8}$             & $\PGaL(2,3^{2}) $    & \ref{case8}, \ref{MintX456}, \ref{MintX7Cor}, \ref{MinTX8}, \ref{MintX123} \\ 
\ref{12084_1}& $2 ^4$ & $120$ & $1020$ & $68$  & $8$  & $4$       & $\PSL(2,2^{4}):Z_2 $ &
			$Z_{17} : Z_{4}$         & $D_{8}$             & $\PSL(2,2^{4}):Z_2 $    & \ref{MintX7}, \ref{MinTX8}, \ref{MintX123}, \ref{MintX9}, \ref{MintX9Cor} \\
 \ref{12084_2}& $2 ^4$ & $120$ & $1020$ & $68$  & $8$  & $4$       & $\PGaL(2,2^{4}) $ &
			$Z_{17} : Z_{8}$         & $(Z_{4} \times Z_{2}) : Z_{2}$ & $\PGaL(2,2^{4}) $    & \ref{MintX7}, \ref{MinTX8}, \ref{MintX123}, \ref{MintX9}, \ref{MintX9Cor} \\
 \ref{12088}& $2 ^4$ & $120$ & $2040$ & $136$ & $8$  & $8$       & $\PGaL(2,2^{4}) $ & 
			$Z_{17} : Z_{8}$         & $D_{8}$             & $\PGaL(2,2^{4}) $    & \ref{MintX7}, \ref{MinTX8}, \ref{MintX123}, \ref{MintX9}, \ref{MintX9Cor} \\
 \ref{49641}& $2^5$& $496$& $20460$& $165$ & $4$ & $1$ & $\PGaL(2,2^5)$ & $D_{66}:Z_{5}$ & $Z_{2}^3$ & $\PGaL(2,2^5)$ & \ref{MintX9} \\
 \ref{4962214}& $2 ^5$ & $496$ & $7440$ & $330$ & $22$ & $14$      & $\PGaL(2,2^{5})$ & 
			$D_{66}:Z_{5}$ & $D_{22}$            & $\PGaL(2,2^{5})$    & \ref{MinTX8} \\
\ref{2016324_1}& $2 ^6$ & $2016$ & $16380$ & $260$ & $32$ & $4$      & $\PSL(2,2^{6}):Z_{2}$ & 
			$D_{66}:Z_{2}$ & $Z_{2}^{4}:Z_{2}$           & $\PSL(2,2^{6}):Z_{2}$    & \ref{MintX9Cor} \\
\ref{2016324_2}& $2 ^6$ & $2016$ & $16380$ & $260$ & $32$ & $4$      & $\PGaL(2,2^{6})$ & 
			$D_{130}:Z_{6}$ & $Z_{2}^{4}:Z_{6}$            & $\PGaL(2,2^{6})$    & \ref{MintX9Cor} \\
\ref{20163212_1}& $2 ^6$ & $2016$ & $49140$ & $780$ & $32$ & $12$      & $\PGaL(2,2^{6})$ & 
			$D_{130}:Z_{6}$ & $Z_{2}^{4}:Z_{2}$            & $\PGaL(2,2^{6})$    & \ref{MintX9Cor} \\
\ref{20163212_2}& $2 ^6$ & $2016$ & $49140$ & $780$ & $32$ & $12$      & $\PGaL(2,2^{6})$ & 
          $D_{130}:Z_{6}$ & $Z_{2}^{4}:Z_{2}$            & $\PGaL(2,2^{6})$    & \ref{MintX9Cor} \\
        \bottomrule
		\end{tabular}
	}
	\label{tab:alldesigns}
 \end{table} 

\section{The new flag-transitive linear space with $v=496$ and $k=4$}\label{Examples}
In this section, we analyze the linear space \textcolor{green}{listed} in Line 44 of Table \ref{tab:alldesigns}. Before proceeding, some research history and preliminary remarks are necessary.

In 1986, Delandtsheer \cite{De} proved that the Witt-Bose-Shrikhande space is the unique linear space admitting $\PSL(2,q)$ as a flag-transitive automorphism group. Later, in 1990, Buekenhout et al. \cite{BDDKLS} announced the classification of linear spaces admitting a flag-transitive automorphism group, except when the automorphism group is a semilinear $1$-dimensional group. Nevertheless, it was only in 2002 that Saxl \cite{Saxl} provided a proof of the classification when the socle of the group is almost simple. In particular, it is proven in \cite[Lemma 7.1]{Saxl} that, if the socle of the automorphism group is $\PSL(2,q)$,
then the group must be flag-transitive, and hence the conclusion of \cite{De} holds. However, in \cite[Lemma 7.1]{Saxl}, for $q$ is even and $r\neq q+1$, it was asserted that no example occur for $\log_{2}q \leq 17$. We show below that this minor oversight led to a missing case in the aforementioned classification, namely, a $2$-$(496,4,1)$ linear space admitting $\PGaL(2,2^5)$ as a flag-transitive automorphism group.

We provide two proofs of the existence and uniqueness (up to isomorphism) of the flag-transitive $2$-$(496,4,1)$ linear space: one theoretical and the other by means of the package \texttt{Design} of \texttt{GAP}. For the theoretical proof, we first recall the following facts related to $\PSL(2,q)$, $q$ even, viewed as a collineation group of $\PG(2,q)$:

\noindent \textbf{(1)} An irreducible conic $\mathcal{C}$ of $\PG(2,q)$, $q=2^{f}$, is a $%
(q+1)$\emph{-arc}, namely a set of $q+1$ points no three of them collinear, by 
\cite[Lemma 7.7]{Hir}. Any line of $\PG(2,q)$ is either \emph{secant}, \emph{%
tangent} or \emph{external} according as it has $2$, $1$ or $0$ points in
common with $\mathcal{C}$, respectively. The set of secant, tangent or external lines has size $\frac{q(q+1)}{2}$, $q+1$, or $\frac{q(q-1)}{2}$, respectively, by \cite[Corollary 8.2]{Hir}. The tangent lines to $\mathcal{C}$
are all concurrent to a point $N$ called \emph{nucleus} of $\mathcal{C}$ by 
\cite[Corollary 7.11]{Hir}, and $\mathcal{J}=\mathcal{C}\cup \left\{
N\right\} $ is a $(q+2)$-arc called \emph{regular} \emph{hyperoval }(see \cite[%
Section 8.4]{Hir}). Clearly, the lines of $\PG(2,q)$ are either secants or external to $\mathcal{J}$, that is they have either $2$ or $0$ points in common with $%
\mathcal{J}$. The set $\mathcal{E}$ of the external lines to $\mathcal{J}$ coincides with the set of the external lines to $\mathcal{C}$, and hence it has size $q(q-1)/2$. The number of
points of $\PG(2,q)\setminus \mathcal{J}$ is $q^{2}-1$ and through each point
there are exactly $\frac{q}{2}+1$ secant lines to $\mathcal{J}$ and $q/2$ external lines to $\mathcal{J}$ by \cite[Corollary 8.8]{Hir}.

\noindent \textbf{(2)} $\PGL(3,q)$ has a unique conjugacy class of subgroups isomorphic
to $\PSL(2,q)$ by \cite[Table 8.3]{BHRD} (note that $\PSL(2,q) \cong \Omega (3,q)$ by \cite[Corollary 7.14]{Hir} is reducible
and not maximal in $\PGL(3,q)$ when $q$ is even. Further, the irreducible conics do not arise from
polarities in this case), and each of these groups is the stabilizer in $%
\PGL(3,q)$ of a suitable regular hyperoval of $\PG(2,q)$. The converse is also
true as a consequence of \cite[Theorem 7.4]{Hir}. In particular, each $\PGaL(2,q)$ fixes the nucleus of its invariant hyperoval and acts $2$-transitively on the remaining $q+1$ points of this one  by \cite[Corollary 7.15]{Hir}. 

\noindent \textbf{(3)} $\PGL(3,q)$ has a unique conjugacy class of elements of order $2$, and if $%
\tau $ is any of these then $\tau $ is a $(P_{\tau },m_{\tau })$%
-elation of $\PG(2,q)$. That is, $\tau $ fixes each of the $q+1$ points of $%
m_{\tau }$ including $P_{\tau }$ and fixes setwise each of the $q+1$ lines
of $\PG(2,q)$ containing $P_{\tau }$ by \cite[Exercise IV.4.6]{HP}. No
other points or lines of $\PG(2,q)$ are fixed by $\tau $.

\noindent \textbf{(4)} If $T$ is any Sylow $2$-subgroup of $\PSL(2,q)$, then the following hold by \cite[Lemma 2.3]{MZZZ}:
\begin{enumerate}
    \item[(i)] $T$ fixes the nucleus $N$, a unique point $Q$ of $\mathcal{C}$ and acts regularly on $\mathcal{C}\setminus \left\{ Q\right\}$;
    \item[(ii)] $T$ fixes $m$ pointwise, where $m=NQ$, and acts semiregularly on $\PG(2,q)\setminus m$;
    \item[(iii)]  $\tau_{a}$, with $a \in \mathbb{F}_{q}$ is a $(P_{\tau_{a}},m)$-elation of $\PG(2,q)$;
    \item[(iv)]  Any cyclic subgroup of $\PSL(2,q)$ normalizing $T$ acts regularly on the set $m\setminus \left\{ N,Q\right\}=\left\{ P_{\tau _{a}}:a \in \mathbb{F}^{\ast}_{q}\right\}$;
    \item[(v)] If $E_{a}$ is the set of $q/2$ lines through $P_{\tau_{a}}$ which are external to $\mathcal{J}$ (see (2)), then $T$ acts transitively on $E_{a}$ with action kernel $\left\langle \tau _{a} \right\rangle$. Moreover, $\mathcal{E}$ is partitioned into $q-1$ orbits under $T$, namely, $E_{a}$ with $a\in \mathbb{F}^{\ast}_{q}$.
\end{enumerate}

\begin{example}
There is a unique $2$-$(496,4,1)$ linear space admitting $\PGaL(2,2^{5})$ as a flag-transitive automorphism group. Moreover, $\PGaL(2,2^{5})$ is its full automorphism group.    
\end{example}
\begin{proof}
After fixing a projective frame in $\PG(2,2^{5})$ with coordinates $%
(Y_{0}:Y_{1}:Y_{2})$, according to (1),  we may assume that 
\[
\mathcal{C}:Y_{0}Y_{2}+Y_{1}^{2}=0\text{.}
\]%
Hence, $\mathcal{J}=\mathcal{C}\cup \left\{ N\right\} $, with $N=Y_{\infty}=(0:1:0)$ by 
\cite[Corollary 7.12]{Hir}, is a regular hyperoval of  $\PG(2,2^{5})$. Again
by (1), there is a copy of $\PSL(2,2^{5})$ inside $\PSL(3,2^{5})$
preserving $\mathcal{J}$, say $X$. Note that, the collineation $\varsigma
:(Y_{0}:Y_{1}:Y_{2})\rightarrow (Y_{0}^{2}:Y_{1}^{2}:Y_{2}^{2})$ of $%
\PG(2,2^{5})$ induced by the automorphism of $\mathbb{F}_{2^{5}}$ preserves $%
\mathcal{C}$, $N$ and hence $\mathcal{J}$. Therefore $\mathcal{C}$, $N$ and $%
\mathcal{I}$ are left invariant by $G=X:\left\langle \varsigma \right\rangle
\cong \PGaL(2,2^{5})$.

For $a\in \mathbb{F}_{2^{5}}$, let $\tau _{a}$ be the collineation of $\PG(2,2^{5})$ represented by the matrix 
\[
\left( 
\begin{array}{ccc}
1 & 0 & 0 \\ 
0 & 1 & 0 \\ 
a^{2} & a & 1%
\end{array}%
\right) \text{,}
\]%
and let $T=\left\langle \tau _{a}:a\in \mathbb{F}_{2^{5}}\right\rangle $ be
a Sylow $2$-subgroup of $X$, and hence of $G$. Then $m=\ell_{\infty}=\left[0:0:1\right]$, $Q=(1:0:0)
$ by (4.i) and (4.ii). Let $P_{1}=(1:1:0)$ on $m$, then $\tau _{1}$ is a $(P_{1},m)$%
-elation of $\PG(2,2^{5})$ (see (4.iii)). Moreover, $T/\left\langle \tau _{1}\right\rangle $
acts regularly on $E_{1}$, the set of $2^{4}$ lines incident with $P_{1}$
and external to $\mathcal{J}$, which therefore can be endowed with the
structure of a $4$-dimensional $\mathbb{F}_{2}$-space. Now, $\left\langle
\tau_{1} \right\rangle $ fixes $P_{1}$ and permutes the $2^{4}$ lines in $%
E_{1}$. Moreover, $\left\langle \varsigma \right\rangle $ normalizes $%
T/\left\langle \tau_{1} \right\rangle $ since $\left\langle \varsigma
\right\rangle $ normalizes $T$ and centralizes $\tau _{1}$. Indeed, $\tau
_{a}^{\varsigma }=\tau _{a^{2}}$. Hence, $\left\langle \varsigma
\right\rangle $ induces a group of linear maps on the $4$-dimensional $%
\mathbb{F}_{2}$-space $E_{1}$ by \cite[Proposition 4.2]{Pass}. Actually, $%
\left\langle \varsigma \right\rangle $ induces a Sylow $5$-subgroup of $%
\GL(4,2)\cong A_{8}$, which also is a subgroup of a Singer cycle (see \cite[%
Section 4.2]{Hir}). Then there is a unique $\left\langle \varsigma
\right\rangle $-orbit on the set of the $2$-subspaces of $E_{1}$ which is a
(regular) $2$-spread $\mathcal{S}$ of $E_{1}$ by \cite[Lemma 2.1]{Dru}.

Let $\mathcal{D}$ be the incidence structure $(\mathcal{E},B^{G})$, where $B$
is any of the $2$-subspaces of $E_{1}$ belonging to the $\left\langle
\varsigma \right\rangle $-invariant $2$-spread of $E_{1}$. It has parameters 
$v=496$ and $k=4$. Moreover, $G_{B}\leq G_{P_{1}}=T:\left\langle \varsigma
\right\rangle $ since $B$ consists of $4$ lines of $PG(2,2^{5})$ concurrent in $P_{1}$. Now, since $T/\left\langle \tau _{1}\right\rangle $ and $%
\left\langle \varsigma \right\rangle $ act regularly and irreducibly on $%
E_{1}$, respectively, it follows that $G_{B}\cong Z_{2}^{3}$, with $\tau_{1}\in G_{B}$, and hence $%
b=20460$. Consequently, $r=165$ since $\mathcal{D}$ is a $1$-design by \cite[%
1.2.6]{Demb}.

Finally, let $\ell _{1},\ell _{2}\in \mathcal{E}$ with $\ell _{1}\neq \ell
_{2}$. Clearly, $\ell _{1}\cap \ell _{2}$ is a point of $\PG(2,2^{5})%
\setminus \mathcal{J}$. Since $G$ acts transitively on $\PG(2,2^{5})\setminus 
\mathcal{J}$, we have $\ell _{1}\cap \ell _{2}$ $=\left\{ P_{1}^{g}\right\} $ for some $g \in G$. Hence, $\ell _{1}^{g^{-1}},\ell _{2}^{g^{-1}}\in E_{1}$. Then there is $%
\tau _{a_{0}}\in T$ such that $\ell _{1}^{g^{-1}\tau _{a_{0}}}$ corresponds
to the null vector of $E_{1}$ (recall that $4$-dimensional $\mathbb{F}_{2}$-space on which $T$ induces the translation group). So, there is a unique element of $\mathcal{S}$
containing both $\ell _{1}^{g^{-1}\tau _{a_{0}}}$ and $\ell _{2}^{g^{-1}\tau
_{a_{0}}}$, that is to say $\ell _{1}^{g^{-1}\tau _{a_{0}}},\ell
_{2}^{g^{-1}\tau _{a_{0}}}\in B^{\varsigma ^{i_{0}}}$ for a unique $i_{0}\in
\left\{ 0,1,2,3,4\right\} $. Hence, $B^{\varsigma ^{i_{0}}\tau _{a_{0}}g}$
is the unique element of $B^{G}$ containing both $\ell _{1}$ and $\ell _{2}$%
. Thus, $\mathcal{D}$ is a $2$-$(496,4,1)$ design, that is a linear space.

The group $G$ acts flag-transitively on $\mathcal{D}$ since $G$ acts
block-transitively on $\mathcal{D}$ and $G_{B}$ acts transitively on $%
\mathcal{D}$. Moreover, since $\Aut(\mathcal{D})$ preserves $\mathcal{E}$, $%
\mathcal{J}$ and hence $\mathcal{C}$, it follows that $\Aut(\mathcal{D}%
)=G\cong \PGaL(2,2^{5})$.

Let $\mathcal{D}^{\prime}$ be any $2$-$(496,4,1)$ linear space admitting $G$ as a flag-transitive automorphism group. Then $G_{\alpha} \cong D_{66}:Z_{5}$, where $\alpha$ is any point of $\mathcal{D}^{\prime}$. Since $G$ has a unique conjugacy class of subgroups isomorphic to $D_{66}:Z_{5}$, and this is the stabilizer of a line of $\PG(2,2^{5})$ external to $\mathcal{C}$, we may identify the point set of $\mathcal{D}^{\prime}$ with $\mathcal{E}$, the set of external lines of $\mathcal{C}$ (or, equivalently to $\mathcal{H}$). Now, if $C$ is any block of $\mathcal{D}^{\prime}$, since $G_{C}\cong Z^{3}_{2}$, there is an element of $G_{C}$ fixing each line in $C$, and hence $C$ consists of $4$ concurrent lines of $\PG(2,2^{5})$ by (4.iii). Since $G$ acts transitively on $\PG(2,2^{5})\setminus \mathcal{J}$, we may assume that the concurrency point is $P_{1}$, therefore $C \subset E_{1}$. Now, since $\mathcal{D}^{\prime}$ is a linear space, then $C^{\left\langle \varsigma \right\rangle}$ is forced to be a $2$-spread of $E_{1}$, and hence $C^{\left\langle \varsigma \right\rangle}=B^{\left\langle \varsigma \right\rangle}$ by \cite[Lemma 2.1]{Dru}. Thus $C \in B^{G}$, and hence $\mathcal{D}^{\prime}=\mathcal{D}$ since $G$ is block-transitive. This proves the existence and the (up to isomorphism) uniqueness of $\mathcal{D}$.
\end{proof}

\bigskip

A coordinate description of $B$ is as follows. Let $E_{1}=\left\{\left[1:1:\omega ^{n}\right]:n\in \mathcal{N}\right\}$, where 
\[
\mathcal{N}=\{0,3,5,6,9,10,11,12,13,17,18,20,21,22,24,26\}\text{.}
\]%
Here, $\omega $ is a a primitive element of the copy of $\mathbb{F}_{2^{5}}$ obtained as the
splitting field of $z^{5}+z^{2}+1\in \mathbb{F}_{2}[z]$. The elements $\omega ^{n}$
with $n\in \mathcal{N}$ are all the elements $\mathbb{F}_{2^{5}}$ with
absolute trace equal to $1$. Furthermore, $\tau _{a}$ acts on on $E_{1}$ as 
$[1:1:\omega ^{n}]\rightarrow \lbrack 1:1:\omega ^{n}+a^{2}+a]$ since $T_{%
\mathbb{F}_{2^{5}}/\mathbb{F}_{2}}(\omega ^{n}+a^{2}+a)=T_{\mathbb{F}%
_{2^{5}}/\mathbb{F}_{2}}(\omega ^{n})=1$ for any $a \in \mathbb{F}_{2^{5}}$, and $\varsigma $ acts on $E_{1}$
as $[1:1:\omega ^{n}]\rightarrow $ $[1:1:\omega ^{2n}]$. Then%
\[
B=\left\{ \left[1:1:1\right], \left[1:1:\omega^{6}\right], \left[1:1:\omega^{11}\right],\left[1:1:\omega^{20}\right]\right\} 
\]%
and $\mathcal{S}=B^{\left\langle \varsigma \right\rangle} $.

\bigskip 

Alternatively, the ingredients for a \texttt{GAP}-aided construction, via the package \texttt{Design}, are provided in Section \ref{Appendix}.

\section{Reductions for the case $G$ acting point-primitively on $\mathcal{D}$}\label{FT-PP-Section}
In this section, we assume that $G$ acts point-primitively on $\mathcal{D}$. Our aim is to prove the following reduction theorem.

\bigskip
\begin{theorem}\label{FT-PP-reduction}
Let $\mathcal{D}$ be a $2$-$(v,k,\lambda)$ design admitting $\PSL(2,q) \unlhd G\leq \PGaL(2,q)$ as a flag-transitive and point-primtitive automorphism group. Then one of the following holds:
\begin{enumerate}
    \item $\mathcal{D}$ is a $2$-$(q+1,k,\lambda)$ design, with the projective line $\PG(1,q)$ as a point set, as in Construction \ref{Build}, and $G$ acts point-$2$-transitively on $\mathcal{D}$; 
    \item $\mathcal{D}$ and $G$ are as in Table \ref{tab:alldesigns};
    \item $\left(v,b,k,r\right)=\left(2^{f-1}(2^{f}-1),2^{f-1}3^{2}(2^{f}-1)\frac{\lambda }{2},\frac{2^{f}+1}{3},3(2^{f}+1)\frac{\lambda }{2}\right)$ with $f$ odd and $\lambda$ even. Moreover, $X_{\alpha}\cong D_{2(2^f+1)}$ and $X_{B}\leq D_{2(2^f+1)}$; 
    \item$\left(v,b,k,r\right)=\left(2^{f-1}(2^{f}-1),(2^{2f}-1)\lambda,2^{f-1},(2^{f}+1)\lambda\right)$ with $\lambda \mid 2f$. Moreover, $X_{\alpha}\cong D_{2(2^f+1)}$ and $X_{B}\leq Z^{f}_{2}$, and either $\mathcal{D}$ is the Witt-Bose-Shrirkhande linear space of order $2^{f}$ with $f\geq 3$, or $\lambda>1$ and $f \geq 8$;
    \item $\left(v,b,k,r\right)=\left(\frac{q(q-1)}{2},\frac{q(q+1)\lambda}{4},q-1,\frac{(q+1)\lambda}{2}\right)$ with $q>50$ and $\lambda \mid 4f$. Moreover, $X_{\alpha}\cong D_{\frac{2(q+1)}{(2,q-1)}}$ and $X_{B}\leq D_{\frac{2(q-1)}{(2,q-1)}}$.
\end{enumerate}
\end{theorem}

Before proceeding, we provide some useful facts in both design theory and group theory.
If a group $G$ acts transitively on a set $\mathcal{P}$ and $\alpha\in \mathcal{P}$, then the
\emph{subdegree} $d$ of $G$ is the length of some $G_\alpha$-orbit.  We say that $d$ is non-trivial if the orbit is not $\{\alpha\}$. 

\bigskip

\begin{lemma} \label{lem:basic-params}
Let $\mathcal{D}$ be a $2$-$(v,k,\lambda)$ design admitting a flag-transitive
automorphism group $G$. Let $\alpha \in \mathcal{P}$, then
\begin{enumerate}
\item[\rm(i)]   $r(k-1)=\lambda(v-1)$, $vr=bk$ and $b \geq v$;
\item[\rm(ii)]  $r\mid \lambda \gcd\left(\left\vert G_\alpha\right\vert, v-1\right)$, and $r^{2}>\lambda v$;
\item[\rm(iii)]  $r\mid \lambda d$ for every nontrivial subdegree $d$ of $G$.
\end{enumerate}
\end{lemma}

\begin{proof}
These are well-known facts
\end{proof}

\begin{lemma}\label{lem:NotNov}
 $X_\alpha$ is a maximal subgroup of $X$.
\end{lemma}

\begin{proof}
Suppose that $X_\alpha$ is not a maximal subgroup of $X$. Then$(G,G_\alpha,v,R)$, where $R=(\left\vert G_{\alpha}\right\vert,\lambda (v-1))$, is as in Table \ref{Tabb1} by \cite{gmaxi} since $G$ acts point-primitively on $\mathcal{D}$.

In Lines 1--9 of Table \ref{Tabb1}, one has $R=\theta \cdot(4,\lambda)$, with $\theta=3,4,5$, and $\theta ^{2}\cdot(4,\lambda)^{2} >\lambda v$ forces $(4,\lambda)=2$ or $4$. Moreover, $r=\frac{\theta \cdot(4,\lambda)}{x}$ with $x<(4,\lambda)$. Therefore, $r(k-1)=(v-1)\lambda$ and $bk=vr$ lead only to cases as in Lines 2, or 3,5 or 7 with $v=21$ or $v=36$, and all of them with $\lambda=2$. However, none of these cases occurs by \cite[Theorem 1.1]{MZZZ}. 
Finally, assume that Line 10 of Table \ref{Tabb1} occurs. Then

\[
R=\left( 24,\lambda \left( \frac{p(p^{2}-1)}{24}-1\right) \right) 
\]%
Then $576>\lambda \frac{p(p^{2}-1)}{24}$ and since  $p\equiv \pm 11,19\pmod{40}$, it follows that $p\geq 11$. Therefore $576>\lambda 55$ and so
either $p=11$ $\lambda \leq 10$ or $(p,\lambda )=(19,2)$. The latter is
ruled out $R^{2}=64<\lambda v$, whereas in the former one has $R=8(3,\lambda
)$. If $(3,\lambda )=1$, then  $R^{2}=64<\lambda 55$ for $\lambda >1$. Then $%
\lambda =3,6$ or $9$ as $\lambda \leq 10$, and hence $r=R=24$. Now, $r=24=%
\frac{(55-1)\lambda }{k-1}$ implies $(k-1)4=9\lambda $ and so $4$ divides $%
\lambda $, a contradiction. Thus $X_\alpha$ is maximal in $X$.    
\end{proof}

\begin{table}[h!]
\footnotesize
\caption{The possibilities of  $(G,G_{\alpha},v,R)$.}
\label{Tabb1}
\tiny
\begin{tabular}{lllccc}
\toprule
Line & $G$ & $G_\alpha$ & $v$ & $R$ & $(v,k,\lambda,r,b)$ \\
\midrule
1 & $\mathrm{PGL}(2,7)$         & $\mathrm{D_{12}}$                & 28 &  $3\cdot(4,\lambda)$ & -\\  
2 & $\mathrm{PGL}(2,7)$         & $\mathrm{D_{16}}$                & 21 &  $4\cdot(4,\lambda)$ & $(21, 6, 2, 8, 28)$\\  
3 & $\mathrm{PGL}(2,3^{2})$         & $\mathrm{D_{20}}$                & 36 &  $5\cdot(4,\lambda)$ & $(36, 8, 2, 10, 45)$\\  
4 & $\mathrm{PGL}(2,3^{2})$         & $\mathrm{D_{16}}$                & 45 &  $4\cdot(4,\lambda)$ & -\\  
5 & $\mathrm{M_{10}}$           & $C_{5}\rtimes C_4$               & 36 &  $5\cdot(4,\lambda)$ & $(36, 8, 2, 10, 45)$\\  
6 & $\mathrm{M_{10}}$           & $C_{8}\rtimes C_{2}$             & 45 &  $4\cdot(4,\lambda)$ & -\\  
7 & $\mathrm{\PGaL}(2,3^{2})$   & $C_{10}\rtimes C_{4}$            & 36 &  $5\cdot(4,\lambda)$ & $(36, 8, 2, 10, 45)$\\   
8 & $\mathrm{\PGaL}(2,3^{2})$   & $C_{8}\cdot \Aut(C_8)$            & 45 &  $4\cdot(4,\lambda)$ & -\\  
9 & $\mathrm{PGL}(2,11)$        & $\mathrm{D_{20}}$                & 66 &  $5\cdot(4,\lambda)$ & -\\ 
10& $\mathrm{PGL}(2,p),p\equiv \pm 11,19(\text{mod}~ 40)$ & $\mathrm{S_{4}}$  &$\frac{p(p^2-1)}{24}$  & $\left( 24,\lambda \left( \frac{p(p^{2}-1)}{24}-1\right) \right)$ &  \\
\bottomrule
\end{tabular}
\end{table}

\bigskip

It follows from Lemma \ref{lem:NotNov} that $X_\alpha$ is the maximal subgroup of $X$. Hence, by \cite{Di, Hup}, $X_\alpha$ is one of the following groups:

\begin{table}[H]
\tiny
    \centering
    \caption{Maximal subgroups of $\PSL(2,q)$}\label{MaxDickson}
    \begin{tabular}{cccl}
        \toprule
        Line & Maximal subgroup & & cond. \\
        \midrule
        1 & $\mathrm{PGL}(2,q_0)$ & &$q=q_0^2$ odd \\
        2 & $\mathrm{PSL}(2,q_0)$ & &$q=q_0^t$ odd, $t$ odd prime \\
        3 & $\mathrm{PGL}(2,q_0)$ & &$q=2^f=q_0^t$, $t$ prime and $q_0\neq 2$ \\
        4 & $\mathrm{A_4}$ & &$q=p\equiv \pm 3 \text{(mod 8)}$ and $q\not\equiv \pm 1\text{(mod 10)}$ \\
        5 & $\mathrm{S_4}$ & &$q=p\equiv \pm 1\text{(mod 8)}$ \\
        6 & $\mathrm{A_5}$ & &$q\equiv \pm 1\text{(mod 10)}$, and either \\
        & & &$q=p$ or $q=p^2$ and $p\equiv \pm 3\text{(mod 10)}$ \\
        7 & $\mathrm{D}_{2(q-1)/\gcd(2,q-1)}$ & & - \\
        8 & $\mathrm{D}_{2(q+1)/\gcd(2,q-1)}$ & &- \\
        9 & $Z_p^f:Z_{(q-1)/\gcd(2,q-1)}$ & & -\\ 
        \bottomrule
    \end{tabular}
\end{table}

In what follows, we will analyze each of these cases separately.

\begin{table}
\tiny
\centering
\renewcommand{\arraystretch}{1.4} 
\begin{threeparttable}
\caption{Subdegrees of $\PSL(2,q)$ on the set of cosets of some of its subgroups}
\label{tab:subdegrees}
\begin{tabular}{
  >{\raggedright\arraybackslash}p{0.20\textwidth}
  >{\raggedright\arraybackslash}p{0.59\textwidth}
  >{\raggedright\arraybackslash}p{0.24\textwidth}
}
\toprule
$H$ & Subdegrees & Cond. \\ 
\midrule

$\PGL(2,q_0)$
& $1,\ \frac{q_0(q_0-1)}{2},\ q_0^2-1,\ (q_0(q_0-1))^{\frac{q_{0}-5}{4}},$ \newline
  \vspace*{1.5ex}
  $(q_0(q_0+1))^{\frac{q_0-1}{4}}$
& $q=q_0^2$, \newline $q_0\equiv 1 \pmod{4}$ \\  
\addlinespace[0.8ex] 

$\PGL(2,q_0)$
& $1,\ \frac{q_0(q_0+1)}{2},\ q_0^2-1,\ (q_0(q_0-1))^{\frac{q_{0}-3}{4}},$ \newline
  \vspace*{1.5ex}
  $(q_0(q_0+1))^{\frac{q_0-3}{4}}$
& $q=q_0^2$, \newline $q_0\equiv -1\pmod{4}$ \\ 
\midrule

$\PSL(2,q_0)$
& $1,\ \left(\frac{q_0^2-1}{2}\right)^{2a_1},\ (q_0(q_0-1))^{a_2},\ (q_0(q_0+1))^{a_3},$ \newline
  \vspace*{1.5ex}
  $\left(\frac{q_0(q_0^2-1)}{2}\right)^{2a_4}$
& $q=q_0^t$ odd, \newline $t$ odd and prime \\ 
\midrule

$\PGL(2,q_0)$
& $1,\  q_0^2-1,\  (q_0(q_0-1))^{\frac{q_0-2}{2}},\ (q_0(q_0+1))^{\frac{q_0}{2}}$ 
& $q=2^f=q_0^2$ \\ 
\addlinespace[0.8ex] 
$\PGL(2,q_0)$
& $1,\ (q_0^2-1)^{a_1},\ (q_0(q_0-1))^{a_2},\ (q_0(q_0+1))^{a_3},$ \newline
  \vspace*{0.8ex}
  $ (q_0(q_0^2-1))^{a_4}$
& $q=2^f=q_0^t$, \newline $t$ odd and prime \\ 
\midrule

$\mbox{$C_p^f : C_{\frac{q-1}{\gcd(2,q-1)}}$}$
& $1,\ q$
& always \\
\midrule

$D_{2(q+1)}$ & $1,\ (q+1)^{\frac{q-2}{2}}$ & $q$ even \\ 
\addlinespace[0.8ex] 
$D_{q+1}$    & $1,\ (\frac{q+1}{2})^{\frac{q-3}{2}},\ (q+1)^{\frac{q-1}{4}}$ & $q \equiv 1 \pmod{4}$ \\ 
\addlinespace[0.8ex] 
$D_{q+1}$    & $1,\ (\frac{q+1}{4})^{2},\ (\frac{q+1}{2})^{\frac{q-3}{2}},\ (q+1)^{\frac{q-3}{4}}$ & $q \equiv -1 \pmod{4}$ \\ 
\midrule

$D_{2(q-1)}$ & $1,\ (q-1)^{\frac{q-2}{2}},\ 2(q-1)$ & $q$ even \\ 
\addlinespace[0.8ex] 
$D_{q-1}$    & $1,\ (\frac{q-1}{4})^{2},\ (\frac{q-1}{2})^{\frac{q-5}{2}},\ (q-1)^{\frac{q+7}{4}} $ & $q \equiv 1 \pmod{4}$ \\ 
\addlinespace[0.8ex] 
$D_{q-1}$    & $1,\ (\frac{q-1}{2})^{\frac{q-1}{2}},\ (q-1)^{\frac{q+5}{4}}$ & $q \equiv -1 \pmod{4}$\\
\bottomrule
\end{tabular}

\begin{tablenotes}
  \small
  \item \textit{Notes.} The multiplicities are defined as follows:
  \\[1.2ex] 
  $$\displaystyle a_1 = \frac{q_0^{t-1}-1}{q_0-1}, \quad a_2 = \frac{q_0^t-q_0}{2(q_0+1)}, \quad a_3 = \frac{q_0^t-q_0}{2(q_0-1)}, \quad \text{and} $$
  $$\displaystyle a_4 = \frac{q_0^{3t-2} - q_0^{t-2}(q_0^4+q_0^3+2q_0^2-q_0) + (q_0^3+q_0^2+q_0-1)}{(q_0^2-1)^2}.$$
\end{tablenotes}
\end{threeparttable}
\end{table}

\begin{lemma}\label{case1}

If $X_{\alpha}= \mathrm{PGL}(2,q_0)$, for $q=q_0^2$ odd, then $q_0=3$ and $\mathcal{D}$ is as in Line 18 of Table \ref{tab:alldesigns}, namely $\mathcal{D}$ is the symmetric $2$-$(15,8,4)$ complementary design of $\PG(3,2)$ and $\PSL(2,3^{2})\unlhd G \leq \PSiL(2,3^{2})$. 
\end{lemma}

\begin{proof}
Let $X_{\alpha}=\mathrm{PGL}(2,q_0)$, for $q=q_0^2$ odd, then $v=\frac{q_0(q_0^2+1)}{2}$, and hence $v-1=\frac{q_0^3+q_0-2}{2}$.
Moreover, we have that $r\mid |X_{\alpha}||\Out(X)|=2fq_0(q_0^2-1)$.
\\
Suppose $q_0\equiv -1 \pmod{4}$. It follows from Table \ref{tab:subdegrees} that $X$ has unique subdegree equal to
$\frac{q_0(q_0+1)}{2}$ and $q_0^2-1$.
Then by Lemma \ref{lem:basic-params}\,(iv) and $r\mid \lambda (v-1)$, we obtain 
\[
r\mid \lambda\left(\frac{q_0(q_0+1)}{2}, \ q_0^2-1, \ \frac{q_0^3+q_0-2}{2} \right).
\]
Since $\left(\frac{q_0(q_0+1)}{2}, \ q_0^2-1\right)=\frac{q_0+1}{2}$ and $\left(\frac{q_0+1}{2}, \ \frac{q_0^3+q_0-2}{2}\right)=2$, we get $r|2\lambda$. Actually, we get $r=2\lambda$ since $r>\lambda$.
Since $r(k-1)=\lambda(v-1)$, we have $k=\frac{v+1}{2}=\frac{q_0(q_0^2+1)+2}{4}$, implying that $k$ and $v$ are co-prime,
and so $(k,q_0)=1$ and $k\mid r$.
It follows from $r\mid 2fq_0(q_0^2-1)$ that $k\mid 2f(q_0^2-1)$, implying
\[
q_0(q_0^2+1)+2=(q_0+1)(q_0^2-q_0+2)\mid 8f(q_0^2-1),
\]
and so $(q_0^2-q_0+2)\mid 8f(q_0-1)$.
As $(q_0^2-q_0+2,q_0-1)=2$, we obtain $(q_0^2-q_0+2) \mid 16f$, forcing $f<4$.
Hence, $f=2$, $q_0=p$ and $p^2-p+2\mid 32$, 
implying $p=3$, as $q_0\equiv-1\bmod 4$. 
Thus, $q=9$, $v=15$, $k=8$, and hence $G$ has primitive permutation representation of degree $15$, forcing $\PSL(2,3^{2})\unlhd G \leq \PSigL(2,3^{2})$ by \cite{At}. Moreover, $k\le r$, $k\mid r$ and $r\mid 2fq_0(q_0^2-1)$,
we obtain $r=8, 16, 24, 32, 48$ or $96$. Actually, by \texttt{GAP}, only $r=k=8$ yields to an example: the $2$-$(15,8,4)$ complementary design of $\PG(3,2)$ and $\PSL(2,3^{2})\unlhd G \leq \PSigL(2,3^{2})$. 

Suppose $q_0\equiv 1 \pmod{4}$. 
Then, by Lemma \ref{lem:basic-params}\,(iv) and  Table \ref{tab:subdegrees}, we obtain that
\begin{equation}\label{eq:r.div.gcd.1}
    r\mid \lambda \left(\frac{q_0(q_0-1)}{2}, \ q_0^2-1, \ tq_0(q_0+1) \right),
\end{equation}
where $1\le t \le \tfrac{1}{4}(q_0-1)$ and $t\mid |\Out(X)|=2f$. 
Since $\left(\frac{q_0(q_0-1)}{2}, \ q_0^2-1\right)=\frac{q_0-1}{2}$ 
and $\left(\frac{q_0-1}{2}, tq_0(q_0+1)\right)$ divides $t(q_0-1,q_0(q_0+1))=2t$, 
we deduce $r|2t\lambda$,  and so $r\mid4f\lambda$.
Let $r=4f\lambda/x$, for some positive integer $x$, where $x<4f$ as $r>\lambda$.
If $x=2f$, then $r=2\lambda$.
A similarly argument to that of case $q_0\equiv 1 \pmod{4}$ leads to $(q_0^2-q_0+2)\mid 8f(q_0-1)$,
and so $f=2$, $q_0=p$ and $p^2-p+2\le 32$, 
implying $q_0=5$, as $q_0\equiv1\pmod{4}$. This however is ruled out since it does not satisfy $(q_0^2-q_0+2)\mid 8f(q_0-1)$.
Thus, $x<4f$ and $x\neq 2f$.
From $r(k-1)=\lambda(v-1)$, we have 
$k=\frac{x}{4f}(v-1)+1$.
It follows from $k\mid vr$ that $k\mid (k,v)(k,r)$.
On one hand, we have 
\[
(k,v) \mid (4fk,v)=(x(v-1)+4f,v)=(4f-x,v)\le 4f-x<4f.
\]
On the other hand, from $r\mid 2fq_0(q_0^2-1)$ we obtain 
$(k,r) \mid (4fk,2fq_0(q_0^2-1))$ and so 
$(k,r)$ divides $2f\left(x(v-1)+4f,q_0(q_0^2-1)\right)$.
Note that 
\begin{align*}
(x(v-1)+4f,q_0)&=\left(x\frac{q_0(q_0^2+1)}{2}-x+4f,q_0\right)=\gcd(4f-x,q_0)<4f,\\
(x(v-1)+4f, q_0-1)&=\left(x(q_0-1)\left(\frac{q_0^2+q_0+2}{2}\right)+4f,q_0-1\right)
        =(4f,q_0-1)\le 4f,\\
(x(v-1)+4f, q_0+1)&=\left(x(q_0+1)\left(\frac{q_0^2-q_0+2}{2}\right)-2x+4f,q_0+1\right)\\
         &=(4f-2x,q_0+1)\le |4f-2x|<4f,           
\end{align*}
so that $(k,r)<2f\cdot (4f)^3$.
Therefore, $k<2f\cdot (4f)^4=512f^5$.
Combining this with $k=\frac{x}{4f}(v-1)+1>\frac{1}{4f}(v-1)$, 
we have $q_0(q_0^2+1)\le 4096f^6$, and so we get the admissible values as in Table \ref{tab:C2}.

\begin{table}[H]
\centering 
\caption{admissible $f$ and $q_0$} 
\label{tab:C2} 
\begin{tabular}{ccr}
\toprule
Line & $f$ & $q_0$ \\ 
\midrule 
1&  $2$&  $5,13,17,29,37,41,53,61$ \\
2&  $4$ & $3^{2},5^{2},7^{2},11^{2},13^{2}$ \\
3 & $6$ & $5^{3}$ \\
4 & $8$ & $3^{4},5^{4}$ \\
5 & $12$ & $3^{6}$\\
\bottomrule 
\end{tabular}
\end{table}
However, only $(f,q_0)=(2,5)$ with $x=3$ or $(4,3^{2})$ with $x=1,7$ are the unique cases fulfilling $k \mid \left\vert X \right\vert \cdot 2f$. They lead to $(k,v)=(25,65)$ and $\PSL(2,5^{2})\unlhd G\leq \PSigL(2,5^{2})$ or $(k,v)=(24,369),(162,369)$ and $\PSL(2,3^{4})\unlhd G\leq \PSigL(2,3^{4})$, respectively, by \cite[§260]{Di} and \cite[Theorem 1.3(5)]{gmaxi}. However, by using \texttt{GAP}, none of the previous cases leads to an example. This completes the proof.
\end{proof}

\begin{lemma}\label{case2}
$X_{\alpha}\neq \mathrm{PSL}(2,q_0)$, for $q=q_0^t$ odd, where $t$ is an odd prime.
\end{lemma}

\begin{proof}
Assume that $X_\alpha=\PSL(2,q_0)$, $q=q_0^t$ odd prime power with $t$ odd prime.
Then $|X_\alpha|=q_0(q_0^2-1)/2$ and $|\Out(X)|=2f$, 
implying that $v=q_0^{t-1}(q_0^{2t}-1)/(q_0^2-1)$ and $r$ divides $|X_\alpha||\Out(X)|=fq_0(q_0^2-1).$
Let $q_0=p^s$, then $f=st\ge 3$, $s\in\mathbb N$.
From Lemma \ref{lem:basic-params} (iii), we obtain 
\[
f^2q_0^2(q_0^2-1)^2\ge r^2>\lambda v \ge\frac{q_0^{t-1}(q_0^{2t}-1)}{q_0^2-1}.
\]
Note that $q>3$ and $f\ge 3$, then we have $q=q_0^t=p^f>f^2$.
Thus, we deduce 
\[
q_0^{t+8}>f^2q_0^2(q_0^2-1)^3> q_0^{t-1}(q_0^{2t}-1)\text{,}     
\]
and so $t=3$, as $t$ is odd.
Hence $f=3s$ and $v=q_0^2(q_0^4+q_0^2+1).$ 
By Lemma \ref{lem:basic-params}\,(ii) and  Table \ref{tab:subdegrees}, we obtain
\[
    r\mid \lambda \left(2fq_0\frac{q_0^2 -1}{2}, \ v-1 \right)\text{.}
\]
 
Since $\left(q_0\frac{q_0^2-1}{2},v-1\right)=2$,
we have $r\mid 4f\lambda=12s\lambda$, as $f=3s$.
Let $r=12s\lambda/x$,
for some positive integer $x$ with $x<12s$, as $r>\lambda$.
From $r(k-1)=\lambda (v-1)$,  we have $k=\frac{x}{12s}(v-1)+1$.
It follows from $k\mid vr$ that $k\mid (k,v)(k,r)$.
On one hand, we have 
\[
(k,v)\mid (12sk,v)=(x(v-1)+12s,v)=(12s-x,v)\le 12s-x<12s.
\]
On the other hand, from $r\mid fq_0(q_0^2-1)$ we obtain 
$(k,r)\mid (12sk,3sq_0(q_0^2-1))$, and so $(k,r)\mid (x(v-1)+12s,3sq_0(q_0^2-1))$.
Note that 
\begin{align*}
(x(v-1)+12s,q_0)     &=\left(xq_0^2(q_0^4+q_0^2+1)-x+12s,q_0\right)=(12s-x,q_0)<12s,\\
(x(v-1)+12s, q_0^2-1)&=\left(x(q_0^2-1)(q_0^4+2q_0^2+3)+2x+12s,q_0^2-1\right)\\
                     &=(12s+2x,q_0-1)\le 12s+2x<36s,          
\end{align*}
so that $(k,r)<3s\cdot 12s\cdot 36s$.
Therefore, $k<\frac{3(12s)^4}{4}$.
Combining this with $k=\frac{x}{12s}(v-1)+1\ge \frac{1}{12s}(v-1)+1$,
we have $v\le 12sk-12s+1< \frac{3(12s)^5}{4}-12s+1$.
As $v=q_0^2(q_0^4+q_0^2+1) = p^{6s} + p^{4s} + p^{2s}$,
we deduce
\[
p^{6s}+p^{4s}+p^{2s}<\frac{3(12s)^5}{4}-12s+1,
\]
implying $(p,s)=(3,1),(5,1),(7,1)$, or $(3,2)$. Recalling that $r\mid fq_0(q_0^2-1)$, $f=3s$, $q_0=p^s$ and $r^2>v$, 
we list all the possible $r$ in Table \ref{tab:lemma3.3}:
\begin{table}[H]
\tiny
\centering 
\caption{Admissible $q_0$, $v$ and $r$} 
\label{tab:lemma3.3} 
\begin{tabular}{cccc}
\toprule
Line & $q_0$ & $v$ & $r$ \\ 
\midrule 
1 &   $3$&    $819$ &      $36$     \\
2  &     &          &      $72$     \\
3 &   $5$&    $16275$ &      $180$     \\
4 &      &            &     $360$     \\
5 &   $7$&    $120099$ &      $504$    \\
6 &     &            &    $1008$   \\
7 &   $9$&    $538083$ &      $864$    \\
8 &      &             &      $1080$   \\
9&       &             &    $1440$   \\
10&      &           &    $2160 $   \\
11 &     &           &    $4320$  \\
\bottomrule
\end{tabular}
\end{table}
Since $r(k-1)=\lambda(v-1)$,
we have $(v-1)\mid r(k-1)$,
and so $(v-1)/(v-1,r)$ divides $k-1$,
implying $(v-1)/(v-1,r)\le k-1\le r-1$.
However, none of the cases in Table \ref{tab:lemma3.3} satisfies the condition.
Hence, $X_{\alpha}\neq \mathrm{PSL}(2,q_0)$.
\end{proof}


\begin{lemma}\label{case3}
$X_{\alpha}\neq \mathrm{PGL}(2,q_0)$, for $q=2^f=q_0^t$, where $t$ is prime and $q_0\neq 2$.
\end{lemma}

\begin{proof}
Assume that $X_\alpha=\PGL(2,q_0)$, where $q=2^f=q_0^t$ with $t$ prime and $q_0\neq 2$.
Then $v=q_0^{t-1}(q_0^{2t}-1)/(q_0^2-1)$ and $|\Out(X)|=f$,
implying $r\mid |X_\alpha||\Out(X)|=fq_0(q_0^2-1)$.
Let $q_0=2^s$, for some integer $s\ge 2$, so $f=st\ge 4$,
implying $f^2\le 2^f=q_0^t$.
Combining this with $r^2 >\lambda v$
we obtain 
\[
q_0^{t+8}>f^2q_0^2(q_0^2-1)^3>\lambda q_0^{t-1}(q_0^{2t}-1)\ge q_0^{t-1}(q_0^{2t}-1)\text{,}    
\]
and so $t=2$ or 3, as $t$ prime.

First assume that $t=2$. 
Then $q=q_0^2$, $f=2s$ and $v=q_0(q_0^2+1)$. 
It follows from Table \ref{tab:subdegrees} that $X$ has a unique subdegree $q_0^2-1$.
Then by Lemma \ref{lem:basic-params} (iv) we have 
\[
r\mid \lambda (q_0^2-1,q_0^3+q_0-1).
\]
Note that $(q_0^2-1,q_0^3+q_0-1)=(q_0^2-1,2q_0-1)$,
$(q_0^2-1,2q_0-1)=(q_0+1,3)$.
If $(q_0+1,3)=1$, then $r\mid \lambda$, contradicting $r>\lambda$.
Hence, $(q_0+1,3)=3$, implying $s$ odd, and so $s\ge3$.
Moreover, we deduce $r\mid 3\lambda$.
Let $r=3\lambda/x$. Then $x=1$ or $2$.
If $r=3\lambda$, then $k=\frac{v+2}{3}$, and so $(k,v)=(k,3k-2)=2$.
It follows from $k\mid vr$ that $k\mid 2r$.
Combining this with $r\mid 2s q_0(q_0^2-1)$, we obtain $v+2\mid 12sq_0(q_0^2-1)$,
and so $(q_0^2-q_0+2)\mid 12sq_0(q_0-1)$, implying $(q_0^2-q_0+2)\mid 24s$.
Therefore, $2^{2s}-2^s+2\le 24s$,  implying $s=3$.
However, if $s=3$ then $q_0=8$ and $q_0^2-q_0+2=58\nmid72=24s$, a contradiction.
Thus, $r=3\lambda/2$ and $k=\frac{2v+1}{3}$, so $(k,v)=1$, implying $k\mid r$.
So $2v+1\mid 6sq_0(q_0^2-1)$.
Since $(2v+1,q_0)=1$, $(2v+1,q_0-1)=(q_0-1,5)=(2^s-1,5)=1$ 
and $(2v+1,q_0+1)=(q_0+1,3)=3$, as $s$ is odd,
we have $2v+1\mid 18s$.
Hence, $2^{3s+1}+2^{s+1}+1 \le 18s$, a contradiction. 

Thus, $t=3$, $f=3s$, $v=q_0^2(q_0^4+q_0^2+1)$.
Similarly, from Table \ref{tab:subdegrees}, we obtain that
\[
 r\mid f\lambda \left((q_0^2-1), q_0(q_0-1), q_0(q_0+1)\right).
\]
Since $((q_0^2-1), q_0(q_0-1))=q_0-1$ and $(q_0-1,q_0(q_0+1))=(q_0-1,2)=1$,
we have $r\mid f\lambda$, and so $rf\lambda\ge r^2>\lambda v$, implying $rf>v$.
Combining this with $r\mid fq_0(q_0^2-1)$,
we have 
\[
f^2q_0^3>f^2q_0(q_0^2-1)>q_0^2(q_0^4+q_0^2+1)>q_0^6,
\]
so $f^2>q_0^3=q_0^{t}$, a contradiction.
\end{proof}





\begin{lemma}\label{case4}
Suppose that $X_{\alpha}=\A_4$, for $q=p\equiv \pm 3 \pmod 8$ and $q\not\equiv \pm 1\pmod{10}$. Then $\PSL(2,5)\leq G\leq \PGL(2,5)$ and $\mathcal{D}$ is the complete $2$-$(5,3,3)$ design as in Line 1 of Table \ref{tab:alldesigns}.
\end{lemma}
\begin{proof}
Here, $v=q(q^2-1)/24$ and $|\Out(X)|=2$, and so $r\mid 24$.
By $r^2>v$ and $q>3$, we obtain $24^3>q(q^2-1)$, and hence$q=5$, $13$, and $v=5$, $91$, respectively.

If $(q,v)=(13,91)$, then from $r^2>v$ and $r\mid 24$, 
we get $r=12,24$, and by $r(k-1)=\lambda(v-1)$ we get $k=15\lambda/2+1, 15\lambda/4+1$, respectively. However, both cases are ruled out since $k \mid vr$ and $k \leq r$.

If $(q,v)=(5,5)$, then $k=3$ as $2<k<v-1$.
From $k\mid vr$, $r^2>v$ and $r\mid 24$, 
we obtain $r\in \{3,4,6,8,12,24\}$. If $r=3$ then $\lambda$ is not an integer, while if $r=4,8$ then $k\nmid vr$. Hence $r\in \{6,12,24\}$, and accordingly $\lambda\in \{3,6,12\}$.
If $r=6$ then $b=vr/k=10$, hence we have the complete $2$-$(5,3,3)$ design with $\PSL(2,5)\leq G\leq \PGL(2,5)$.
If $r=12$ or 24, then $b=20$ or 40.
Recall that $G$ is flag-transitive, so there exists a subgroup $G_B$ of index $b$, which has an orbit $\mathcal{O}$ of size $k=3$. Moreover, from the block-transitivity of $G$, we have $\left\vert\mathcal{O}^{G}\right\vert=b$.
However, by \texttt{GAP}, in both of these two cases there is no such $\mathcal{O}$ with $|\mathcal{O}^{G}|=b$.
\end{proof}

\begin{lemma}\label{case5}
If $X_\alpha = \S_4$ for $q=p\equiv \pm 1\text{(mod 8)}$, then 
$G = \PSL(2,7)$, and one of the following holds:
\begin{enumerate}
    \item $\mathcal D$ is either $\PG(2,2)$ or its complementary design as in Lines 7 and 8 of Table \ref{tab:alldesigns}, respectively;
    \item $\mathcal D$ is the $2$-$(7,3,4)$ design as in Line 9 of Table \ref{tab:alldesigns}.
 \end{enumerate}
\end{lemma}

\begin{proof}
Suppose that $X_{\alpha}=\S_4$, for $q=p\equiv \pm 1 \pmod{8}$.
Then $v=q(q^2-1)/48$ and $\left\vert\Out(X)\right\vert=2$, and so $r\mid 48$.
From $r^2> v$, we obtain $48^3>q(q^2-1)$, 
implying $q=7$, $17$, $23$, $31$, $41$ and so $v=7$, $102$, $253$, $620$, $1435$ respectively. 

If $v=102$ or $620$, then $v-1$ is a prime.
From $r\mid \lambda(v-1)$ and $r\mid 48$,
we deduce $r\mid (\lambda, 48)$,  and so $r\le \lambda$, a contradiction.
If $(q,v)=(23,253)$, then $G=\PSL(2,23)$ by \cite{At}. Moreover, $r^2>v$ and $r\mid 48$ imply $r=24,48$. Actually, $r=24$ since $vr\mid \left\vert G\right\vert $, and hence $(b,k,\lambda)=(276,22,2)$. However, no examples occur by \texttt{GAP}. Similarly, we get no examples for $(q,v)=(41,1435)$.

Finally, consider $(q,v)=(7,7)$. Then $G=\PSL(2,7)$. Moreover, from $vr\mid \left \vert G \right\vert $ and $r^2>v$, we have $r=3,4,6,8,12$ or $24$ and hence the admissible $2$-design parameters $(v,b,r,k,\lambda)$ 
listed in Table \ref{tab:S4}:
\begin{table}[H]
\tiny
\centering 
\caption{$2$-design parameters for $v=7$ and $G=\PSL(2,7)$} 
\label{tab:S4} 
\begin{tabular}{cccccc}
\toprule
Line &  $v$ & $b$ & $r$ & $k$ & $\lambda$ \\ 
\midrule
1 &  $7$ & $7$ & $3$ & $3$ & $1$\\
2 &    &      & $4$ & $4$ & $2$\\
3 &    & $14$ & $6$ & $3$ & $2$\\
4 &   &       & $8$ & $4$ & $4$  \\
5 &  & $28$ & $12$ & $3$ & $4$   \\
6 &  & $21$ & $12$ & $4$ & $6$  \\  
7 &  & $56$ & $24$ & $3$ & $8$  \\ 
8 &  & $42$ & $24$ & $4$ & $12$  \\
\bottomrule
\end{tabular}
\end{table}
The $2$-designs arising from Lines 1 and 2 are $\PG(2,2)$ and its complementary design respectively, that in Line 3 is the $2$-$(7,3,4)$ design as in Line 8 of Table \ref{tab:alldesigns}. Easy computations show that the remaining parameters in Lines 4--8 lead to no examples.
\end{proof}

\begin{lemma}\label{case6}
Suppose that $X_{\alpha}=\mathrm{A_5}$, for $q\equiv \pm 1\text{(mod 10)}$, where either $q=p$ or $q=p^2$ and $p\equiv \pm 3\text{(mod 10)}$. Then one of the following holds:
\begin{enumerate}
    \item  $\PSL(2,3^{2})\unlhd G \leq \PSiL(2,3^{2})$ and $\mathcal{D}$ is a $2$-design as in Lines 4, 6 of Table \ref{tab:alldesigns};
    \item $G=\PSL(2,11)$ and $\mathcal{D}$ is a $2$-design as in Lines 11--17 of Table \ref{tab:alldesigns}.
\end{enumerate}
\end{lemma}
\begin{proof} 
Let $X_{\alpha}=\mathrm{A_5}$ and $q$ be as in the statement.
Then $v=q(q^2-1)/120$ and $\left\vert \Out(X)\right\vert =2f$, so $r\mid 120f$, with $f=1$ or $2$.
It follows from $r^2> v$ that 
$ 120^3\cdot f^2>q(q^2-1)$.
If $f=1$, then $q=p\le 120$.
Since $v=q(q^2-1)/120$ and $q \equiv \pm 1 \pmod{10}$, 
we obtain $(q, v)=(11, 11), (19, 57), (29, 203), (31, 248), (41, 574)$, 
$(59, 1711), (61, 1891), (71, 2982)$ $(79, 4108)$, $(89,5874)$, $(101,8585)$ or $(109,10791)$. 
If $f=2$, then $q=p^2\le 132$, and similarly we have 
$(q,v)=(9, 6)$, $(49, 980)$ or $(169,40222)$.


If $(q,v)=(31,248),(71,2982),(49,980)$, or $(89,5874)$, then from $r|120f$ and $r|\lambda(v-1)$, we get that $r\mid(\lambda,120f)$ since $(120f,v-1)=1$. Then $r\le \lambda$, a contradiction. 

If $(q,v)=(29,203)$, then from $r\mid 120$ and $r\mid \lambda(v-1)$, we get that $r\mid 2(\lambda,120)$ since $(120,v-1)=2$. Then $r=2 \lambda$ and $\lambda\mid60$. From $\lambda(v-1)=r(k-1)$ we get $k=(v-1)/2+1=102$. Since $r\ge k$, we get $\lambda\ge 51$, hence $\lambda=60$ and $r=120$. However, $k=102$ does not divide $vr=203\cdot120$, a contradiction. 

If $(q,v)=(41,574)$, then from $r\mid 120$ and $r\mid \lambda(v-1)$, we get that $r\mid 3(\lambda,120)$ since $(120,v-1)=3$. Hence, $r\mid 3\lambda$, that is $r=3\lambda/x$ with $1\le x\le 2$. From $\lambda(v-1)=r(k-1)$ we get $k=x(v-1)/3+1\ge192$. However, $r \le 120 < k$, a contradiction.

If $(q,v)=(79, 4108)$, then similarly we get $k=x(v-1)/3+1$, with $1\le x \le 2$, hence $k> r$, a contradiction.

If $(q,v)=(101, 8585)$ then, since $(120,v-1)=8$, similarly we get $k=x(v-1)/8+1=1073x+1$, with $1\le x \le 7$, hence $k>r$, a contradiction.

If $(q,v)=(109, 10791)$ then, since $(120,v-1)=10$, similarly we get $k=x(v-1)/10+1=1079x+1$, with $1\le x \le 7$, hence $k> r$, a contradiction.

If $(q,v)=(169, 40222)$ then, since $(240,v-1)=3$, similarly we get $k=x(v-1)/3+1=13407x+1$, with $1\le x \le 2$, hence $k> r$, a contradiction.

If $(q,v)=(59, 1711)$, then similarly we get $r=30\lambda/x$ and $k=x(v-1)/30+1=57x+1$, with $1\le x< 30$. Since $k \le r$ and $r\le 120$, from $k\le 120$ we have $x=1,2$. If $x=2$, then $r=15\lambda$, $\lambda\mid8$ and $k=115$. From $k\mid vr$ we get that $115\mid1711\cdot 15$ as $k$ is odd, a contradiction. Hence, $x=1$, $r=30\lambda$, $\lambda\mid4$ and $k=58$. From $k\mid vr$ we get $k\mid 29\cdot 30\cdot\lambda$. Since $r^2>\lambda v$ we get $30^2\lambda> 1711$, that is $\lambda \ge 2$. Hence, $(r,\lambda)=(60,2),(120,4)$. Moreover, $G=\PSL(2,59)$ by \cite[§259]{Di} since $G$ must have a subgroup  $G_\alpha$ of index $v=1711$, hence $G_\alpha=A_{5}$. Then $(r,\lambda)=(60,2)$ since $r \mid G_{\alpha}$. However, no examples occur by \texttt{GAP}.

If $(q,v)=(61, 1891)$, then similarly we get $r=30\lambda/x$ and $k=x(v-1)/30+1$, with $1\le x< 30$. Since $k \le r$ and $r\le 120$, form $k\le 120$ we have $x=1$. Then, $r=30\lambda$, $\lambda\mid 4$ and $k=64$. From $k\mid vr$ we get $k\mid r$ as $v$ is odd, hence $64\mid 30\cdot 4$, a contradiction.

If $(q,v)=(19,57)$, then $G=\PSL(2,19)$ since $\PGL(2,19)$ does not have representation of degree $57$ by \cite[§259]{Di}. From $r\mid 120$ and $r\mid \lambda(v-1)$ we get that $r\mid 8(\lambda,120)$ since $(120,v-1)=8$. Hence, $r\mid 8\lambda$, that is $r=8\lambda/x$ with $1\le x \le 7$. Since $r\mid120$ we have that $\lambda\mid15x$. From $\lambda(v-1)=r(k-1)$ we get $k=7x+1$. With these conditions, we obtain $x=2$ and $\lambda=5,15$, hence $r=20,60$ and $b=76,228$, respectively. However, $\PSL(2,19)$ does not have subgroups of index $76$ or $228$, a contradiction.

If $(q,v)=(9, 6)$ then $G=\PSL(2,9)$ or $\PSigL(2,9)$ \cite{At}. From $r\mid240$ and $r\mid\lambda(v-1)$, we get that $r\mid 5(\lambda,240)$ since $(240,v-1)=5$. Hence, $r\mid 5\lambda$, that is $r=5\lambda/x$ with $1\le x\le 4$. From $\lambda(v-1)=r(k-1)$ we get $k=x+1$. Since $2<k<v-1=5$ we get $x=2,3$. The stabilizer $G_B$ of any block $B$ has index $30\lambda/(x(x+1))$ in $G$. By computing $G_B$ and verifying that it coincides with the stabilizer of one of its obits of size $k$, we obtain the $2$-designs with parameters as in Table \ref{tab:lem3.7_q9} since the permutation representation of $G$ of degree $6$ is $2$-transitive. 
\begin{table}[H]
\tiny
\centering
 \caption{Admissible parameters for $X_\alpha=\A_5$ and $q=9$}
    \label{tab:lem3.7_q9}
    \centering
    \begin{tabular}{cccccccccc}
    \toprule
         Line & $x$ & $v$ & $b$   & $r$   & $k$ & $\lambda$  & $G$ & $G_{\alpha}$ & $G_B$ \\
     \midrule    
         1 & $2$ & $6$ & $20$ & $10$ & $3$ & $4$ & $\PSL(2,9)$ &  $\A_5$ & $(\mathbb{Z}_3 \times \mathbb{Z}_3) \colon \mathbb{Z}_2$ \\
         2 &     &     &      &      &     &     & $\PSigL(2,9)$ & $\S_5$ & $\S_3\times\S_3$ \\
         3 & $3$ & $6$ & $15$ & $10$ & $4$ & $6$  & $\PSL(2,9)$ &  $\A_5$ & $\S_4$ \\
         4 &     &     &      &      &     &     & $\PSigL(2,9)$ & $\S_5$ & $\mathbb{Z}_2\times \S_4$\\
         \bottomrule
    \end{tabular}
\end{table}
These cases give Lines 3, 5 of Table \ref{tab:alldesigns}.

If $(q,v)=(11,11)$, then $G=\PSL(2,11)$ by \cite{At}. From $r\mid 120$ and $r\mid \lambda(v-1)$ we get that $r\mid 10(120,\lambda)$ as $(120,v-1)=10$. Hence, $r\mid 10\lambda$, that is $r=10\lambda/x$ with $1\le x \le 9$. Since $r\mid 120$, we get that $\lambda\mid 12x$. From $\lambda(v-1)=r(k-1)$ we find $k=x+1$, and since $2<k<v-1$, we get $2\le x \le 8$. With these conditions, we get the parameters in Table \ref{tab:lem3.7_q11}.
\begin{table}[H]
\tiny
\centering
    \caption{\small Admissible parameters for $G=\PSL(2,11)$ and $G_\alpha=\A_5$}
    \label{tab:lem3.7_q11}
    \centering
    \begin{tabular}{ccccccc}
    \toprule
         Line& $x$ & $v$  & $b$   & $r$  & $k$  & $\lambda$ \\
         \midrule
        1 & $2$ $ $&$11$ & $55$  & $15$ & $3$  & $3$  \\
        2& &  & $110$ & $30$ &  & $6$          \\
        3& &   & $220$ & $60$ &   & $12$         \\

        4& $3$ &  & $55$  & $20$ & $4$  & $6$          \\
        5 & & & $165$ & $60$ &   & $18$         \\

        6& $4$ &  & $11$ & $5$ & $5$  & $2$         \\
        7 & & & $22$  & $10$ &   & $4$          \\
        8 & & & $33$  & $15$ &   & $6$          \\
        9 & & & $44$  & $20$ &   & $8$          \\
        10 & & & $66$  & $30$ &   & $12$         \\
        11 & & & $132$ & $60$ &   & $24$         \\

        12& $5$ &  & $11$  & $6$  & $6$  & $3$          \\
        13& &  & $22$  & $12$ &   & $6$          \\
        14& &  & $55$  & $30$ &   & $15$         \\
        15& &  & $110$ & $60$ &   & $30$         \\
         \bottomrule
    \end{tabular}    
\end{table}
\normalsize
From lines 1, 2, 4, 6, 10, 12, and 14 of Table \ref{tab:lem3.7_q11} we get assertion (2). In particular, Lines 6 and 12 of Table \ref{tab:lem3.7_q11} yields, respectively, the $(11,5,2)$ Paley design and its complement, that is the symmetric $(11,6,3)$ design. The remaining cases are ruled out by using \texttt{GAP}.
\end{proof}

\begin{lemma}\label{case7}
Suppose that $X_{\alpha}=\mathrm{D}_{2(q-1)/(2,q-1)}$, 
the one of the following holds:
\begin{enumerate}
\item $G\cong \PGaL(2,2^2)$ and $\mathcal{D}$ is the $2$-$(10,4,2)$ design as in Line 10 of Table \ref{tab:alldesigns};
\item $G\cong\PSL(2,2^3)$ and $\mathcal{D}$ is the $2$-$(36,6,2)$ design as in Line 36 of Table \ref{tab:alldesigns};
\item $G\cong \PGaL(2,2^3)$ and $\mathcal{D}$ is one of the three $2$-$(36,6,6)$ designs as in Lines 37--39 of Table \ref{tab:alldesigns}.
\end{enumerate}
\end{lemma}

\begin{proof}
Suppose $X_{\alpha}=\mathrm{D}_{2(q-1)/(2,q-1)}$. 
Then $v=\frac{q(q+1)}{2}$ and $|\Out(X)|=(2,q-1)f$.
From $r\mid |X_{\alpha}||\Out(X)|$ and 
$r\mid \lambda(v-1)$, we have $r\mid 2(q-1)f$ and $r\mid \frac{\lambda(q+2)(q-1)}{2}$.

(1) If $q=2^f$, $f \geq 2$, from Table \ref{tab:subdegrees} we obtain that $X$ has a unique subdegree $2(q-1)$, 
then $r$ divides $2\lambda(q-1)$.
Combining this with $r\mid \frac{\lambda(q-2)(q-1)}{2}$,
we obtain that $r$ divides $\lambda\gcd(2(q-1), \frac{(q+2)(q-1)}{2})=\lambda(q-1)\gcd(2,\frac{q+2}{2})=\lambda(q-1)$,
for $q=2^f\ge 4$.

Let $r=\frac{\lambda(q-1)}{m}$, where $m$ is a positive integer.
From $r(k-1)=\lambda(v-1)$, we have $k=\frac{mq+2m+2}{2}$.
As $k\mid vr$ and $r\mid 2(q-1)f$,  we deduce that 
\[
(mq+2m+2)\mid 2q(q-1)(q+1)f.
\]
If $q=4$, then $r\mid 12$ and $k=3m+1$, and since $k\leq r$, we obtain $v=10$ and $(r,k,\lambda,m)=(6,4,2,1),(12,4,4,1)$ since $k$ divides the order of $G$ and hence that of $\PGaL(2,4)$ and $k<v$. Only the former occurs yielding a $2$-design as in (1). Hence, we may assume that $q>4$. Then $(mq+2m+2,2q)=2(mq/2+m+1,q)$ divides $2$, $4$, or $4(m+1,q)$ according as $m$ is even, or $m \equiv 1 \pmod{4}$ or $m \equiv 3 \pmod{4}$, respectively. 
Furthermore, $(mq+2m+2,q-1)=(3m+2,q-1)$ and $(mq+2m+2,q+1)=(m+2,q+1)$. 

If $m$ is even, then 
$$(mq+2m+2)\mid 2(3m/2+1)(m/2+1)f$$
and hence $$q+2 < 2(3m/2+1)(1/2+1/2)f=2(3m/2+1)f\text{.}$$
As $r\ge k$, we have
\[
2qf> r\ge k=\frac{mq+2m+2}{2}>\frac{mq}{2}
\]
 and so $4f>m$. Therefore, $2^{f}+2<2(6f+1)f$ and so $3\leq f \leq 10$.

If $m \equiv 3 \pmod{4}$, then 
$$mq+2m+2 \mid 4(m+1)(3m+2)(m+2)f\text{,}$$
and hence
$2^{f}+2 < 4(1+1/m)(3m+2)(m+2)f\leq 16(3m+2)(m+2)f/3 < 16(12f+2)(4f+2)f/3$ and $3\leq f\leq 21$.

If $m \equiv 1 \pmod{4}$, then $$mq+2m+2 \mid 4(3m+2)(m+2)f\text{,}$$ and hence $2^f+2<12(3m+2)f<12(12f+2)f$ and $3\le f \le 14$.

We give all the possibilities of $f,q,v,m$ and $k$,
satisfying $v>k$ in Table \ref{tab:Diedr_q_even}.
\begin{table}[htbp]
\tiny
  \centering
  \caption{\small Possible parameter for $X_\alpha \cong D_{2(q-1)}$ with $q$ even }\label{tab:Diedr_q_even}
  \label{tab:d2q-1_parameters}
  \begin{tabular}{ccccc}
    \toprule
    Line &$q$ & $v$ & $m$ &  $k$ \\
    \midrule
    1 &$8$  & $36$   & $1$  & $6$   \\
    2 &    &        & $4$  & $21$  \\
    \addlinespace
    3& $16$ & $136$  & $1$  & $10$  \\
     4 &      &        & $7$  & $64$  \\
    \addlinespace
   5 & $32$ & $528$  & $7$  & $120$ \\
    \addlinespace
    6 & $64$ & $2080$ & $11$ & $364$ \\
    \bottomrule
  \end{tabular}
\end{table}
\normalsize
 Finally, \texttt{GAP} aided computations show that only the case as in Line 1 occurs leading to cases (2) and (3) as in the Lemma statement.  

(2) If $q\equiv 1 \pmod{4}$, from Tab \ref{tab:subdegrees} we obtain that $X$ has two subdegrees $(q-1)/4$, 
then $r$ divides $\lambda(q-1)/2$. Let $r=\frac{\lambda(q-1)}{2m}$, where $m$ is a positive integer.
From $r(k-1)=\lambda(v-1)$, we have $k=mq+2m+1$.
By $k\mid vr$ and $r\mid 2(q-1)f$,  we have 
\[
(mq+2m+1)\mid q(q-1)(q+1)f.
\]
Arguing as above, we obtain 
$$(mq+2m+1) \mid (m+1)(2m+1)(3m+1)f\text{.}$$
Therefore, $mq+2m+1 \leq (m+1)(2m+1)(3m+1)f$, and so  
\begin{equation}\label{equaliz}
q+2 <2(2m+1)(3m+1)f\text{.}    
\end{equation}
As $r\geq k$, we have
$2(q-1)f\geq r \geq k=m(q+2)+1$ so  
$2f>m.$
Combining this with (\ref{equaliz}), we obtain $5^{f}+2 \leq q+2 <2(4f+1)(6f+1)f$
implying $1\le f\le 5$.
Note that $(mq+2m+1)\mid q(q-1)(q+1)f$, $m<2f$ and $q\equiv 1 \pmod 4$,  
so we have $(f,q,m)=(1,5,1)$ or $(2,9,1)$.
In the former case, one has $v=q(q+1)/2=15$ and $k=mq+2m+1=8$, and $G \cong \PGL(2,5)$ since $G$ acts flag-transitively on $\mathcal{D}$. However, $\PGL(2,5)$ does not have primitve permutation representations of degree $15$ by \cite{At}, and hence this case is ruled out.

In the latter case, $v=45$, $k=12$.
As $r=\lambda (q-1)/2m$, we get $r=4\lambda$.
Combining this with $r\ge k$ and $r\mid 2(q-1)f$, we get $\lambda \ge3$ and $\lambda\mid 8$,
so $\lambda =4$ or $8$, then $r=16$ or $32$, with $b=60$ or $120$, respectively.
Therefore, $(v,b,r,k,\lambda)=(45,60,16,12,4), (45,120,32,12,8)$. However, by using \texttt{GAP}, no examples occur.

(3) If $q\equiv -1 \pmod{4}$, from Tab \ref{tab:subdegrees}, 
we have that $X$ has subdegree $\frac{q-1}{2}$ (with multiplicity $\frac{q-1}{2}$) and $q-1$ (with multiplicity $\frac{q+5}{4}$). It follows that $r$ divides $x:=\lambda \frac{(q-1)^2}{4}$ and $y:=\lambda \frac{(q-1)(q+5)}{4}$,
and so $r$ divides $y-x=\frac{3\lambda(q-1)}{2}$.
Let $r=\frac{3\lambda(q-1)}{2m}$, where $m$ is a positive integer.
From $r(k-1)=\lambda(v-1)$, we have $k=(mq+2m+3)/3$.
By $k\mid vr$ and $r\mid 2(q-1)f$,  we have 
\[
(mq+2m+3)\mid 3q(q-1)(q+1)f.
\]
As above,
we obtain $$(mq+2m+3) \mid 9f(2m+3)(m+1)(m+3)\text{.}$$
Therefore, $mq+2m+3 \le 9f(2m+3)(m+1)(m+3)$, and so  
\begin{equation}\label{equaliz2}
q+2 < 18f(2m+3)(m+3)\text{.}    
\end{equation}
As $r\geq k$, we have
$6(q-1)f\geq 3r\geq 3k=m(q+2)+3$ 
so $6f>m$.
Combining this with (\ref{equaliz2}), we obtain $p^{f}+2 = q+2 < 18f(12f+3)(6f+3)$ with $f$ odd since $q=p^{f}$ and $q\equiv -1 \pmod{4}$. This implies $q\leq 3^{13}$. In Table \ref{tab:q_mod_3-} are listed give all the possibilities of $f,q,v,m$ and $k$, bearing in mind that $mq+2m+3\mid 3q(q-1)(q+1)f$ with $m<6f$, $q\equiv -1 (\mod 4)$ and $v>k$.
\begin{table}[htbp]
\tiny
    \centering
    \caption{Possible parameters for $X_\alpha \cong D_{q-1}$ when $q \equiv -1 \pmod 4$}
    \label{tab:q_mod_3-}
    \begin{tabular}{ccccc}
        \toprule
        $f$ & $q$ & $v$ & $m$ & $k$ \\
        \midrule
        1 & 7 & 28 & 1 & 4 \\
          &   &    & 2 & 7 \\
         &   &    & 5 & 16 \\         
        [1ex]
        1 & 19 & 190 & 1 & 8 \\
          &    &     & 2 & 15 \\
          &   &    & 5 & 36 \\         
        [1ex]
        1 & 31 & 496 & 1 & 12 \\
        [1ex]
        1 & 67 & 2278 & 1 & 24 \\
        [1ex]
        3 & 27 & 378 & 12 & 117 \\
        \bottomrule
    \end{tabular}
\end{table}
Recall that $r=\frac{3\lambda(q-1)}{2m}$ and $r\mid 2(q-1)f$,
we have $3\lambda \mid 4mf$. If $f=1$ then $3\lambda \mid 4m$, and so $3\mid m$, contradiction.
Thus, $f=3$ and $m=12$, and so $\lambda \mid 48$.
Combining this with $r= \frac{3\lambda(26)}{24} = \frac{13\lambda}{4}$ and $r>k=117$,
we have $\lambda=48$, and so $r=156$, $b=504$.
Therefore, $(v,b,r,k,\lambda)=(378,504,156,117,48)$. However, by using \texttt{GAP} we see that this case leads to no examples.
\end{proof}

\bigskip

\begin{lemma}\label{case8} If $X_{\alpha }$\bigskip $\cong \D_{\frac{2(q+1)}{(2,q-1)}}$, then there is a positive integer $x$ such that
$$
\left(v,k,r\right)=\left(\frac{q(q-1)}{2},\frac{(q-2)ex}{2\theta}+1,\frac{\lambda (q+1)\theta }{ex}\right)\text{,}  
$$
with $\lambda <20f^{2}$ and $x<\frac{10 \theta f}{e}$, where $(e,\theta )=(2,1)$ or $(1,f)$ according as $q$ is odd or not. Moreover, if $B$ is any block of $\mathcal{D}$ and 
$M$ is a maximal subgroup of $G$ containing $G_{B}$ and not $X$, then one of the following holds:
\begin{enumerate}
    \item $M\cap X$ is non-maximal in $X$, $G \cong \PGaL(2,3^{2})$ and $\mathcal{D}$ is the $2$-$(36,8,8)$ design as in Line 40 of Table \ref{tab:alldesigns}.
    \item $M\cap X$ is one of the subgroups of $\PSL(2,q)$ listed in Table \ref{MaxDickson}.
\end{enumerate}
\end{lemma}

\begin{proof} 
Suppose $X_{\alpha}=\mathrm{D}_{2(q+1)/(2,q-1)}$. 
Then $v=\frac{q(q-1)}{2}$ and $|\Out(X)|=(2,q-1)f$.
From $r\mid |X_{\alpha}||\Out(X)|$, we have $r\mid 2(q+1)f$. 

Moreover, from Table \ref{tab:subdegrees}, 
if $q\equiv 1 \pmod{4}$, then $X$ has subdegree
$(q+1)/2$ (with multiplicity $(q-3)/2$) and 
$q+1$ (with multiplicity $(q-1)/4$), and so 
$$ r\mid \lambda\cdot \left(\frac{(q+1)(q-1)}{4}-\frac{(q+1)(q-3)}{4} 
  \right)\text{,}$$
that is, $r\mid \lambda (q+1)/2$;
if $q\equiv -1 \pmod{4}$, then $X$ has subdegree $(q+1)/4$ (with multiplicity $2$), and so $r\mid \lambda (q+1)/2$;
if $q=2^f$, then $X$ has subdegree
$q+1$ (with multiplicity $(q-2)/2$), 
and so $r\mid \lambda f(q+1)$.
Hence, we may denote 
\[
r=\frac{\lambda (q+1)\theta }{ex},
\]
with $(e,\theta )=(2,1)$ or $(1,f) $
according as $q$ is odd or not. 
Furthermore, from $r(k-1)=\lambda (v-1)$, we have
\[
k=\frac{(q-2)ex}{2\theta}+1.
\]
Since $k\le r$, one obtains $\frac{(q-2)ex}{2\theta}<  \frac{\lambda (q+1)\theta }{ex}$ and hence $\left(\frac{ex}{\theta}\right)^{2}< 5\lambda$. On the other hand, $\frac{\lambda (q+1)\theta }{ex} \mid 2(q+1)f$ implies $\frac{\lambda}{2f}\leq \frac{ex}{\theta} $. Therefore, $\left(\frac{\lambda}{2f}\right)^{2}\leq \left(\frac{ex}{\theta}\right)^{2}< 5\lambda$ and hence $\lambda <20f^{2}$, implying $x<\frac{10 \theta f}{e}$.  

Let $B$ be any block of $\mathcal{D}$ and $M$ a maximal subgroup of $G$ containing $G_{B}$ but not $X$. As it can be easily deduced by \cite[Tables 8.2--8.3]{BHRD}, such a maximal group $M$ exists. If $M\cap X$ is non-maximal in $X$. Then $(G,X,M,M\cap X,k)$ is listed in Table \ref{Tab:Giudici} by \cite[Theorem 1.1]{gmaxi} and bearing in mind that $k \mid \left\vert M \right \vert$. By using \texttt{GAP}, we see that the only line that yields a flag-transitive, point-primtive design is Line 8 of Table \ref{Tab:Giudici}, and the design is the one in statement (1). 
\begin{table}[h!]
\footnotesize
\caption{The possibilities for $(G,X,M,M\cap X,k)$.}
\label{Tab:Giudici}
\tiny
\begin{tabular}{cllllll}
\toprule
Line & $G$ & $X$ & $M$ & $M\cap X$ & $v$ & $k$ \\
\midrule
1 & $\mathrm{PGL}(2,7)$         & $\mathrm{PSL}(2,7)$          & $\mathrm{D_{12}}$               &  $\mathrm{D_{6}}$ & $21$  & $6$\\  
2 & $\mathrm{PGL}(2,7)$         &                              & $\mathrm{D_{16}}$               &  $\mathrm{D_{8}}$ &  $21$ & $16$\\  
3 & $\mathrm{PGL}(2,3^{2})$         &  $\mathrm{PSL}(2,3^{2})$         & $\mathrm{D_{20}}$               &  $\mathrm{D_{10}}$ & $36$ & -\\  
4 & $\mathrm{PGL}(2,3^{2})$         &                              & $\mathrm{D_{16}}$               &  $\mathrm{D_{8}}$ & $36$ & $8$\\  
5 & $\mathrm{M_{10}}$           &                              & $C_{5}\rtimes C_4$              &  $\mathrm{D_{10}}$ & $36$ & -\\  
6 & $\mathrm{M_{10}}$           &                              & $C_{8}\rtimes C_{2}$            &  $\mathrm{D_{8}}$&  $36$ & $8$\\   
7 & $\mathrm{\PGaL}(2,3^{2})$   &                              & $C_{10}\rtimes C_{4}$           &  $\mathrm{D_{10}}$ & $36$ & $8$\\   
8 & $\mathrm{\PGaL}(2,3^{2})$   &                              & $C_{8}\cdot \Aut(C_8)$           &  $\mathrm{D_{8}}$ & $36$ & $8$\\ 
9 & $\mathrm{PGL}(2,11)$        &  $\mathrm{PSL}(2,11)$        & $\mathrm{D_{20}}$               &  $\mathrm{D_{10}}$ & $55$ &$10$\\
10 & $\mathrm{PGL}(2,11)$       &                              & $\mathrm{S_{4}}$                &  $\mathrm{A_{4}}$ & $55$ & -\\
11& $\mathrm{PGL}(2,19)$        &  $\mathrm{PSL}(2,19)$        & $\mathrm{S_{4}}$                &  $\mathrm{A_{4}}$ & $171$& -\\
\bottomrule
\end{tabular}
\end{table}
If $M\cap X$ is maximal in $X$, then $M\cap X$ is one of the subgroups of $\PSL(2,q)$ groups listed in Table \ref{MaxDickson}.
\end{proof}

\bigskip

Throughout the remainder of this section $M$ will have the same meaning as in Lemma \ref{case8}.

\bigskip

\begin{lemma}\label{MintX456}
$M\cap X$ cannot be isomorphic to $\S_{4}$. If $M\cap X \cong \A_{4}$ or $\A_{5}$, then one of the following occurs:
\begin{enumerate}
    \item $G\cong \PGL(2,5)$ and $\mathcal{D}$ is a $2$-$(10,4,2)$ design as in Line 10 of Table \ref{tab:alldesigns}; 
    \item $G \cong \PGaL(2,3^{2})$ and $\mathcal{D}$ is a $2$-$(36,8,8)$ design as in Line 40 of Table \ref{tab:alldesigns}.
\end{enumerate}
\end{lemma}
\begin{proof}
Assume that $M\cap X=\A_{4}$ with 
$q=p\equiv \pm 3\pmod{8}$ and $q\not\equiv \pm 1%
\pmod{10}$. There is one conjugacy class of subgroups of $X$ isomorphic to $\A_{4}$, and the normalizer of $\A_{4}$ in $\PGL(2,q)$ is $\S_{4}$ by \cite[§ 257]{Di}. Then $\frac{\left( q-2\right) ex%
}{2\theta }+1\mid 24$, thus $\left( q-2\right) x+1\mid 24$ since $(e,\theta)=(2,1)$; the only possibilities are $(q,x)=(5,1)$ or $(13,1)$, and hence $(v,k)=(10,4)$ or $(78,12)$, respectively. The former implies that $G\cong \PGL(2,5)$ and $\mathcal{D}$ is a $2$-$(10,4,2)$ design as in Line 9 of Table \ref{tab:alldesigns}, that is the assertion $(1)$. For the latter case no designs exist.

If $M\cap X=\S_{4}$, then $q=p\equiv \pm 1\pmod{8}$ and there are
two conjugacy classes of subgroups of $X$ isomorphic to $\S_{4}$, and these are fused in $\PGL(2,q)$ by \cite[§ 257]{Di}. Thus $\frac{\left( q-2\right) ex}{%
2\theta }+1\mid 24$, thus $\left( q-2\right) x+1\mid 24$ since $(e,\theta)=(2,1)$; the only possibility is $(q,x)=(7,1)$, and hence $v=21$ and $k=6$, but this case is ruled out by using \texttt{GAP}.

If $M\cap X=\A_{5}$ then $q\equiv \pm 1\pmod{10}$, where either $q=p$ or $q=p^2$ and $p\equiv \pm 3\text{(mod 10)}$ . There are
two conjugacy classes of subgroups of $X$ isomorphic to $\A_{5}$, and these are fused in $\PGL(2,q)$ by \cite[§ 259]{Di}. Thus $\frac{%
\left( q-2\right) ex}{2\theta }+1\mid 60\log _{p}q$, that is $\left( q-2\right) x+1\mid 60 \cdot \log_{p}q$ since $(e,\theta)=(2,1)$. 
Recalling that $x<10\theta f/e=5f$, we get $(q,x)=(9,1),(9,2), (11,1),(31,1), (61,1)$ 
but, by using \texttt{GAP}, only $(q,x)=(9,1)$ yields a design, and we obtain the assertion $(2)$.
\end{proof}

\begin{lemma}\label{MintX7}
If $M\cap X=\D_{\frac{2(q-1)}{(2,q-1)}} $, then one of the  following holds: 
\begin{enumerate}
    \item $\left(v,b,k,r\right)=\left(\frac{q(q-1)}{2},\frac{q(q+1)\lambda}{4},q-1,\frac{(q+1)\lambda}{2}\right)$ with $\lambda \mid 4f$;
    \item $G\cong \PGaL(2,2^{2})$ and $\mathcal{D}$ is the (complete) $2$-$(6,4,6)$ design as in Line 4 of Table \ref{tab:alldesigns};
    \item $G\cong \PGaL(2,2^{3})$ and one of the following holds:
    \begin{enumerate}
        \item $\mathcal{D}$ is one of the two $2$-$(28,3,2)$ designs as in Lines 19, 20 of Table \ref{tab:alldesigns}; 
        \item $\mathcal{D}$ is the $2$-$(28,3,4)$ design as in Line 21 of Table \ref{tab:alldesigns};
        \item $\mathcal{D}$ is the $2$-$(28,6,5)$ design as in Line 22 of Table \ref{tab:alldesigns};
        \item $\mathcal{D}$ is one of the five $2$-$(28,6,10)$ designs as in Lines 23--27 of Table \ref{tab:alldesigns}. 
    \end{enumerate}
    \item $G\cong \PSL(2,2^{4})$ and $\mathcal{D}$ is the Witt-Bose-Shrikhande linear space of order $2^4$;
    \item $G\cong \PSL(2,2^{4}):Z_{2}$ and $\mathcal{D}$ is the $2$-$(120,8,4)$ design as in Line 41 of Table \ref{tab:alldesigns};
     \item $G\cong \PGaL(2,2^{4})$ and one of the following holds:
     \begin{enumerate}
         \item $\mathcal{D}$ is the $2$-$(120,8,4)$ design as in Line 42 of Table \ref{tab:alldesigns};

         \item $\mathcal{D}$ is the $2$-$(120,8,8)$ design as in Line 43 of Table \ref{tab:alldesigns};
     \end{enumerate}
 \end{enumerate}
\end{lemma}
\begin{proof} If $M\cap X=\D_{\frac{2(q-1)}{(2,q-1)}} $, then $\frac{\left( q-2\right) ex}{2\theta }%
+1\mid 2(q-1)f$. Since $\gcd \left( \frac{\left( q-2\right) ex}{%
2\theta }+1,q-1\right) $ divides $\gcd \left( \left( q-2\right) ex+2\theta
,q-1\right) $, which in turn divides $ex-2\theta $, one has 
\begin{equation}\label{arch}
\frac{\left( q-2\right) ex}{2\theta }+1\mid 2(ex-2\theta )f\text{.} 
\end{equation} Assume that $ex\neq 2\theta $. Then 
\[
\left( q-2\right) ex+2\theta \leq 4\theta \left\vert ex-2\theta \right\vert f\text{.} 
\]%
If $q$ is odd, then $(e,\theta )=(2,1)$ and hence $x > 1$. Moreover, one has 
$$\left( q-2\right) 2+\frac{2}{x}\leq 8\frac{%
x-1 }{x}f\text{.}$$

Hence $q-2<q-2+\frac{1}{x}\leq 4%
\frac{x-1}{x}f\leq 4f$ and so $q=5$ or $9$. 
However, from \ref{arch} we have $(q-2)x+1\mid 4(x-1)f$, implying $x=q$, and so $k=(q-2)x+1=q(q-2)+1>v=q(q-1)/2$, a contradiction.

If $q=2^{f}$ is even, then $(e,\theta )=\left(1,f
\right) $ and $x \neq 2f$. Therefore,  
\begin{equation*}
\left( 2^{f}-2\right)x\leq 4\left\vert x-2f\right\vert f^{2} -2f=4\left( \left\vert x-2f\right\vert -1/(2f)\right) f^{2}
\end{equation*}%
Then $\left\vert x-2f\right\vert >1/(2f)$, as $f\geq 2 $ and $x \geq 1$, and
hence either $x<2f-1/(2f)$ or $x>2f+1/(2f)$.
In the latter case, one has 
\begin{equation*}
2^{f}-2<4\left( 1-\frac{2f+1/(2f)}{x}\right) f^{2}<4f^{2}\text{,} 
\end{equation*}
and hence $2^{f}-2-4f^{2}<0$, which implies $2\le f\le 8$. Now it easy to check that the admissible solutions of (\ref{arch}) with $2f+1/(2f)<x<\left(2^{f}+1\right)f$ lead to $(q,x)=(2^2,6),(2^3,13),(2^3,20),(2^5,103)$. These led to the admissible parameters of $\mathcal{D}$ as in Lines 1,2,4--7,19 of Table \ref{tab:q_even+/-}.

If $x<2f-1/(2f)$, then 
\begin{equation*}
2^{f}-2 \leq 4\left( 2f/x-1-1/(2xf)%
\right) f^{2}<4 \left( 2f-1\right) f^{2} \text{.} 
\end{equation*}
and hence $2\le f \le {14}$. By using (\ref{arch}) with $x<2f-1/(2f)$, 
and $\frac{\left( q-2\right) ex}{2\theta }+1\mid 2(q-1)f$ and $k\geq 3$, we obtain Lines 3, 8--18 of Table \ref{tab:q_even+/-}. From Table \ref{tab:q_even+/-}, using \texttt{GAP} we obtain assertions (2)--(6).
\begin{table}[htbp]
\tiny
    \centering
    \caption{\tiny Admissible numerical parameters for $\mathcal{D}$ when  $M\cap X=D_{\frac{2(q-1)}{(2,q-1)}} $}
    \label{tab:q_even+/-}
    \begin{tabular}{ccccccc}
    \toprule
    Line & $q$ & $v$ & $b$ & $r$ & $k$ & $\lambda$ \\
       \midrule
       1 & $2^2$ & $6$ & $15$ & $10$ & $4$ & $6$ \\
       2 & &  & $30$ & $20$ & & $12$ \\
        3 & $2^3$ & 28 & 126 & 27 & 6 & 5\\
        4 &  &  & $54$  &  &  $14$ & $13$  \\
        5 & & & $36$  &   & $21$ & $20$  \\
        6 & & & $72$  &  $54$ &  & $40$  \\
        7 &  &  & $108$ &  & $14$ & $26$ \\
       8 &  &  & $252$ & $27$ & $3$ & $2$ \\
       9 &  &  & $252$ & $54$ & $6$ & $10$ \\
       10 &         &   &   $504$ &       & $3$ & $4$ \\ 
        11 & $2^4$&  $120$    & $255$ & $17$ &   $8$  & $1$ \\
        12& &      & $510$ & $34$ &    & $2$ \\
       13 & &     & $1020$ & $68$ &    & $4$ \\
       14 &           &      & $2040$ & $136$ &     & $8$ \\
       15 & $2^5$ & 496 & 2728 & 55 & 10 & 1\\
       16 &  &  & 5456 & 110 & & 2\\
       17 & & & 8184 & 165 &  & 3\\
       18 &  &  & 
       $16368$ & $330$ &  & $6$ \\
       19 &  &  & $528$ &  & $310$ & $206$ \\
        \bottomrule
    \end{tabular}
\end{table}
\normalsize

Finally, if $ex=2\theta $, then $k=q-1$, $r=\frac{(q+1)\lambda}{2}$ and $b=\frac{q(q+1)\lambda}{4}$. Moreover, $\frac{\lambda (q+1)}{2}\mid 2(q+1)f$ and so $\lambda
\mid 4f$, and we obtain assertion (1).
\end{proof}

\begin{corollary}\label{MintX7Cor} In case (1) of Lemma \ref{MintX7}, if $q<50$ then one of the following holds:
\begin{enumerate}
    \item $\mathcal{D}$ is the $2$-$(6,3,2)$ and $G\cong \PSL(2,2^{2})$ as in Line 2 of Table \ref{tab:alldesigns};
    \item $\mathcal{D}$ is the $2$-$(6,3,4)$ and $G\cong \PGaL(2,2^{2})$ as in Line 3 of Table \ref{tab:alldesigns};
    \item $\mathcal{D}$ is the $2$-$(10,4,2)$ and $G\cong \PGL(2,5)$ as in Line 10 of Table \ref{tab:alldesigns};
    \item $\mathcal{D}$ is the $2$-$(28,7,2)$ and $G\cong \PSL(2,2^{3}),\PGaL(2,2^{3})$ as in Line 28 of Table \ref{tab:alldesigns};
    \item $\mathcal{D}$ is the $2$-$(28,7,6)$ and $G\cong \PGaL(2,2^{3})$ as in Line 29 of Table \ref{tab:alldesigns};
    \item $\mathcal{D}$ is the $2$-$(36,8,8)$ and $G\cong \PGaL(2,3^{2})$ as in Line 40 of Table \ref{tab:alldesigns}.
\end{enumerate}    
\end{corollary}
\begin{proof}
The result is simply the output of an exhaustive search of the $2$-designs with parameters case (1) of Lemma \ref{MintX7} and $\PSL(2,q)\leq G \leq \PGaL(2,q)$ for $q<50$. Note that, the $2$-$(36,8,4)$ design with $G\cong \PSigL(2,3^{2})$ is ruled out at this stage as we are assuming that $G$ acts point-primitively on $\mathcal{D}$.       
\end{proof}

\begin{lemma}\label{MinTX8} If $M\cap X=D_{\frac{2(q+1)}{(2,q-1)}}$, then one of the following holds:
\begin{enumerate}
    \item $\left(v,b,k,r\right)=\left(2^{f-1}(2^{f}-1),2^{f-1}3^{2}(2^{f}-1)\frac{\lambda }{2},\frac{2^{f}+1}{3},3(2^{f}+1)\frac{\lambda }{2}\right)$ with $f$ odd and $\lambda$ even. Moreover, $X_{\alpha}\cong D_{2(2^f+1)}$ and $X_{B}\leq D_{2(2^f+1)}$;
    \item $G\cong \PGaL(2,3^2)$ and $\mathcal{D}$
        is the $2$-$(36,8,8)$ design as in Line 40 of Table \ref{tab:alldesigns};
    \item $G\cong \PGL(2,5)$ and $\mathcal{D}$ is the $2$-$(10,4,2)$ design as in Line 10 of Table \ref{tab:alldesigns};
    \item $G\cong \PGaL(2,2^2)$ and $\mathcal{D}$ is the is the (complete) $2$-$(6,4,6)$ design as in Line 5 of Table \ref{tab:alldesigns};
    \item $G\cong \PGaL(2,2^3)$ and one of the following holds:
    \begin{enumerate}
        \item $\mathcal{D}$ is the $2$-$(28,6,5)$ design as in Line 22 of Table \ref{tab:alldesigns};
        \item $\mathcal{D}$ is one of the five $2$-$(28,6,10)$ designs as in Lines 23--27 of Table \ref{tab:alldesigns};
        \item $\mathcal{D}$ is one of the two $2$-$(28,9,8)$ designs as in Lines 30, 31 of Table \ref{tab:alldesigns};
        \item $\mathcal{D}$ is the $2$-$(28,9,16)$ design as in Line 32 of Table \ref{tab:alldesigns};
        \item $\mathcal{D}$ is the $2$-$(28,18,34)$ design as in Line 35 of Table \ref{tab:alldesigns};
    \end{enumerate}
    \item $G\cong \PSL(2,2^{4})$ and $\mathcal{D}$ is the Witt-Bose-Shrikhande $2$-$(120,8,1)$ design of order $2^4$;
    \item $G\cong \PSL(2,2^4):Z_{2}$ and $\mathcal{D}$ is the $2$-$(120,8,4)$ design as in Line 41 of Table \ref{tab:alldesigns};
    \item $G\cong \PGaL(2,2^4)$ and one of the following holds:
    \begin{enumerate}
        \item $\mathcal{D}$ is the $2$-$(120,8,4)$ design as in Line 42 of Table \ref{tab:alldesigns};

        \item $\mathcal{D}$ is the $2$-$(120,8,8)$ design as in Line 43 of Table \ref{tab:alldesigns};
    \end{enumerate}
    \item $G\cong \PGaL(2,2^5)$ and $\mathcal{D}$ is the $2$-$(496,22,14)$ design as in Line 45 of Table \ref{tab:alldesigns};
    
\end{enumerate}
\end{lemma}

\begin{proof}
If $M\cap X=D_{\frac{2(q+1)}{(2,q-1)}}$, then $\frac{\left( q-2\right) ex}{2\theta }%
+1\mid 2(q+1)f$. Since 
\[
\gcd \left( \left( q-2\right) ex+2\theta ,q+1\right) =\gcd \left(
3ex-2\theta ,q+1\right) \text{,}
\]%
it follows that 
\begin{equation}\label{friday}
\frac{\left( q-2\right) ex}{2\theta }+1\mid 2(3ex-2\theta )f\text{.} 
\end{equation}
If $3ex\neq 2\theta $, then%
\[
\frac{\left( q-2\right) ex}{2\theta }+1\leq 2\left\vert 3ex-2\theta
\right\vert f\text{.}
\]%
If $q$ is odd, then $(e,\theta)=(2,1)$, then $\left( q-2\right) x+1\leq 2(6x-2)f$
and so $q-2<2\left(6-2/x\right) f<12f$. Then $q=3^{2},3^{3},5,5^{2},7,11,13$. 
Actually, only $(q,x)=(3^2,1),(5,1),(7,3)$ are admissible by (\ref{friday}), $x<\frac{10\theta f}{e}=5f$ and Lemma \ref{lem:basic-params}. These values lead to $(v,k)=(36,8),(10,4), (21,16)$, respectively, and \texttt{GAP}-aided computations show that we obtain assertions (2)--(3).

If $q=2^{f}$ is even, then $(e,\theta )=\left( 1,f\right) $. Then 
\begin{equation*}
\left( q-2\right) x \leq 4\left\vert 3x-2f\right\vert
f^{2}-2f
\end{equation*}%
Then $\left\vert 3x-2f\right\vert >1/2f$ since $f\ge 2$, and
hence $x<\frac{2}{3}f-\frac{1}{6f}$ or $x>\frac{2}{3}f+\frac{1}{6f}$. In the latter case,%
\[
2^{f}-2 < \left( 1-\frac{2f}{3x}\right) \cdot 12f^{2}<12f^{2}
\]
and so $2^{f}-2-12f^{2}<0$, which implies $2 \leq f \leq 10$. Now, it easy to check that the admissible solutions of (\ref{friday}) with $\frac{2}{3}f+\frac{1}{6f}<x<\left(2^{f}+1\right)f$, $2 \leq f\leq 10$ and Lemma \ref{lem:basic-params}, lead to the admissible parameters listed in Lines 1--12, 17--27 of Table \ref{tab:parameters.3.9.5}, and using \texttt{GAP} we obtain assertions (4)--(9).

\begin{table}[H]
\tiny
    \centering 
    \caption{\tiny Admissible numerical parameters for $\mathcal{D}$ when  $M\cap X=D_{\frac{2(q+1)}{(2,q-1)}} $, $q$ even
    }%
     \begin{tabular}{ccccccc}
     \toprule
    Line & $q$ & $v$ & $b$ & $r$ & $k$ & $\lambda$ \\
    \midrule
    1 & $2^2$ & $6$ & $15$ & $10$ & $4$ & $6$ \\
       2 & &  & $30$ & $20$ & & $12$ \\
        3 & $2^3$ & 28 & 126 & 27 & 6 & 5\\
        4 &  &  & $252$ & $54$ & & $10$ \\
       {5} & &  & 84 & 27 & 9 & 8 \\
       {6} &  &  & 168 & 54 &  & 16 \\
        {7} &  &  & 42 & 27 & 18 & 17 \\
       {8} & &  & 84 & 54 & & 34 \\
       9 & $2^4$&  $120$    & $255$ & $17$ &   $8$  & $1$ \\
        10& &      & $510$ & $34$ &    & $2$ \\
       11 & &     & $1020$ & $68$ &    & $4$ \\
       12 &           &      & $2040$ & $136$ &     & $8$ \\
       13 & $2^5$ & 496 & 2728 & 55 & 10 & 1\\
       14 &  &  & 5456 & 110 & & 2\\
       15 & & & 8184 & 165 &  & 3\\
       16 &  &  & 
       $16368$ & $330$ &  & $6$ \\
       {17} &  &  & 3720  &  165 &  22 & 7  \\
        {18} &  &  & 7440  &  330 &   & 14  \\
        {19} &  &  &  496 & 55  & 55  & 6  \\
        {20} &  &  & 992  &  110 &   & 12  \\
        {21} &  &  & 1488  &  165 &   & 18  \\
        {22} &  &  & 2976  &  330 &   & 36  \\
        {23} & $2^6$ & 2016 & 2016  &  156 & 156  & 12  \\
        {24} &  &  & 2520  &  195 &   & 15  \\
        {25} &  &  & 3360  &  260 &   & 20  \\
        {26} &  &  & 5040  &  390 &   & 30  \\
        {27} &  &  & 10080  & 780 &  & 60  \\
        \bottomrule
    \end{tabular}\label{tab:parameters.3.9.5}
\end{table}

If $1 \leq x<\frac{2}{3}f-\frac{1}{6f}$, then%
\begin{equation*}
2^{f}-2 \leq \left(2^{f}-2\right) x \leq 4(2f-3x)
f^{2}-2f \leq 8f^{3}-12f^{2}-2f
\end{equation*}
Then $2\leq f \leq 14$, and we obtain Lines 13--16 of Table \ref{tab:parameters.3.9.5}, and using \texttt{GAP} we see that these sets of parameters do not give any design.

If $3ex=2\theta $ then $k=\frac{q+1}{%
3}$ with $q\equiv 2 \pmod{3}$, and so $r=\frac{3(q+1)\lambda }{2}$ with $r\mid 2(q+1)f$ and
hence $\lambda \mid \frac{4f}{3}$. If $q$ is odd, then $r\mid \lambda \frac{q+1}{2}$ by Table \ref{tab:subdegrees}, whereas $r=\frac{3(q+1)\lambda }{2}$, a contradiction. Therefore, $q$ is even then $\lambda$ is even since $r=\frac{3(q+1)\lambda }{2}$, and hence $f$ is odd since $q\equiv 2 \pmod{3}$, that is to say, the assertion (1).

\end{proof}

\begin{lemma}\label{MintX123}
If either $M\cap X=\PSL(2,q^{1/t})$  with $q$ odd and $t$ odd prime, or $M\cap X=\PGL(2,q^{1/t})$ with $t$ prime and $q^{1/t} \neq 2$ for $q$ even or with $t=2$ for $q$ odd, then $t=2$ and one of the following holds:
\begin{enumerate}
 \item $G\cong \PSL(2,2^{4})$ and 
 $\mathcal{D}$ is the Witt-Bose-Shrikhande $2$-$(120,8,1)$ design;
 
 \item $G\cong \PSL(2,2^4):Z_2$ and $\mathcal{D}$ is the $2$-$(120,8,4)$ design as in Line 41 of Table \ref{tab:alldesigns};
 
 \item $G\cong \PGaL(2,2^4)$ and one of the following holds:
 \begin{enumerate}
     \item $\mathcal{D}$ is the $2$-$(120,8,4)$ design as in Line 42 of Table \ref{tab:alldesigns};

     \item $\mathcal{D}$ is the $2$-$(120,8,8)$ design as in Line 43 of Table \ref{tab:alldesigns};
 \end{enumerate}
 
    \item $G\cong \PGaL(2,3^{2})$ and $\mathcal{D}$ is a $2$-$(36,8,8)$ design as in Line 40 of Table \ref{tab:alldesigns};
      \item $G\cong \PGaL(2,11^{2})$ and $(v,b,r,k,\lambda)=(7260,29524,488,120,8)$;
     \item $G\cong \PGaL(2,2^{8})$ and $(v,b,r,k,\lambda)=(32640,1048560,4112,128,16)$.
\end{enumerate}
\end{lemma}

\begin{proof}
Assume that $M\cap X=\PSL(2,q^{1/t})$  with $q$ odd and $t$ odd prime, or $M\cap X=\PGL(2,q^{1/t})$ with $t$ prime and $q^{1/t} \neq 2$ for $q$ even or with $t=2$ for $q$ odd. 
In these cases, there are $1$, $1$ or $2$ conjugacy $\PSL(2,q)$-classes, respectively, and in the latter case the $2$ classes are fused in $\PGL(2,q)$ by \cite[§260]{Di}. In particular, the normalizer in $\PGL(2,q)$ of $M\cap X$ is $\PGL(2,q^{1/t})$. Then $k \mid \left\vert G_B \right\vert$ and $\left\vert G_B \right\vert \mid \left\vert \PGL(2,q^{1/t}) \right\vert \cdot f$, and hence 
\begin{equation}\label{redPSL}
\frac{\left( q-2\right) ex}{2\theta }+1\mid q^{1/t}(q^{2/t}-1)\cdot f\text{.} 
\end{equation}
\noindent \textbf{(i)} Assume that $t\geq 5$. If $q$ is odd, then $(e,\theta )=(2,1)$ and $q-1 \leq q^{1/t}(q^{2/t}-1)\cdot f\text{,}$ a contradiction. Then $q=2^{f}$ with $f>t$ and $(e,\theta )=(1,f)$,
hence
$$(2^{f-1}-1)x \leq 2^{f/5}(2^{2f/5}-1)\cdot f^{2}-f\text{,}$$
and so $x \leq 11$, with either $t=5$ and $f = 10,15,20,25$, or $t=7$ and $f= 14$, as $t|f$. However, none of the cases satisfy (\ref{redPSL}). 


\noindent \textbf{(ii)} Assume that $t=3$. Then $\left( q-2\right) ex+2\theta\leq
q^{1/3}(q^{2/3}-1)\cdot 2 \theta f$ and so we obtain 
\begin{equation}\label{ime}
xe<2\theta f\text{.}    
\end{equation}

\begin{eqnarray*}
\gcd \left( \left( q-2\right) ex+2\theta ,q^{1/3}\right)  &=&\gcd \left(
2ex-2\theta ,q^{1/3}\right) \mid \gcd(2,q)\cdot |ex-\theta | \\
\gcd \left( \left( q-2\right) ex+2\theta ,q^{1/3}-1\right)  &=&\gcd \left(
ex-2\theta ,q^{1/3}-1\right) \mid |ex-2\theta|  \\
\gcd \left( \left( q-2\right) ex+2\theta ,q^{1/3}+1\right)  &=&\gcd \left(
3ex-2\theta ,q^{1/3}+1\right) \mid |3ex-2\theta| 
\end{eqnarray*}%
and from (\ref{redPSL}), we obtain
\begin{equation}\label{stum}
\frac{\left( q-2\right) ex}{2\theta }+1\mid \gcd(2,q)\cdot |ex-\theta |\cdot \left|
ex-2\theta \right| \cdot \left| 3ex-2\theta \right| \cdot f
\end{equation}
Assume that $ex=\theta $. Then $k=\frac{\left( q-2\right) ex}{2\theta }%
+1=\frac{q}{2}$, with $q$ even, divides $q^{1/3}(q^{2/3}-1)f$ by (\ref{redPSL}), and hence $%
q^{2/3}\mid 2f$, a contradiction. Similarly, if $%
ex=2\theta $ then $k=q-1$ divides $q^{1/3}(q^{2/3}-1)f$,
and hence $q^{2/3}+q^{1/3}+1\mid f$, again a contradiction. Finally, if $3ex=2\theta $ then $k=\frac{%
q+1}{3}$ divides $q^{1/3}(q^{2/3}-1)f$ and
hence $q^{2/3}-q^{1/3}+1\mid 3f$, a contradiction. Therefore, $ex\neq 2\theta/3,\theta, 2\theta $.

If $q$ is odd, then $(e,\theta )=(2,1)$, and $2\leq x<f$ by $ex\neq 2\theta $ and (\ref{ime}). 
Furthermore, from (\ref{stum}) we have
\begin{equation}\label{BaptLeonD}
\left( q-2\right) x+1 \mid (2x-1)\cdot \left( 2x-2\right) \cdot \left(
6x-2\right) \cdot f\text{,} 
\end{equation}
implying
\begin{equation*}
q-2 < (2x-1) 
\cdot \left( 1-\frac{1}{x}\right) 
\cdot \left( 3x-1\right) \cdot 4f 
<(2f-1) \cdot \left( 3f-1\right) \cdot 4f. 
\end{equation*}
Thus, we obtain $q=3^{3},3^{6},5^{3}$, or $7^{3}$. 
If $f=3$ then $x=2$;
if $f=6$ then $2\le x\le 5$.
However, none of these cases fulfill (\ref{redPSL}), a contradiction.

If $q=2^{f}$, then $(e,\theta )=(1,f)$ and 
$x\neq 2f/3,f,2f$ by
$ex\neq 2\theta/3,\theta, 2\theta$,
and $x<2f^{2}$ by (\ref{ime}). Then from (\ref{stum}) we obtain 
\begin{equation*}
\frac{\left( 2^{f-1}-1\right) x}{f }+1\mid 2\cdot \left\vert x-f \right\vert\cdot \left\vert
x-2f \right\vert \cdot \left\vert 3x-2f\right\vert \cdot f.
\end{equation*}
Hence, we have 
\begin{equation}\label{q+}
2^{f-1}-1 < \left\vert x-f\right\vert \cdot \left\vert x-2f\right\vert \cdot \left\vert 1-\frac{2f}{3x}\right\vert \cdot  6f^{2}.
\end{equation}%

Assume that $1\leq x<\frac{2}{3}f$. Then
\begin{equation*}
2^{f-1}-1 
<(f-x)\cdot(2f-x)\cdot\left(\frac{2f}{3x}-1\right)\cdot6f^2
\leq \left(f -1\right) \cdot \left(2f-1\right) \cdot\left(\frac{2f}{3}-1\right) \cdot  6f^{2},
\end{equation*}
and hence $f=6,9,12,15,18,21,24,27$ since $f \equiv 0 \pmod{3}$. However, no pair $(q,x)$ fulfills (\ref{redPSL}).

Assume that $\frac{2}{3}f<x<2f^{2}$. Then (\ref{q+}) becomes
\begin{equation*}
2^{f-1}-1 
<2f^2\cdot2f^2\cdot6f^2,
\end{equation*}
hence, $6\le f \le 36$ with $f \equiv 0 \pmod{3}$. However, no pair $(q,x)$ fulfills (\ref{redPSL}). 

\noindent \textbf{(iii)} Assume that $t=2$. Then
\begin{equation}\label{repubblica2}
\frac{\left( q-2\right) ex}{2\theta }+1\mid q^{1/2}(q-1)\cdot f
\end{equation}
\begin{eqnarray*}
\gcd \left( \left( q-2\right) ex+2\theta ,q^{1/2}\right)  &=&\gcd \left(
2ex-2\theta ,q^{1/2}\right) \mid 2|ex-\theta | \\
\gcd \left( \left( q-2\right) ex+2\theta ,q-1\right)  &=&\gcd \left(
ex-2\theta ,q-1\right) \mid |ex-2\theta| \text{.}
\end{eqnarray*}%
therefore, 
\[
\frac{\left( q-2\right) ex}{2\theta }+1\mid \left\vert ex-\theta \right\vert \cdot \left\vert ex-2\theta \right\vert \cdot
2f
\]%

Assume that $ex\neq \theta, 2\theta $. If $q$ is odd, then $(e,\theta )=(2,1)$ and $1<x<5f$ as $x<\frac{10\theta f}{e}$ by Lemma \ref{case8}.
Moreover, we have  
\[
\left( q-2\right)x+1\mid \left( 2x-1 \right) \cdot \left( x-1 \right) \cdot4f,
\]
and hence $q-2<(2x-1)\cdot(1-\frac{1}{x})\cdot4f<(10f-1)4f$, which implies 
$q\in \{3^2,5^2,7^2,11^2,3^4,5^4,3^6\}$. 
However, none of them fulfill (\ref{repubblica2}) given $1<x<5f$.

If $q=2^{f}$, then $(e,\theta )=(1,f)$ with $f>2$,
 $1\le x< 10f^2$ by Lemma \ref{case8}, and $x \neq f,2f$. 
Moreover, we have 
\[
\frac{\left( 2^{f-1}-1\right)x}{f}+1
\mid 
|x-f|\cdot |x-2f| \cdot2f,
\]
and so 
$(2^{f-1}-1)x<|x-f|\cdot |x-2f| \cdot2f^2$.
If $1\le x <f$, then 
$2^{f-1}-1<2f^{2}(f-1)(2f-1)$, implying 
$4\le f\le 20$ with $f$ even;
if $f<x<2f$, then 
$2^{f-1}-1<2f^{2}(1-\frac{f}{x})(2f-x)<2f^3$,  implying 
$4\le f\le 12$ with $f$ even;
if $2f<x<10f^2$, then 
$2^{f-1}-1<2f^{2}(x-f)(1-\frac{2f}{x})<2f^2\cdot 10f^2$,  implying 
$4\le f\le 22$ with $f$ even.
Therefore, we have $4\le f\le 22$ with $f$ even.
However, for each $f$, there is no $x$ satisfying  
(\ref{repubblica2}) with $1\le x<10f^2$, bearing in mind also that $x<(q+1)\theta/e=(q+1)f$ as $r>\lambda$.


If $ex=\theta $, then $k=\frac{q}{2}$, 
and hence $q$ even. By (\ref{repubblica2}),
we have $\frac{q}{2}$ divides $q^{1/2}(q-1)f$, 
and so $q^{1/2}\mid 2f$, that is, $2^{f/2}\mid 2f$,
implying $q=2^4$ or $2^8$.
Since $ex=\theta$ and $(e,\theta)=(1,f)$, we have 
$x=f$. Moreover, from Lemma \ref{case8} we obtain $v=q(q-1)/2$, $r=\lambda(q+1)f/x=\lambda(q+1)$, 
and so $\lambda \mid 2f$ as $r\mid 2(q+1)f$. 
If $q=2^4$, we have $v=120$, $k=8$ and $r=17\lambda$. 
As $\lambda \mid 8$, we deduce the possible parameter $(v,b,r,k,\lambda)$:
\[
(120,255,17,8,1), (120,510,34,8,2), 
(120,1020,68,8,4), (120,2040,136,8,8).
\]
The first one is the Witt-Bose-Shrikhande $2$-$(120,8,1)$ design with $G\cong \PSL(2,2^{4})$.
For the remaining, by Magma, up to isomorphism, 
there exist two designs with 
$(v,b,r,k,\lambda)$ $=(120,1020,68,8,4)$, 
$G\cong \PGaL(2,2^{4})$ or $\PSL(2,2^4):Z_2$, respectively,
and a unique design with 
$(v,b,r,k,\lambda)=(120,2040,136,8,8)$, 
$G\cong \PGaL(2,2^{4})$. Hence, we obtain assertions (1)--(3).

Thus $q=2^{8}$, $q^{1/2}=2f$ and $\left\vert G_{B}\right\vert \mid q^{1/2}(q-1)f$, and hence 
$G\cong \PGaL(2,2^{8})$ and $G_{B}\leq \PGL(2,2^{4}): Z_{2^{3}}$. Moreover, $(v,k)=(32640,2^{7})$ and $r=257\lambda$ with $\lambda \mid 2^{4}$. Then $3\cdot 5 \cdot 17 \mid b$ forcing $\left\vert G_{B}\right\vert =2^{7}$. Thus $b=1048560$, $r=4112$ and so $\lambda=16$, which is the assertion (6). 

If $ex=2\theta $, then $k=q-1$ and $r=\frac{\lambda }{2}(q+1)$. Let  $\rho=\left\vert G:X \right\vert$, where $\rho \mid \gcd(2,q-1)f$. Since $X$ acts transitively on $\mathcal{D}$, it follows that $\left\vert G_{\alpha}:X_{\alpha}\right\vert=\rho$. Then $\left\vert G_{\alpha}\right\vert=\frac{2(q+1)}{\gcd(2,q-1)}\rho$ and $\left\vert G_{\alpha,B}\right\vert=\frac{4\rho}{\gcd(2,q-1)\lambda}$, and hence $\left\vert G_{B}\right\vert=(q-1)\frac{4\rho}{\gcd(2,q-1)\lambda}$. Then 
\begin{equation}\label{icsB}
\frac{q-1}{\gcd(2f,q-1)}\mid \left \vert X_{B}\right\vert
\end{equation}
with $X_{B}\leq \PGL(2,q^{1/2})$. Since $X_{\alpha}\cong D_{\frac{2(q+1)}{\gcd(2,q-1)}}$, it follows that $\left\vert X_{B}\right\vert_{p}\mid \gcd(2,p)$, and since $q^{1/2}\neq 2$, it follows that $X_{B}\neq \PGL(2,q^{1/2})$, and filtering the list of the subgroups of \cite[Theorem 2]{COT} with respect to (\ref{icsB}), then $X_{B}$, $q$, $v$ and $k$ are as in Table \ref{tab:XBinPGL}.
\begin{table}[htbp]
\tiny
    \centering
    \caption{Admissible $X_B\leq PGL(2,q^{1/2})$ for $k=q-1$}\label{tab:XBinPGL}
    \begin{tabular}{clccc}
    Line & $X_{B}$ & $q$ & $v$ & $k$ \\
       \toprule
       1 & $A_{4}$  & $3^{2}$    & $36$ &  $8$    \\
       2  &         & $5^{2}$     & $300$     & $24$ \\
       3  &  $S_{4}$  &  $7^{2}$    &   $1176$  & $48$ \\
       4  &  $A_{5}$       &  $11^{2}$  & $7260$ &  $120$ \\
       5  &  $\leq D_{2(q^{1/2}-1)}$       &  $3^{2}$ &  $36$ &  $8$ \\
       6  &  $\leq D_{2(q^{1/2}+1)}$       &  $3^{2}$  & $36$ &  $8$ \\
       7  &         &  $3^{4}$  & $3240$ &  $80$ \\
       8  &          & $5^{2}$     & $300$  & $24$ \\
        \bottomrule
    \end{tabular}
\end{table}
It is easy to see that in case as in Line 4 of Table \ref{tab:XBinPGL}, one has $G_{B}\cong S_{5}$ since $k=120$, and hence $G=\PGaL(2,11^{2})$. Now, easy computations lead to assertion (5). 

In the remaining cases, by using \texttt{GAP}, we see that only cases as in Line 5 or 6 of Table \ref{tab:XBinPGL} really occur, yielding either to a $2$-$(36,8,4)$ design or to a $2$-$(36,8,8)$ design, according as $G$ is isomorphic to $\PSiL(2,3^{2})$ or $\PGaL(2,3^{2})$, respectively. Thus we obtain assertion (4) since $\PSiL(2,3^2)$ is not primitive on $36$ points. Hence, the proof is completed. 
\end{proof}

\begin{lemma}\label{MintX9} If $M\cap X=Z_{p}^{f}:Z_{\frac{p^{f}-1}{\gcd(2,p^{f}-1)}}$, then one of the following holds:
\begin{enumerate}
   \item  $\left(v,b,k,r\right)=\left(2^{f-1}(2^{f}-1),(2^{2f}-1)\lambda,2^{f-1},(2^{f}+1)\lambda\right)$ with $\lambda \mid 2f$; Moreover, $X_{\alpha}\cong D_{2(2^f+1)}$ and $X_{B}\leq Z^{f}_{2}$; 
    \item $\left(v,b,k,r\right)=\left(\frac{q(q-1)}{2},\frac{q(q+1)\lambda}{4},q-1,\frac{(q+1)\lambda}{2}\right)$ with $\lambda \mid 4f$. Moreover, $X_{\alpha}\cong \D_{\frac{2(q+1)}{(2,q-1)}}$ and $X_{B}\leq \D_{\frac{2(q-1)}{(2,q-1)}}$
    \item $\mathcal{D}$ and $G$ are as in cases (2)--(6) Lemma \ref{MintX7};
        \item $G\cong \PGaL(2,2^{3})$ and $\mathcal{D}$ is one of the two $2$-$(28,12,11)$ designs as in Lines 33, 34 of Table \ref{Intro};
       \item $G\cong \PGaL(2,2^5)$ and $\mathcal{D}$ is the $2$-$(496,4,1)$ design as in Line 44 of Table \ref{Intro}.
       
\end{enumerate}    
\end{lemma}

\begin{proof} If $M\cap X=Z_{p}^{f}:Z_{\frac{p^{f}-1}{\gcd(2,p^{f}-1)}}$, then either $X_{B}=Z_{p}^{t}$ with $t\leq f$ or 
$X_{B}=Z_{m}$ with $m \mid \frac{p^{f}-1}{\gcd(2,p^{f}-1)}$, or $X_{B}=Z_{p}^{t}:Z_{m}$ with $m \mid p^{\gcd(t,f)}-1$ and $m \leq \frac{p^{f}-1}{\gcd(2,p^{f}-1)}$ by \cite[Proposition III.17.3]{Pass}. In the second case, there is a maximal subgroup $N$ of $G$ such that $X_{B}\leq N\cap X=D_{\frac{2(q-1)}{\gcd(2,q-1)}}$, and hence the same conclusions of Lemma \ref{MintX7} hold, which are the assertions (2) and (3).

Assume that $X_{B}=(Z_{p})^{t}$. Then 
\begin{equation}\label{May}
\frac{\left( q-2\right) ex}{2\theta }+1\mid p^{t}\cdot \gcd(2,q-1)\cdot f\text{.}
\end{equation}
Since 
$$\gcd
\left( \left( q-2\right) ex+2\theta ,p^{t}\right) =\gcd \left( 2ex-2\theta
,p^{t}\right)$$
divides $\gcd(2,q)\cdot \left\vert ex-\theta \right\vert$,
it follows that 
\[
\frac{\left( q-2\right) ex}{2\theta }+1\mid |ex-\theta |\cdot 2f
\]
If  $ex=\theta$, then $k=\frac{q}{2}$ with $q$ even, and $k$ divides $2^{t}\cdot f$, 
$r=\lambda (q+1)$, and so $\lambda\mid 2f$ as $r\mid 2(q+1)f$,
and we obtain assertion (1). Hence, suppose that $ex\neq \theta$. If $q$ is odd, then $(e,\theta )=(2,1)$ and $1\le x<5f$. Then $\left( q-2\right) x+1\mid (2x-1)\cdot
2f$ and hence $q-2<(2-\frac{1}{x})\cdot 2f<4f$, and so $q=3^2$ or $5$.
However, by (\ref{May}) and $r=\frac{\lambda(q+1)}{2x}>\lambda$, there no such $x$ exists.
Thus $q=2^{f}$ with $f \geq 2$. Then $(e,\theta )=(1,f)$, $x \neq f$, 
$x<10f^2$ and 
\[
\frac{\left( 2^{f-1}-1\right) x}{f}+1\mid  2f \cdot \left\vert x-f\right\vert. 
\]%
If $1\leq x <f$ then $2^{f-1}-1< 2f^{2}(f-1)$, forcing $2\leq f\leq 12$;
if $f<x<10f^2$ then $2^{f-1}-1<2f^{2}(1-\frac{f}{x})<2f^{2}$, forcing $2\leq f\leq 8$.
Hence, $2\leq f\leq 12$.
Note that $r=\lambda(q+1)f/x$ and $r\mid 2(q+1)f$, so $\lambda\mid 2x$,
and then by (\ref{May}) we get all the possible parameters listed in Table \ref{tab:3.9.8}.
(Note Lines 3, 4, 12--17, 21--24 correspond to $1\leq x <f$, and remaining Lines to $f<x<10f^2$, respectively.) 
By \texttt{GAP}, Lines 2,6,8,10,11,13-20 are ruled out. For Lines 1,3,4,5,7,9 and 12, we obtain assertions (3), (4) and (5).
In Lines 21--24, the order $X_{B}$ is $2^{3}3$, $2^(3)$, $2^{2}3$ or $2^{2}$. However, $X\cong \PSL(2,2^{7})$ has no subgroups of order $2^{3}3$ By \cite{Di}. Moreover, in the third and in the fourth cases there is a $Z_{7}$ centralizing $X_{B}$ since $\left\vert G_{B}\right\vert$, however this is impossible since $C_G(Z_{7})\cong Z_{7} \times S_{3}$. The same reasoning forces $G_{B} \cong Z_{2}^{3}:Z_{7}$ in Line 23. By using \texttt{GAP}, we see that $G$ contains two conjugacy classes of subgroups isomorphic to $Z_{2}^{3}:Z_{7}$. Now, analyzing the orbits of length $56$ for a fixed representative of any of such conjugacy classes classes we see that non of these orbits leads to a design with $\lambda=3$. Thus, there are no example with parameters as in Line 23 of Table \ref{tab:3.9.8}.  
\begin{table}[H]
    \centering
    \tiny
    \caption{\small Table of parameters lemma 3.9 case 8}\label{tab:3.9.8}
    \begin{tabular}{cllllll}
    \toprule
    Line &$q$ &$v$   & $b$    & $r$  & $k$ &$\lambda$\\
    \toprule
    1 &$2^2$&$6   $&$15$   &$10  $&$4  $&$6   $\\
    2 &     &      &$30$   &$20  $&$4  $&$12  $\\
    3 & $2^3$&   $28$&   $252$   & $27$ &$3$ & $2$ \\
    4 & & & $504$   & $54$ &$3$ & $4$ \\
    5 &      &      & $252$ &$54$ &$6$ & $10$\\
    6 &     &      &$189   $&$54  $&$8  $&$14  $\\
    7 &     &      &$126   $&$27  $&$6  $&$5   $\\
    8 &     &      &$126   $&$54  $&$12 $&$22  $\\
    9 &     &      &$63    $&$27  $&$12 $&$11  $\\
    10&     &      &$63    $&$54  $&$24 $&$46  $\\
    11&$2^4$&$120 $&$255   $&$136 $&$64 $&$72  $\\
    12& $2^5$&   $496$&   $20460$   &  $165$   &$4$ & $1$ \\
    13 & & & $40920$   & $330$ &$4$ & $2$ \\  
14&                    &                     &   $2728$   &  $55$  &$10$ & $1$ \\
15 & & & $5456$   & $110$ &$10$ & $2$ \\    
    16&     &      &$8184$& $165$  & $10$ & $3$  \\
    17&     &      &$16368 $&$330 $&$10 $&$6   $\\
    18&     &      &$2046  $&$165 $&$40 $&$13  $\\
    19&     &      &$4092  $&$330 $&$40 $&$26  $\\
    20&     &      &$1023  $&$330 $&$160$&$106 $\\
    21& $2^7$&   $8128$&  $87376$   &  $301$   &$28$ & $1$ \\

22 & & & $174752$  & $602$ &$28$ & $2$ \\
    
    23&                  &      &$262128$&$903 $&$28 $&$3   $\\
    24&     &      &$524256$&$1806$&$28 $&$6   $\\
    \bottomrule
\end{tabular}
\end{table}

If $X_{B}=(Z_{p})^{t}:Z_{m}$ with $m \mid p^{\gcd(t,f)}-1$ and 
$m \leq \frac{p^{f}-1}{\gcd(2,p^{f}-1)}$, then $|X_{B,\alpha}|_{p}$ divides $\gcd(2,p)$,
where $\alpha$ is any point in the block $B$
since $X_{\alpha}\cong D_{\frac{2(q-1)}{\gcd(2,q-1)}}$. Since $X_B\trianglelefteq G_B$ and $G_B$ acts transitively on $B$, we obtain that
\begin{equation}\label{smeagol}
\frac{\left( q-2\right) ex}{2\theta }+1=p^{s}m _{0}u,
\end{equation}
where $\left\vert X_{B}:X_{\alpha,B} \right\vert =p^{s}m _{0}$ with $t+1-\gcd(2,p)\leq s \leq t \leq f$, $m_{0} \mid m$ and $u\mid \gcd(2,p^{f}-1)f$. If $s=0$ then $X_B$ is contained in $Z_{(q-1)/\gcd(2,q-1)}$ and hence in $D_{2(q-1)/(2,q-1)}$,
which has been analyzed in Lemma \ref{MintX7}.
If $m=1$ then $X_{B}=(Z_{p})^{t}$, and this case was settled above in this proof.
Therefore, we may assume that $s\geq 1$ and $m\geq2$.

Assume that $q$ is odd. Then $(e,\theta )=(2,1)$, 
$x<5f$ and $s=t\leq f$.
From (\ref{smeagol}) we have 
\begin{equation}\label{gollum}
(q-2)x+1=p^{s}m_{0}u
\end{equation}
implying $p^{s}\mid 2x-1$.
Let $2x-1=yp^{s}$ for some positive integer $y$.
Then 
\begin{equation}\label{DeathToIDF}
\left( q-2\right) y+\frac{q}{p^{s}} =2m _{0}u
\end{equation}
If $2\leq m \leq \sqrt{q}-1$, then 
from (\ref{DeathToIDF}) we have 
\[
q-1 \leq \left( q-2\right) y+\frac{q}{p^{s}} 
\leq 2mu \leq 2\left( \sqrt{q}-1\right) \cdot 2f
\]
and so $\sqrt{q}+1\leq 4f$, implying 
$q=3^{2},3^{3},3^{4},3^{5},5^{2},7^{2}$.
Combing this with $x<5f$,  from (\ref{gollum}) we have 
$(q,x,s,u,m_{0})=(3^{2},5,2,2,2),(3^{2},5,2,4,1)$.
However, it implies that $k=(q-2)x+1=36=v$, a contradiction.

If $m > \sqrt{q}-1$, then $t=f=s$ since $m \mid p^{\gcd(t,f)}-1$ and $t \leq f$, and so $q\mid 2x-1$.
Combining this with $x<5f$, we have $q=2x-1$ with $q=5,7,9,27$,
forcing $r=\frac{\lambda(q+1)}{2x}=\lambda$, a contradiction.

Finally, assume that $q=2^{f}$, $f\geq 2$. 
Then $(e,\theta)=(1,f)$, $x< 10f^2$ and
\begin{equation}\label{Bilbo}
\left( 2^{f-1}-1\right) x+f =2^{s}m _{0}uf\text{,}
\end{equation}
where $t-1\leq s\leq t\leq f$, $m_{0} \mid m$ and $u \mid f$.

If $x=f$, then $m_{0}=1$, and $s\leq f-1$ and $u=2^{f-1-s}$. Moreover, $k = 2^{f-1}$, $r = (2^f + 1)\lambda$, implying $\lambda \mid 2f$ as $r\mid 2(q+1)f$. Now, $\left\vert X_{\alpha,B} \right\vert=2^{t-s}m$ with $m \mid 2^{t-s}-1$ since $X_{\alpha,B}$ is a Frobenius group. Thus $m\mid 2^{(t-s,f)}-1$, and hence $t=s\leq f-1$ since $m \geq 2$, $t-1\leq s \leq t$ and $s\leq f-1$. Now, the number of $X_{B}$-orbits in $B$ is $2^{f-1-t}$, and $X_{\alpha,B}\cong Z_{m}$, with $m\mid 2^{(t,f)}-1$, fixes a point in each of them. If $t<f-1$, then there is a point $\sigma$ of $\mathcal{D}$, with $\sigma \neq \alpha$, such that $Z_{m} \cong X_{\alpha,B}\leq X_{\alpha, \sigma}$. Then $m\mid 2(2^{f}+1)$ since $ X_{\alpha, \sigma} \cong D_{2(2^{f}+1)}$, and hence $m=1$, a contradiction. Thus $s=t=f-1$, and again $m=1$ since $m\mid 2^{(t,f)}-1$, which is not the case. 

If $x> f$, then $2^{s-1}\mid x-f$ and hence $x=2^{s-1}y+f$ for some positive integer $y$.
From (\ref{Bilbo}) we obtain that 
\[
(2^{f-1}-1)y+2^{f-s}f =2m_{0}uf.
\]
If $m_{0} \leq 2^{f/2}-1$, then 
\[
(2^{f-1}-1)+2\leq 2\left( 2^{f/2}-1\right) u\cdot f\leq 2\left( 2^{f/2}-1\right)f^{2},
\]
which implies $2\leq f \leq 21$. 
Hence, we obtain Lines 1--2, 6, 8--11 and 20 of Table \ref{tab:3.9.8}.


If $m_{0} > 2^{f/2}-1$, then $ 2^{f/2}-1<m\leq 2^{\gcd(t,f)}-1$ and hence $t=f$. Moreover, $X_{B}=(Z_{2})^{f}:Z_{m}$ and so
$b\mid \left( 2^{f}+1\right)
\cdot \frac{2^{f}-1}{m} \cdot f$,
and by Fisher's inequality one has 
\begin{equation}\label{orodruin}
2^{f-1}(2^{f}-1)\leq \left( 2^{f}+1\right)
\cdot \frac{2^{f}-1}{m}\cdot f < \left(2^{f}+1\right)
\cdot \left(2^{f/2}+1\right) \cdot f\text{.}
\end{equation}
Therefore $2\leq f\leq 8$. This case leads to no admissible parameters. 

Finally, assume that $1\leq x < f$. Then $x =f - 2^{s-1}y$ by (\ref{Bilbo}) for some positive integer $y$, and hence $1 \leq x\leq f-2^{s-1}$. Therefore, one has $2^{t-1}\leq 2^{s}<2f$ since $t-1\leq s\leq t$. Moreover, $m_{0}\leq m\leq 2^{\gcd(t,f)}-1\leq 2^{t}-1<4f$. Then (\ref{Bilbo}) implies 
$$2^{f-1}-1\leq (2^{f-1}-1)x=2^{s}\cdot m_{0}\cdot u\cdot f-f<2f\cdot 4f \cdot f \cdot f-f=8f^{4}-f\text{,}$$
since $u \mid f$, and then $2\leq f\leq 21$. However, there is no such $x$ satisfying (\ref{Bilbo}).
This completes the proof.

\end{proof}

\begin{corollary}\label{MintX9Cor} In case (1) of Lemma \ref{MintX9}, if $q<2^{8}$ then one of the following holds:
\begin{enumerate}
    \item $\mathcal{D}$ is the Witt-Bose-Shrikhande linear space of order $2^{3}$;
    \item $\mathcal{D}$ is the Witt-Bose-Shrikhande linear space of order $2^{4}$;
    \item $\mathcal{D}$ is the $2$-$(120,8,4)$ design and $G\cong \PSigL(2,2^{4})$ as in Line 41 of Table \ref{tab:alldesigns};
    \item $\mathcal{D}$ is the $2$-$(120,8,4)$ design and $G\cong \PGaL(2,2^{4})$ as in Line 42 of Table \ref{tab:alldesigns};
    \item $\mathcal{D}$ is the $2$-$(120,8,8)$ design and $G\cong \PGaL(2,2^{4})$ as in Line 43 of Table \ref{tab:alldesigns};
    \item $\mathcal{D}$ is the Witt-Bose-Shrikhande linear space of order $2^{5}$;
    \item $\mathcal{D}$ is the Witt-Bose-Shrikhande linear space of order $2^{6}$;
    \item $\mathcal{D}$ is the $2$-$(2016,32,4)$ as in Line 46 of Table \ref{tab:alldesigns}, and $G\cong \PSL(2,2^{6}):Z_{2}$;
    \item $\mathcal{D}$ is the $2$-$(2016,32,4)$ as in Line 47 of Table \ref{tab:alldesigns}, and $G\cong \PGaL(2,2^{6})$;
    \item $\mathcal{D}$ is one of the $2$-$(2016,32,12)$ designs as in Lines 48, 49 of Table \ref{tab:alldesigns}, and $G\cong \PGaL(2,2^{6})$;
     \item $\mathcal{D}$ is the $2$-$(2016,32,12)$ and $G\cong \PGaL(2,2^{6})$ as in Line X of Table \ref{tab:alldesigns};
\end{enumerate}    
\end{corollary}
\begin{proof}
The result is simply the output of an exhaustive search of the $2$-designs with parameters case (1) of Lemma \ref{MintX7} and $\PSL(2,q)\leq G \leq \PGaL(2,q)$ for $q=2^{f}$ with $f<8$.
\end{proof}

\bigskip

\begin{proof}[Proof of the Theorem \ref{FT-PP-reduction}]
Let $\mathcal{D}$ be a $2$-$(v,k,\lambda)$ design admitting a flag-transitive automorphism group $G$ with socle $X \cong \PSL(2,q)$. Then $X$ acts point-primitively on $\mathcal{D}$ by Lemma \ref{lem:NotNov}. Therefore, $X_{\alpha}$ is one of the groups listed in Table \ref{MaxDickson}, and the corresponding $\mathcal{D}$ and $G$ are determined in Lemmas \ref{case1}-\ref{MintX9} and Corollaries \ref{MintX7Cor} and \ref{MintX9Cor}; these are exactly those reported in the Theorem statement.    
\end{proof}

\section{The case $G$ acting point-imprimitively on $\mathcal{D}$}\label{FT-PI-Section}
In this section we suppose that any flag-transitive automorphism $G$ of $\mathcal{D}=\left(\mathcal{P},\mathcal{B}\right)$ acts point-imprimitively on $\mathcal{D}$, that is $G$ preserves a partition $\Sigma $ of $\mathcal{P}$ with $v_{1}>1$ classes each of size $v_{0}>1$. In this context, our aim is to completely determine $\left(\mathcal{D},G\right)$. More precisely we prove the following result:

\bigskip
\begin{theorem}\label{FTPI}
Let $\mathcal{D}$ be a $2$-$(v,k,\lambda)$ design admitting $PSL(2,q) \unlhd G\leq \PGaL(2,q)$ as a flag-transitive and point-imprimitive automorphism group. Then one of the following holds:
\begin{enumerate}
\item $\mathcal{D}$ is the $2$-$(15,8,4)$ design complementary to $\PG(3,2)$, and $ G\cong
\PGL(2,5)$;
\item $\mathcal{D}$ is the $2$-$(36,8,4)$ design as in \cite[Construction 9]{DP}, and $G\cong \PSigL(2,9)$;
\item $\mathcal{D}$ is the symmetric $2$-$(85,64,48)$ design complementary to $\PG(3,4)$ and $G\cong \PGaL (2,2^4)$.
\end{enumerate}
\end{theorem}

\bigskip 
Before proceeding, let us recall the following fact. Let $\alpha$ be any point of $\mathcal{D}$, and let $\Delta$ be the element of the $G$-invariant partition $\Sigma$ of $\mathcal{P}$ containing $x$ \textcolor{green}{$\alpha$}, then $G_{\alpha}<G_{\Delta}<G$. The converse is also true by \cite[Theorem 1.5A]{DM}. Therefore, in the sequel we say that $\Sigma$ is \emph{maximal} if and only if $G_{\Delta}$ is a maximal subgroup of $G$.

Assume that $G$ preserves a further non-trivial partition $\Sigma^{\prime}$ of the point-set of $\mathcal{D}$ in $v_{1}^{\prime}$ classes of size $v_{0}^{\prime}$. We may apply Theorem \ref{CamZie} to the triple $(\mathcal{D},G,\Sigma^{\prime})$ deriving information on the incidence structures $\mathcal{D}_{0}^{\prime}$ and $\mathcal{D}_{1}^{\prime}$ which have the same meaning as $\mathcal{D}_{0}$ and $\mathcal{D}_{1}$, respectively, but they are related to the partition $\Sigma^{\prime}$. In this setting, we say that $\Sigma^{\prime}$ \emph{refines} $\Sigma$, or that $\Sigma^{\prime}$ is \emph{finer} than $\Sigma$, and write $\Sigma^{\prime} \leq \Sigma$, if any element of $\Sigma^{\prime}$ is contained in a (unique) element of $\Sigma$. The binary relation '$\leq$' on the set of the $G$-invariant non-trivial partitions of the point-set of $\mathcal{D}$ is a partial ordering. Therefore, it makes sense to call any such partition \emph{maximal} if it is maximal with respect to ordering '$\leq$'. Clearly, this notion of maximality for a partition are equivalent to that given above. It is worth noting that, for any given non-trivial $G$-invariant partition $\Sigma$ of the point-set of $\mathcal{D}$, there always is a maximal non-trivial $G$-invariant partition $\Sigma^{\prime}$ of the point-set of $\mathcal{D}$ refining $\Sigma$.

\bigskip

A landmark for the determinations of the pairs $\left(\mathcal{D},G\right)$ with $G$ acting flag-transitive and point-imprimitively on $\mathcal{D}$ is the following key result proved by Camina and Zieschang \cite{CZ} in 1990. 

\bigskip

\begin{theorem}[Camina-Zieschang]\label{CamZie}
Let $\mathcal{D}=(\mathcal{P}, \mathcal{B})$ be a $2$-$(v,k,\lambda)$ design admitting a
flag-transitive, point-imprimitive automorphism group $G$ preserving a nontrivial 
partition $\Sigma $ of $\mathcal{P}$ with $v_1$ classes of size $v_0$. Then $v=v_{0}v_{1}$ and the following hold:
\begin{enumerate}
    \item[(1)] There is a constant $k_0 \geq 2$ such that $\left\vert B\cap \Delta \right\vert =0$ or $%
k_0 $ for each $B\in \mathcal{B}$ and $\Delta \in \Sigma $. The parameter $k_0 $ divides $k$. Moreover,
\begin{equation}\label{rel1}
\frac{v-1}{k-1}=\frac{v_{0}-1}{k_{0}-1}\text{,}    
\end{equation}
and the following hold:
\begin{enumerate}
\item[(i)] For each $\Delta \in \Sigma $, let $\mathcal{B}_{\Delta }=\left\{ B\cap \Delta
\neq \varnothing :B\in \mathcal{B}\right\} $. Then the set $$\mathcal{R}_{\Delta}=\{(B,C) \in \mathcal{B}_{\Delta } \times \mathcal{B}_{\Delta }:  B\cap \Delta =\C\cap \Delta \}$$ is an equivalence relation on $\mathcal{B}_{\Delta }$ with each equivalence class of size $\theta$.
\item[(ii)] The incidence structure $\mathcal{D}_{\Delta }=(\Delta ,%
\mathcal{B}_{\Delta })$ is either a symmetric $1$-design with $k_{0}=v_{0}-1$, or a $2$-$(v_0,k_0 ,\lambda
_{0})$ design with $\lambda_{0}=\frac{\lambda}{\theta}$.
\item[(iii)] The group $G_{\Delta}^{\Delta}$ acts flag-transitively on $\mathcal{D}_{\Delta }$.
\end{enumerate}
\item[(2)] For each block $B$ of $\mathcal{D}$ the set $B(\Sigma)=\{\Delta \in \Sigma: B \cap \Delta \neq \varnothing\}$ has a constant size $k_{1}=\frac{k}{k_{0}}$. Moreover, 
\begin{equation}\label{rel2}
\frac{v_{1}-1}{k_{1}-1}=\frac{k_{0}(v_{0}-1)}{v_{0}(k_{0}-1)}    
\end{equation}
and the following hold:
\begin{enumerate}
\item[(i)] The set $$\mathcal{R}=\{(C,C^{\prime}) \in \mathcal{B} \times \mathcal{B}:C(\Sigma)=C^{\prime}(\Sigma)\}$$ is an equivalence relation on $\mathcal{B}$ with each class of size $\mu$;
\item[(ii)] Let $\mathcal{B}^{\Sigma}$ be the quotient set defined by $\mathcal{R}$, and for any block $C$ of $\mathcal{D}$ denote by $C^{\Sigma}$ the $\mathcal{R}$-equivalence class containing $C$. Then the incidence structure $\mathcal{D}^{\Sigma}=\left(\Sigma, \mathcal{B}^{\Sigma}, \mathcal{I} \right)$ with $\mathcal{I}=\{(\Delta,C^{\Sigma}) \in \Sigma \times \mathcal{B}^{\Sigma}: \Delta \in C(\Sigma) \}$ is either a symmetric $1$-design with $k_{1}=v_{1}-1$, or a $2$-$(v_{1},k_{1} ,\lambda_{1})$ design with $\lambda_{1}=\frac{v_{0}^{2}\lambda}{k_{0}^{2}\mu}$.
\item[(iii)] The group $G^{\Sigma}$ acts flag-transitively on $\mathcal{D}^{\Sigma}$.
\end{enumerate}
\end{enumerate}
\end{theorem}

\bigskip
For a proof see \cite[Propositions 2.1 and 2.3]{CZ}.

\bigskip

Throughout the remainder of the paper the symbols $\Sigma$ and $\Delta$ will the same meaning as in Theorem \ref{CamZie}.

We refer to $\mathcal{D}^{\Sigma}$ simply as $\mathcal{D}_{1}$. Moreover, as it is mentioned in the introduction, the designs corresponding to distinct classes $\Delta,\Delta'\in\Sigma$ are isomorphic under elements of $G$ mapping $\Delta$ to $\Delta'$, hence we will refer to $\mathcal{D}_{\Delta }$ as $\mathcal{D}_{0}$. The parameters of $\mathcal{D}_{0}$ and of $\mathcal{D}_{1}$ will be indexed by $0$ and $1$, respectively. Hence, $v_{i}r_{i}=b_{i}k_{i}$, $\lambda (v_{i}-1)=r_{i}(k-1)$, and $k_{i} \leq r_{i}$, where $r_{i}$ and $b_{i}$ are the replication number and the number of blocks of $\mathcal{D}_{i}$, when $\mathcal{D}_{i}$ is a $2$-$(v_{i},k_{i},\lambda_{i})$ design.

\bigskip

Recall that $X$ denotes $\PSL(2,q)$, the socle of $G$. We need a series of preliminary results.

\begin{lemma}
The following hold:
\begin{enumerate}
\item $G$ acts faithfully on $\Sigma $;
\item $X$ acts point-transitively on $\mathcal{D}$;
\end{enumerate}
\end{lemma}

\begin{proof}
Since $X$ is the unique minimal normal subgroup of $G$, and $G_{(\Sigma )}\unlhd G$, either $G_{(\Sigma )}=1$ or $X\trianglelefteq G_{(\Sigma )}$. Assume that the latter occurs. Then $G^{\Sigma }\leq
Z_{(2,q-1)}\times Z_{\log _{p}(q)}$ by \cite[Proposition 2.2.3]{KL}. Then $G^{\Sigma }$ is abelian, and
hence $G^{\Sigma }$ is regular since $G^{\Sigma }$ acts point-transitively on $%
\mathcal{D}_{1}$, however this is impossible since $r_{1}\mid \left\vert
G_{\Delta }^{\Sigma }\right\vert $ with $r_{1}>1$, where $\Delta \in \Sigma $. Thus $G_{(\Sigma )}=1$, which is the assertion (1). 

The group $X$ acts point-transitively on $\mathcal{D}$, otherwise the set of $X$-orbits on the point
set of $\mathcal{D}$ is $G$-invariant partition for which $X$ lies in the
kernel, contrary to (1). Thus, $X$ acts point-transitively on $\mathcal{%
D}$, and hence $X$ acts transitively on $\Sigma $, which is the assertion (2).
\end{proof}

\bigskip
The employed method is as follows. The above argument shows that we may always chose $\Sigma$ to be maximal. Therefore $G_{\Delta}$, with $\Delta \in \Sigma$, is a maximal subgroup of $G$. Since $\PSL(2,q) \unlhd G \leq \PGaL(2,q)$, and $G$ acts flag-transitively and point-primtively on $\mathcal{D}_{1}$. Since $\PSL(2,q) \unlhd G \leq \PGaL(2,q)$, then either $k_{1}=2$ or $v_{1}-1$, and hence $G$ acts $2$-transitively on $\mathcal{D}_{1}$, an so $v_{1}$ is known; or $2<k_{1}<v_{1}-1$ and the pair $(\mathcal{D}_{1},G)$ is as in Theorem \ref{FT-PP-reduction}, and again the pair $(v_{1},k_{1})$ is known. Hence, $(v_{1},k_{1})$ is determined in each case, and consequently the structure of $G_{\Delta}$ is known too. Then, in order to gain information $(v_{0},k_{0})$, we use (\ref{rel1}) or (\ref{rel2}) and the fact that $v_{0}r_{0} \mid \left\vert G_{\Delta}^{\Delta}\right\vert$ and $G_{\Delta}^{\Delta}$ has a transitive permutation representation of degree $v_{0}$, where $G_{\Delta}^{\Delta}$ is a quotient group of $G_{\Delta}$, since $G_{\Delta}^{\Delta}$ acts flag-transitively on $\mathcal{D}_{0}$. Thus, apart from two genuine numerical known examples, the group $G$ acts point-$2$-transitively on the $2$-$(q+1,k_{1},\lambda_{1})$ design $\mathcal{D}$. Now we use the structure of the group together with some results contained in \cite{CZ} in order to show that the admissible $k_{1}$ leads to no examples of $\mathcal{D}$.
\bigskip

\begin{lemma}\label{tris}
If $\Sigma $ is any maximal $G$-invariant partition, then $G$ acts
point-primitively on $\mathcal{D}_{1}$ and one of the following holds:
\begin{enumerate}
\item $\mathcal{D}$ is the $2$-$(15,8,4)$ design complementary to $\PG(3,2)$ and $ G \cong \PGL (2,5)$;
\item $\mathcal{D}$ is the $2$-$(36,8,4)$ design of \cite[Construction 9]{DP}, and $G\cong \PSigL(2,9)$;
\item $G$ acts point-$2$-transitively on $\mathcal{D}_{1}$ and $v_{1}=q+1$.
\end{enumerate}
\end{lemma}
\begin{proof}
The maximality of $\Sigma $ implies that $G$ acts point-primitively on $\mathcal{%
D}_{1}$ by \cite[Theorem 1.5A and Corollary 1.5A]{DM}.

If $k_{1}=2$ or $v_{1}-1$, then $G$ acts $2$-transitively on $\mathcal{D}_{1}$. Then either $v_{1}=q+1$ and assertion (3) follows, or one of the following holds:
\begin{enumerate}
    \item[(i)] $v_{1}=5$, $k_{1}=2,4$ and $\PSL(2,5) \unlhd G \leq \PGL (2,5)$;
    \item[(ii)] $v_{1}=6$, $k_{1}=2,5$ and $\PSL(2,9) \unlhd G \leq \PGaL (2,9)$;
    \item[(iii)] $v_{1}=7$, $k_{1}=2,6$ and $G \cong \PSL (2,7)$;
    \item[(iv)] $v_{1}=11$, $k_{1}=2,10$ and $G\cong PSL(2,11)$. 
\end{enumerate}    
Assume that (iv) occurs. Then $v=11v_{0}$ and $G_{\Delta}\cong A_{5}$. Then $G_{\Delta}^{\Delta}\cong A_{5}$ since $G_{\Delta}^{\Delta}$ acts flag-transitively on $\mathcal{D}_{0}$. If $k_{1}=2$, then $\frac{k_{0}(v_{0}-1)}{v_{0}(k_{0}-1)}=10$ by (\ref{rel2}) with $v_{0}r_{0} \mid \left\vert G_{\Delta}^{\Delta}\right\vert=60$ and $G_{\Delta}^{\Delta}\cong A_{5}$ admitting a transitive permutation representation of degree $v_{0}$. However, no admissible $(v_{0},k_{0})$ arise. The case $k_{1}=10$ gives admissible parameters $(v_0,b_0,r_0,k_0,\lambda_0)=(6, 15, 10, 4, 6)$, hence $v=66$ and $k=40$. However, we see that there are no proper subgroups of $G=\PSL(2,11)$ of order divisible by $40$ by \cite{At}. A similar argument rules out (iii) as well and leads to the following admissible cases: $(v_{1},k_{1},v_{0},k_{0})=(5,4,3,2)$ in (i) and $(6,5,6,3),(6,5,16,4)$ in (ii). In the latter case $A_{5}\leq G_{\Delta}^{\Delta}\leq S_{5}$, which is not divisible by $16$, whereas it should since $G_{\Delta}^{\Delta}$ acts transitively on the $v_{0}=16$ points of $\mathcal{D}_{0}$ by Theorem \ref{CamZie}(1.iii). By using \texttt{GAP}, we see that the remaining cases lead to assertions (1) and (2).
    
If $2<k_{1}<v_{1}-1$, then $\mathcal{D}_{1}$ is a non-trivial $2$-design admitting $G$ as a flag-transitive point-primitive automorphism group, and hence $\mathcal{D}_{1}$ and $G$ are as in Theorem \ref{FT-PP-reduction}. However, it is a routine exrcise to check that no admissible $(v_{1},k_{1},v_{0},k_{0})$ arise from (\ref{rel1}) when $\mathcal{D}_{1}$ and $G$ are as in Table \ref{tab:alldesigns}. Therefore, either the parameters of $\mathcal{D}_{1}$ are those as in the infinte families (I), (II) or (III), or $G$ acts $2$-transitively on $\mathcal{D}_{1}$, $v_{1}=q+1$, which is the assertion (3).

Assume that the the parameters of $\mathcal{D}_{1}$ are those of as in the family (I). Then $$\left(v_{1},b_{1},k_{1},r_{1}\right)=\left(2^{f-1}(2^{f}-1),2^{f-1}3^{2}(2^{f}-1)\frac{\lambda_1 }{2},\frac{2^{f}+1}{3},3(2^{f}+1)\frac{\lambda_1 }{2}\right)\text{,}$$ with $f$ odd and $\lambda_1$ even, and hence
$$\frac{k_{0}(v_{0}-1)}{v_{0}(k_{0}-1)} =\frac{3(2^{f}+1)}{2}$$ 
by (\ref{rel2}). Then
\begin{equation}\label{endJune}
(2^{f}3+1)v_{0}k_{0}=3(2^{f}+1)v_{0}-2k_{0}\text{,}    
\end{equation}
forcing $v_{0}=2k_{0}$ since $v_{0}>k_{0}$. Indeed, $v_{0}=k_{0}$ implies $v=k$ by (\ref{rel1}), a contradiction. Substituting $v_{0}=2k_{0}$, we get $k_{0}=\frac{2^{f}3+2}{2^{f}3+1}  \notin \mathbb{Z}$, and hence the parameter of $\mathcal{D}_{1}$ cannot be those as in (I) of Theorem \ref{FT-PP-reduction}.

Assume that the the parameters of $\mathcal{D}_{1}$ are those of as in the family (II).  Then
$$\left(2^{f-1}(2^{f}-1),(2^{2f}-1)\lambda_1,2^{f-1},(2^{f}+1)\lambda_1\right),$$ with $\lambda_1 \mid 2f$, and hence  
$$\frac{k_{0}(v_{0}-1)}{v_{0}(k_{0}-1)} =2^{f}+1$$ 
by (\ref{rel2}), which implies $2^{f}v_{0}k_{0}=(2^{f}+1)v_{0}-k_{0}$. Then $v_{0}\mid k_{0}$, and hence $v_{0}=k_{0}$ since $k_{0}\leq v_{0}$, which implies $v=k$ by (\ref{rel1}), a contradiction. 

Assume that the the parameters of $\mathcal{D}_{1}$ are those of as in the family (III). Then
$$\left(v_{1},b_{1},k_{1},r_{1}\right)=\left(\frac{q(q-1)}{2},\frac{q(q+1)\lambda_1}{4},q-1,\frac{(q+1)\lambda_1}{2}\right)$$ with $\lambda_1 \mid 4f$. Then
\begin{equation}\label{endJuneIII}
\frac{k_{0}(v_{0}-1)}{v_{0}(k_{0}-1)} =\frac{q+1}{2}
\end{equation}
by (\ref{rel2}), and arguing as above $v_{0}=2k_{0}$, which substituted in (\ref{endJuneIII}) leads to $(q-1)k_{0}=q$. Then $k_{0}=q=2$ and $v_{0}=4$, and so $v_1=k_1=1$, a contradiction. This completes the proof.
\end{proof}

\bigskip

\begin{lemma}
\label{SigmaMax1}If $\Sigma $ is any maximal $G$-invariant partition, then

\begin{enumerate}
\item $v_{0}=p^{a}\theta +1$ with $\theta =\left( k_{0}-1,v_{0}-1\right) $
and $1\leq a\leq f$.

\item $r_{0}=\frac{p^{a}\lambda _{0}}{\left( k_{0}-1\right) /\theta }$;

\item $k_{0}\mid v_{0}\left( p^{f-a},\frac{p^{a}\lambda _{0}}{\left(
k_{0}-1\right) /\theta }\right) $;

\item $k_{1}=\frac{k_{0}-1}{\theta }\cdot \frac{p^{f-a}v_{0}}{k_{0}}+1$ and $%
k_{1}>\frac{p^{f}(k_{0}-1)}{k_{0}}+1$.
\end{enumerate}
\end{lemma}

\begin{proof}
By (\ref{rel1}), we have%
\[
\frac{v_{0}(q+1)-1}{k_{0}k_{1}-1}=\frac{v_{0}-1}{k_{0}-1}\text{.} 
\]%
Set $\theta =(k_{0}-1,v_{0}-1)$, then 
\begin{equation}\label{monsone}
\frac{k_{0}-1}{\theta }(v_{0}q+v_{0}-1)=\frac{v_{0}-1}{\theta }%
(k_{0}k_{1}-1) 
\end{equation}
and so $\frac{v_{0}-1}{\theta }\mid v_{0}q+v_{0}-1$. Then $\frac{v_{0}-1}{%
\theta }\mid v_{0}q$ and hence $\frac{v_{0}-1}{\theta }=p^{a}$ for some $0\leq a\leq f$. If $a=0$, then $v_{0}-1=\theta =(k_{0}-1,v_{0}-1)$ and so $%
v_{0}=k_{0}$, which implies $v=k$, a contradiction. Thus $v_{0}=p^{a}\theta
+1$ with $1\leq a\leq f$ and $r_{0}=\frac{(v_{0}-1)\lambda_{0}}{k_{0}-1}=\frac{p^{a}\lambda _{0}}{\left( k_{0}-1\right) /\theta }$, which are (1) and (2) respectively.

Now, substituting (1) in (\ref{monsone}), we get
\begin{equation*}
\frac{k_{0}-1}{\theta }\cdot \left( p^{a+f}\theta +p^{a}\theta +p^{f}\right)
=p^{a}(k_{0}k_{1}-1)
\end{equation*}%
and so%
\[
k=k_{0}k_{1}=\frac{k_{0}-1}{\theta }\cdot \left( p^{f}\theta
+p^{f-a}\right) +k_{0} 
\]%
and $k_{0}\mid p^{f-a}v_{0}$. 
\[
k_{0}\mid p^{f-a}v_{0}\text{ and }k_{1}=\frac{k_{0}-1}{\theta }\cdot \frac{%
p^{f-a}v_{0}}{k_{0}}+1>\frac{p^{f}(k_{0}-1)}{k_{0}}+1\text{,} 
\]%
which is (4). 
Finally, since $k_{0}\mid v_{0}r_{0}= v_{0}\frac{p^{a}\lambda _{0}}{\left( k_{0}-1\right)
/\theta }$, it follows that $k_{0}\mid v_{0}\left( p^{f-a},\frac{%
p^{a}\lambda _{0}}{\left( k_{0}-1\right) /\theta }\right)$, which is (3).

\end{proof}

\begin{lemma}
\label{OndaDobro}If $\Sigma $ is any maximal $G$-invariant partition, and $\Delta \in \Sigma$, then the following hold:
\begin{enumerate}
    \item $G_{\Delta }^{\Delta }=Z_{v_{0}}:Z_{d}$, with $v_{0}=\frac{p^{f}-1}{%
\left\vert X_{(\Delta )}\right\vert _{p^{\prime }}(2,p^{f}-1)}$ and $d\mid f$;
  \item  $v_{0}=p^{a}\theta +1$ with $\theta =\left( k_{0}-1,v_{0}-1\right) $
and $1\leq a <f$.
\end{enumerate}
\end{lemma}

\begin{proof}
The group $X$ acts $2$-transitively on the $v_{1}=q+1$ points of $\mathcal{D}_{1}$, hence $X_{\Delta }\cong Z^{f}_{p}:Z_{\frac{p^{f}-1}{(2,p^{f}-1)}}$. Furthermore, $X_{\Delta }\trianglelefteq G_{\Delta }\leq A\Gamma L(1,p^{f})$. It follows from Lemma \ref{SigmaMax1}(1) that $(v_{0},p)=1$, hence $(Z_{p})^{f}\leq X_{x}$ for some $x\in \Delta $. Then $(Z_{p})^{f}\leq X_{(\Delta )}$ since $(Z_{p})^{f}\trianglelefteq G_{\Delta }$ and $G_{\Delta }$ acts
transitively on $\Delta $.

If $X_{(\Delta )}=Z_{p}^{f}:Z_{\frac{p^{f}-1}{(2,p^{f}-1)}}$, then $%
G_{\Delta }^{\Delta }$ is isomorphic to a subgroup of $Z_{f}$. Then $%
G_{\Delta }^{\Delta }$ is abelian, and hence $G_{\Delta }^{\Delta }$ is
regular since $G_{\Delta }^{\Delta }$ acts transitively on $\Delta $,
however this is impossible since $r_{0}\mid \left\vert G_{x}^{\Delta
}\right\vert $ with $r_{0}>1$. Thus $Z_{p}^{f}\leq X_{(\Delta )}<
(Z_{p})^{f}:Z_{\frac{p^{f}-1}{(2,p^{f}-1)}}$, and hence $X_{\Delta }^{\Delta
}\cong Z_{c}$, with $c>1$ and $c\left\vert X_{(\Delta )}\right\vert _{p^{\prime }}=%
\frac{p^{f}-1}{(2,p^{f}-1)}$, acts point-semiregularly on $%
\mathcal{D}_{0}$. Since $G_{\Delta }^{\Delta }$ acts transitively on $\Delta 
$ and $1\neq X_{\Delta }^{\Delta }\trianglelefteq G_{\Delta }^{\Delta }$, it
follows that $c\mid v_{0}$. Moreover, $G_{\Delta }^{\Delta }/X_{\Delta
}^{\Delta }\cong Z_{d}$, with $d\mid f$, permutes transitively the $\frac{v_{0}%
}{c}$ orbits on $\Delta $ under $X_{\Delta }^{\Delta }$. Since $G_{\Delta
}^{\Delta }$ acts flag-transitively on $\mathcal{D}_{0}$, it follows that $%
r_{0}\mid \left\vert G_{x}^{\Delta }\right\vert $. On the other hand, $%
G_{x}^{\Delta }\cap X_{\Delta }^{\Delta }=1$ since $X_{\Delta }^{\Delta }$
acts semiregularly on $\Delta $. Therefore, $G_{x}^{\Delta }$ is isomorphic
to a subgroup of $Z_{d}$, and $r_{0}\mid d$.

Note that $X_{\Delta
}^{\Delta }:G_{x}^{\Delta }\trianglelefteq G_{\Delta }^{\Delta }$ since $G_{\Delta }^{\Delta }/X_{\Delta }^{\Delta }\cong Z_{d}$. Then $\Delta $ is a
union of $(X_{\Delta }^{\Delta }:G_{x}^{\Delta })$-orbits of equal lengths.
On the other hand, $x^{X_{\Delta }^{\Delta }}$ is both a $X_{\Delta
}^{\Delta }$-orbit and a $(X_{\Delta }^{\Delta }:G_{x}^{\Delta })$-orbit.
Thus, each $X_{\Delta }^{\Delta }$-orbit on $\Delta $ has length $c$ and is a $(X_{\Delta
}^{\Delta }:G_{x}^{\Delta })$-orbit. In particular, the $x^{X_{\Delta
}^{\Delta }}\setminus \{x\}$ is a union of $G_{x}^{\Delta }$-orbits. If $C$
is any block of $\mathcal{D}_{0}$ containing $x$, then 
\begin{equation*}
\frac{\left( v_{0}-1\right) \lambda _{0}}{k_{0}-1}\left\vert C\cap
(x^{X_{\Delta }^{\Delta }}\setminus \{x\})\right\vert =\left\vert
x^{X_{\Delta }^{\Delta }}\setminus \{x\}\right\vert \lambda _{0}\text{,}
\end{equation*}
by Lemma \ref{lem:basic-params}(iii) with $\mathcal{D}_{0}$ and $G_{\Delta }^{\Delta }$ in the role of $\mathcal{D}$ and $G$, respectively. So, we obtain 
$$\left( v_{0}-1\right) \left\vert C\cap (x^{X_{\Delta }^{\Delta }}\setminus
\{x\})\right\vert =(c-1)(k_{0}-1)\text{,}$$
which implies $v_{0}-1 \mid (c-1)(k_{0}-1)$.

If $G_{\Delta }^{\Delta }\neq X_{\Delta }^{\Delta
}:G_{x}^{\Delta }$, then there is $y\in \Delta \setminus x^{X_{\Delta
}^{\Delta }}$, and hence $y^{X_{\Delta }^{\Delta }}$ is a $(X_{\Delta
}^{\Delta }:G_{x}^{\Delta })$-orbit. Then $y^{X_{\Delta }^{\Delta }}$ is a
union $G_{x}^{\Delta }$-orbits and so%
\begin{equation*}
\frac{\left( v_{0}-1\right) \lambda _{0}}{k_{0}-1}\left\vert C\cap
y^{X_{\Delta }^{\Delta }}\right\vert =\left\vert y^{X_{\Delta }^{\Delta
}}\right\vert \lambda _{0}\text{,}
\end{equation*}%
by Lemma \ref{lem:basic-params}(iii), which implies $v_{0}-1 \mid c(k_{0}-1)$. Therefore $v_{0}-1\mid k_{0}-1$, since we have seen that $v_{0}-1 \mid (c-1)(k_{0}-1)$, which is a contradiction. Thus $%
G_{\Delta }^{\Delta }=X_{\Delta }^{\Delta }:G_{x}^{\Delta }$ and $X_{\Delta
}^{\Delta }$ acts regularly on $\Delta $. In particular $c=v_{0}$ and $%
\left\vert G_{x}^{\Delta }\right\vert \mid f$. Hence, $G_{\Delta }^{\Delta
}=Z_{v_{0}}:Z_{d}$ \ with $d\mid f$, which is the assertion (1).

Thus $v_{0}=\left\vert G_{\Delta }:G_{x}\right\vert=\left\vert X_{\Delta }:X_{(\Delta )}\right\vert $, and hence $v_{0}=\frac{p^{f}-1}{\left\vert X_{(\Delta)}\right\vert _{p^{\prime }}(2,p^{f}-1)}$. Then $a<f$ since $v_{0}=p^{a}\theta+1$ with $\theta =\left( k_{0}-1,v_{0}-1\right)$ and $1\leq a \leq f$ by Lemma \ref{SigmaMax1}(1), and we obtain the asserton (2). 
\end{proof}

\bigskip

\begin{lemma}
\label{Uredu}If $\Sigma $ is any maximal $G$-invariant partition, one of the following holds:
\begin{enumerate}
\item $k_{0}=2$, $\mathcal{D}$ is the $2$-$(15,8,4)$ design complementary to $\PG(3,2)$ and $ G \cong \PGL (2,5)$;
\item $k_{0}>2$, and one of the following holds:
\begin{enumerate}
    \item $\mathcal{D}$ is the symmetric $2$-$(85,64,48)$ design complementary to $\PG(3,4)$ and $G\cong \PGaL (2,2^4)$;
    \item $p\mid \frac{p^{f-a}v_{0}}{k_{0}}$, $%
    (k_{1},p)=1$ and $k_{1}>\frac{2}{3}p^{f}+1$, and one of the following holds:
    \begin{enumerate}
    \item[(i)] $p \nmid \frac{k_{0}-1}{\theta }$ and $p^{a}\mid d$;
    \item[(ii)] $p \mid \frac{k_{0}-1}{\theta }$ and $k_{0}\mid v_{0}$.
    \end{enumerate}
\end{enumerate}

\end{enumerate}
\end{lemma}

\begin{proof}
Assume that $k_{0}=2$, then $\theta =1$ and $v_{0}=p^{a}+1$. Moreover, the flag-transitvity of $G_{\Delta }^{\Delta }=X_{\Delta }^{\Delta }:G_{x}^{\Delta }$ 
on $\mathcal{D}_{0}$ by Theorem \ref{CamZie} (1.iii) implies that
$G_{\Delta }^{\Delta }\cong Z_{v_{0}}:Z_{d}$ 
acts point-$2$-transitively on $%
\Delta $. Thus $v_{0}=p^{a}+1$ is a prime by \cite[Theorem 4.1B]{DM}, and hence $p=2$ and $a=2^{j}$, that is to say $v_{0}$ is a Fermat prime. Thus $v_{0}=2^{2^{j}}+1$ and
$$k_{1}=2^{f-2^{j}-1}(2^{2^{j}}+1)+1\text{.}$$
Again by the $2$-transitivity $G_{\Delta }^{\Delta }$ on $\Delta$, it results that $v_{0}-1=2^{2^{j}}$ divides $d$, and hence $f$, since $d$ divides $f$. Then $f$ is even and either $f>2^{j}+1$, or $j=0$ and $f=2$.

Assume that $f> 2^{j}+1$. Then $f\geq 4$ since $f$ is even, and $k_{1}$ is odd. Since $G$ acts flag-transitively on the $2$-$(2^{f}+1,k_{1},\lambda_{1})$ design $\mathcal{D}_1$, then $k_{1}\mid \left\vert G_{B}\right\vert $, where $B$ is any block of $\mathcal{D}_1$. Hence $k_{1}\mid \left\vert M \right\vert \cdot f$ with $M$ a maximal subgroup of $X$ containing $X_{B}$. Furthermore, $M$ is of the groups listed in Table \ref{MaxDickson}. We are going to prove that the previous divisibility conditions leads to no cases for $f\geq 2^{2^{j}}> 2^{j}+1$. 

Assume that $M$ is one of the groups $A_{4},S_{4}$ or $A_{5}$. Then $f=1,2$ by Table \ref{MaxDickson}, whereas $f\geq 4$, a contradiction.

Assume that $M \cong D_{2(2^{f}+1)}$. Then $k_{1} \mid 2(2^{f}+1)f$, and hence $k_{1}\mid \left( 2^{f}+1\right)\frac{f}{2^{2^j}}$ since $k_{1}$ is odd. Since 
\begin{equation}\label{gcd1}
(k_{1},2^{f}+1)=(2k_{1},2^{f}+1)=(2^{f-2^{j}}+1,2^{f}+1)=(2^{2^{j}}-1,2^{f}+1)\text{,}   
\end{equation}
it follows that $2^{f-2^{j}-1}(2^{2^{j}}+1)<k_{1}\leq (2^{2^{j}}+1)\frac{f}{2^{2^j}}$, and so $2^{f-1}<f$ which does not have admissible solutions.

Assume that $M \cong D_{2(2^{f}-1)}$ or $M \cong Z_{2}^{f}:Z_{2^{f}-1}$ or $M \cong PGL(2,2^{f/2})$. Then $k_{1}\mid \left( 2^{f}-1\right)\frac{f}{2^{2^j}}$ since $k_{1}$ is odd. Since 
\begin{equation}\label{gcd2}
(k_{1},2^{f}-1)=(2k_{1},2^{f}-1)=(2^{f-2^{j}}+3,2^{f}-1)=(3\cdot 2^{2^{j}}+1,2^{f}-1)\text{,}
\end{equation}
it follows that $2^{f-2^{j}-1}(2^{2^{j}}+1)<k_{1}<3 (2^{2^{j}}+1)\frac{f}{2^{2^j}}$, and so $2^{f-1} < 3f$ forcing $f=4$, and hence $j=0$ or $1$, since $2^{2^j} \mid f$. Therefore, $k_1=25$ or $11$, respectively,. However, both cases are excluded since the none of the values of $k_{1}$ divides  $\left\vert G^{\Sigma
}\right\vert $ with $G^{\Sigma }=G\leq \PGaL(2,2^{4})$.

Assume that $M \cong PGL(2,2^{f/t})$ with $t$ an prime. Then $k_{1}\mid \left( 2^{2f/t}-1\right)\frac{f}{2^{2^j}}$  and by using (\ref{gcd1}) and (\ref{gcd2}), we get 
$$2^{f-2^{j}-1}(2^{2^{j}}+1)<3(2^{2^{j}}+1)^{2}\frac{f}{2^{2^j}}$$
which implies 
$$2^{f-1}<3(2^{2^{j}}+1)f\leq 3(f+1)f,$$
and we get $(f,j,t)=(4,0,2)$, $(4,1,2)$, $(6,0,3)$, $(8,0,2)$, or $(8,1,2)$. Then $k_1=13,11,49,193$ or $161$, respectively. However, in both cases we reach a contradiction since $k_1$ does not divide $\left\vert G^{\Sigma
}\right\vert $ with $G^{\Sigma }=G\leq \PGaL(2,2^{f})$.

Finally, assume that $j=0$ and $f=2$. Then $(v_{0},k_{0},v_{1},k_{1})=(3,2,5,4)$ and
hence $\mathcal{D}$ is the complementary design of $\PG(3,2)$ and $G\cong \PGaL(2,4)\cong \PGL(2,5)$. Thus, we obtain (1).

Assume that $k_{0}>2$. Then $k_{1}>\frac{2}{3}p^{f}+1$. Also, $%
r_{0}\left\vert B\cap y^{G_{x}^{\Delta }}\right\vert =\left\vert
y^{G_{x}^{\Delta }}\right\vert \lambda _{0}$ for any $y\in \Delta \setminus
\left\{ x\right\} $ by Lemma \ref{lem:basic-params}(iii) with $\mathcal{D}_{0}$ and $G_{\Delta }^{\Delta }$ in the role of $\mathcal{D}$ and $G$, respectively. Thus 
\[
\frac{p^{a}\theta \lambda _{0}}{k_{0}-1}\left\vert B\cap y^{G_{x}^{\Delta
}}\right\vert =\left\vert y^{G_{x}^{\Delta }}\right\vert \lambda _{0} 
\]%
by Lemma \ref{SigmaMax1}(2), and hence we get $p^{a}\mid \left\vert y^{G_{x}^{\Delta }}\right\vert \frac{k_{0}-1}{%
\theta }$. 

If $p \nmid \frac{k_{0}-1}{%
\theta }$, then $p^{a}\mid \left\vert y^{G_{x}^{\Delta }}\right\vert $ and
so $p^{a}\mid d$. Therefore, $p^{a}\mid f$. \ If $p^{f-a}\mid k_{0}$, then $%
p^{f-a}\mid \left\vert G_{\Delta }^{\Delta }\right\vert $ and so $%
p^{f-a}\mid d$ and hence $p^{f-a}\mid f$. Then $p^{f}\mid f^{2}$, and so $%
(p,f,a)=(2,2,1)$ or $(2,4,2)$. The former gives $v_1=4$) and $3<k_{1}<4$, a contradiction. The latter, gives $v_1=17$ and $12\le k_1\le 16$, $v_0=4\theta+1$, and $v_0\mid \lvert G_\Delta^\Delta\rvert \mid \lvert A\Gamma L(1,2^4)\rvert$. Using \ref{rel2}, we find the admissible parameters $(v_1,k_1,v_0,k_0)=(17, 16, 5, 4)$. Hence, $v=85$ and $k=64$. Using \texttt{GAP}, we obtain (2).
Thus, we can now assume $%
p^{f-a}\nmid k_{0}$ and hence $p\mid \frac{p^{f-a}v_{0}}{k_{0}}$. If $p\mid k_{0}-1$, then $k_{0} \mid v_{0}$ by Lemma \ref{SigmaMax1}(3), and hence $k_{0} \mid p^{f}-1$, and again $p\mid \frac{p^{f-a}v_{0}}{k_{0}}$. Therefore $p\mid \frac{p^{f-a}v_{0}}{k_{0}}$ in any case, and hence 
$(k_{1},p)=1$ since $k_{1}=\frac{k_{0}-1}{\theta }\cdot \frac{p^{f-a}v_{0}}{%
k_{0}}+1>\frac{2}{3}p^{f}+1$ being $k_{0}>2$. Thus, we obtain (3).
\end{proof}

\bigskip

On the basis of Lemma \ref{Uredu}, in the sequel we may assume that $k_{0}>2$. 

Set $N=X_{\Delta }^{\Delta }\cong Z_{v_{0}}$ and $K=G_{x}^{\Delta }\cong
Z_{d}$. Write $\left\vert N\right\vert =\prod_{i=1}^{h}u_{i}^{n_{i}}$ where $%
\left( u_{1},...,u_{h}\right) $ is an ordering of all the pairwise distinct
primes factors of $v_{0}$. Since $K$ normalizes the cyclic group $X_{\Delta
}^{\Delta }$, it follows that $K$ normalizes each of the subgroups $N_{j}$
of $N$ of order $v_{0}^{(j)}=\prod_{i=1}^{j}u_{i}^{n_{i}}$. Then $%
K<N_{1}K<\cdots <N_{h-1}K<N_{h}K=NK$. 

Set $\Theta _{j}=x^{N_{j}K}=x^{N_{j}}$
for $j=1,...,h$. Clearly, $\Theta _{h}=\Delta $. The aforementioned chain of
subgroups leads to this chain of $NK$-invariant partitions of $\Delta $, say 
$\Theta _{1}^{NK}<\cdots <\Theta _{h-1}^{NK}$, by \cite[Theorem 1.5A]{DM}.
Moreover, $(NK)_{\Theta _{j}}=N_{j}K$ for each $j=1,...,h$ since $N$ acts point-regularly on $\Delta$ and $K_{x} \leq (NK)_{x}$. Therefore, $(NK)_{\Theta _{j}}^{\Theta _{j}}\cong Z_{v_{0}^{(j)}}:Z_{d^{(j)}}$ for some divisor $d^{(j)}$ of $d$. Set $K_{j}=(NK)_{x}^{\Theta _{j}}$, then $K_{j}\cong Z_{d^{(j)}}$. 
 
Since $NK$ acts flag-transitively on $\mathcal{D}_{0}$, the set $%
\Theta _{h-1}$ is the point set of a (possibly trivial) $2$-design admitting the group $%
(N_{h-1}K)_{\Theta _{h-1}}^{\Theta _{h-1}}$ as a flag-transitive automorphism group by Theorem \ref{CamZie}(1), say $\mathcal{D}_{0}^{(h-1)}$. By induction, we obtain
a chain $\mathcal{D}_{0}^{(1)}<\cdots <\mathcal{D}_{0}^{(h-1)}<\mathcal{D}%
_{0}<\mathcal{D}$ of flag-transitive $2$-designs. Denote by $%
(v_{0}^{(j)},b_{0}^{(j)},k_{0}^{(j)}, r_{0}^{(j)},\lambda _{0}^{(j)})$ the
parameters tuple $\mathcal{D}_{0}^{(j)}$ for $j=1,...h-1$. It should be
stressed out that different orderings of primes factors of $v_{0}$ lead to
distinct chains of flag-transitive $2$-designs.

\bigskip

\begin{proposition}
\label{v0prime}$v_{0}$ is an odd prime distinct from $p$, and $k_{0}=p^{e}$ with $1\leq e\leq
\min (f-a-1,a)$.
\end{proposition}

\begin{proof}
Let $\left( u_{1},...,u_{h}\right) $ be an ordering of all the pairwise
distinct primes factors of $v_{0}$, and hence let $\mathcal{D}%
_{0}^{(1)}<\cdots <\mathcal{D}_{0}^{(h-1)}<\mathcal{D}_{0}<\mathcal{D}$ be
the corresponding aforementioned chain of flag-transitive $2$-designs obtained. Now, applying Lemma \ref{lem:basic-params}(iii) with to $\mathcal{D}_{0}^{(1)}$ and $(NK)_{\Theta _{1}}^{\Theta _{1}}$ in the role of $\mathcal{D}$ and $G$, one obtains $r_{0}^{(1)}\mid \left\vert
y^{K_{1}}\right\vert \lambda _{0}^{(1)}$, and hence $v_{0}^{(1)}-1\mid \left\vert y^{K_{1}}\right\vert (k_{0}^{(1)}-1)$. Since $(NK)_{\Theta _{1}}^{\Theta _{1}}\cong Z_{v_{0}^{(1)}}:Z_{d^{(1)}}$ and $K_{1}=(NK)_{x}^{\Theta _{1}}\cong Z_{d^{(1)}}$, it follows that $\left\vert y^{K_{1}}\right\vert \mid d^{(1)}$ and $ 
K_{_{1}}$ is isomorphic to a subgroup of $ \mathrm{Aut}(Z_{v_{0}^{(1)}})\cong Z_{u_{1}^{n_{1}-1}}\times
Z_{u_{1}-1}$ (recall that $v_{0}^{(1)}=u_{1}^{n_{1}}$). Therefore $%
v_{0}^{(1)}-1\mid d^{(1)} (k_{0}^{(1)}-1)$  with $d^{(1)}\mid u_{1}^{n_{1}-1}(u_{1}-1)$, and hence
\[
u_{1}^{n_{1}}-1\mid u_{1}^{n_{1}-1}(u_{1}-1)(k_{0}^{(1)}-1) 
\]%
and hence there is positive integer $t_{1}$ such that 
\begin{equation}
k_{0}^{(1)}=\frac{u_{1}^{n_{1}}-1}{u_{1}-1}t_{1}+1  \label{K01}
\end{equation}%
If $u_{1}=2$, then $t_{1}=1$ and $k_{0}^{(1)}=u_{1}^{n_{1}}=v_{0}^{(1)}$.
Then%
\[
1=\frac{v_{0}^{(1)}-1}{k_{0}^{(1)}-1}=\cdots =\frac{v_{0}^{(h-1)}-1}{%
k_{0}^{(h-1)}-1}=\frac{v_{0}-1}{k_{0}-1}=\frac{v-1}{k-1} 
\]%
by Theorem \ref{CamZie}(1)(i) applied to each member the chain $\mathcal{D}%
_{0}^{(1)}<\cdots <\mathcal{D}_{0}^{(h-1)}<\mathcal{D}_{0}<\mathcal{D}$.
Thus $v=k$, a contradiction. Therefore, $u_{1}>2$.

Since $(NK)_{\Theta _{1}}^{\Theta _{1}}$ is a group of order $v_{0}^{(1)}d^{(1)}$ acting flag-transitively on $\mathcal{D}_{0}^{(1)} $, it follows that $r_{0}^{(1)} \mid d^{(1)}$. Then $k_{0}^{(1)}\mid v_{0}^{(1)}d^{(1)}$ since $k_{0}^{(1)}\mid v_{0}^{(1)}r_{0}^{(1)}$, then $%
k_{0}^{(1)}\mid u_{1}^{2n_{1}-1}(u_{1}-1)$ since $v_{0}^{(1)}=u_{1}^{n_{1}}$
and $d^{(1)}\mid u_{1}^{n_{1}-1}(u_{1}-1)$. Therefore, $%
k_{0}^{(1)}=u_{1}^{m_{1}}\frac{u_{1}-1}{\rho _{1}}$ for some $0\leq
m_{1}\leq 2n_{1}-1$ and $\rho _{1} $ a divisor of $u_{1}-1$. Hence, we obtain
\begin{equation}\label{DubiozaKolektiv}
u_{1}^{m_{1}}\frac{u_{1}-1}{\rho _{1}}=\frac{u_{1}^{n_{1}}-1}{u_{1}-1}%
t_{1}+1 
\end{equation}%
by (\ref{K01}). If $m_{1}>0$, then $u_{1}\mid \frac{u_{1}^{n_{1}}-1}{%
u_{1}-1}t_{1}+1$ and hence $u_{1}\mid t_{1}+1$. Then $t_{1}\geq u_{1}-1$ and
hence $k_{0}^{(1)}\geq u_{1}^{n_{1}}=v_{0}^{(1)}$, a contradiction. Thus $%
m_{1}=0$, and hence (\ref{DubiozaKolektiv}) becomes
\[
\frac{\left( u_{1}-1\right) ^{2}}{\rho _{1}}=\left( u_{1}^{n_{1}}-1\right)
t_{1}+\left( u_{1}-1\right) \text{.} 
\]%
and so $n_{1}=1$, that is to say $v_{0}^{(1)}$ is a prime number. Now, repeating the
previous argument with respect the ordering of the primes factors of $%
v_{0}$ via the transposition $(1j)$, namely $\left( u_{1},...,u_{h}\right)
^{(1j)}$, we see that $n_{j}=1$. Therefore, $v_{0}^{j}$ is a prime. Hence, $v_{0}=\prod_{i=1}^{h}u_{i}$ with $u_{i}$ pairwise distinct prime numbers.

If $h>1$, then $\mathcal{D}_{0}^{(2)}$ is a $2$-$%
(v_{0}^{(2)},k_{0}^{(2)},\lambda _{0}^{(2)})$ design, with $%
v_{0}^{(2)}=u_{1}u_{2}$ admitting $(N_{2}K)_{\Theta _{2}}^{\Theta _{2}}\cong Z_{v_{0}^{(2)}}:Z_{d^{(2)}}$ as a
flag-transitive automorphism group. In particular, $(N_{2}K)_{\Theta _{2}}^{\Theta _{2}}$ contains a normal subgroup isomorphic to $ Z_{v_{0}^{(2)}}$ and acting point-regularly on $\mathcal{D}_{0}^{(2)}$. Then either $k_{0}^{(2)}=2$ or $k_{0}^{(2)}=v_{0}^{(2)}-1=u_{1}u_{2}-1$ by \cite[Theorem 3.6]{CZ} (note that, $2$-design is assumed non
trivial in \cite[Theorem 3.6]{CZ}). This forces $(N_{2}K)_{\Theta _{2}}^{\Theta _{2}}$ to act point-$2$-transitively on $\mathcal{D}_{0}^{(2)}$ since it acts flag-transitively on $\mathcal{D}_{0}^{(2)}$. However, this is impossible by \cite[Theorem 4.1B]{DM} since $%
v_{0}^{(2)}=u_{1}u_{2}$, with $u_{1}\neq u_{2}$, and $(N_{2}K)_{\Theta _{2}}^{\Theta _{2}}\cong Z_{v_{0}^{(2)}}:Z_{d^{(2)}}$ is solvable. Therefore $h=1$ and hence $v_{0}$ is
a prime number. 

Since $NK$ is group of order $v_{0}d$ inducing a flag-transitive group on $\mathcal{D}_{0}$, it follows $k_{0}\mid v_{0}d$. Moreover, $d\mid v_{0}-1$ since $v_{0}$ is a prime number, and $k_{0} \neq v_{0}$ by (\ref{rel1}) since $k \neq v$. Then $k_{0} \nmid v_{0}$ again since $v_{0}$ is a prime number, and hence $k_{0}\mid v_{0}-1$ since $k_{0}\mid v_{0}d$ and $d\mid v_{0}-1$. Then $k_{0}\mid
p^{a}$ since $v_{0}=p^{a}\theta +1$ with $1\leq a <f$ by Lemma \ref{OndaDobro}(2). Therefore, $k_{0}=p^{e}$ with $1 \leq e\leq a$. On the
other hand \thinspace $e\leq f-a-1$ by Lemma \ref{Uredu}(2)(b). Thus $1\leq
e\leq \min (f-a-1,a)$. This completes the proof.
\end{proof}

\begin{corollary}
\label{fdosta} $p^{a} \mid f$ and $f\geq 2a+2\geq 4$.
\end{corollary}

\begin{proof}
It follows from Proposition \ref{v0prime} that $p\nmid k_{0}-1$, hence $p^{a} \mid f$ by Lemmas \ref{OndaDobro}(1) and \ref{Uredu}(3).

If $f-a-1\leq a$, then $f-1\leq 2a$. Now, since $p^{a}\mid f$, it
follows that $p^{f-1}\leq p^{2a}\leq f^{2}$. Since $1\leq e\leq f-a$, it
follows that $f\geq a+1\geq 2$. Thus either or $%
p^{f}=2^{2},2^{3},2^{4},2^{5},2^{6},3^{2},3^{3}$. Actually, $%
p^{f}=2^{2},3^{3}$ and $a=1$, or $p^f=2^4$ and $a=2$ since $p^{a}\mid f$ with $f-1\geq a\geq \max
\left( \frac{f-1}{2},1\right) $. 
%
Then $k_{0}=p^e$ with $1\leq e \leq f-a-1$ by Proposition \ref%
{v0prime}, and only $p^f=3^3$ is admissible since $k_0>2$ by our assumption. Then $k_0=3$ and $v_0=3\theta+1$ with $\theta\mid k_0-1=2$. Since $v_0$ is an odd prime by Proposition \ref{v0prime}, we have that $v_0=7$. Then $G_{\Delta}^{\Delta} \leq Z_{7}:Z_{4}$, which is not flag-transitive on $\mathcal{D}_{0}$ as it is not divisible by $k_{0}=3$. Thus $f>2a+1$, and hence $f\geq 2a+2$.

\end{proof}

\begin{lemma}
\label{Boom}$k_{1}=p^{f}+1-p^{f-e}+\frac{p^{e}-1}{\theta }p^{f-a-e}$.
Moreover the following holds:

\begin{enumerate}
\item $\left( k_{1},p^{f}\right) =1$ and $(k_{1},f)\mid \frac{f}{p^{a}}$;

\item $\left( k_{1},p^{f}-1\right) \mid \frac{p^{e}-1}{\theta }+p^{a}\left(
2p^{e}-1\right) $;

\item $\left( k_{1},p^{f}+1\right) \mid p^{a}-\frac{p^{e}-1}{\theta }$
\end{enumerate}
\end{lemma}

\begin{proof}
It follows from Lemmas \ref{SigmaMax1}(4), \ref{OndaDobro}(2) and Proposition \ref{v0prime} that 
\[
k_{1}=\frac{k_{0}-1}{\theta }\cdot \frac{p^{f-a}v_{0}}{k_{0}}+1=p^{f}+1-p^{f-e}+\frac{p^{e}-1}{\theta }p^{f-a-e}
\]
Moreover, since $(k_{1},p)=1$ and $p^{a}\mid f$ by Lemma %
\ref{Uredu}(2) and Corollary \ref{fdosta}, the assertion (1) follows.

Now, since $p^{a+e}k_{1} \equiv 2p^{e+a}-p^{a}+\frac{p^{e}-1}{\theta } \pmod{p^{f}-1}$, it follows that
\begin{equation*}
\left( k_{1},p^{f}-1\right)=\left( p^{a+e}k_{1},p^{f}-1\right)=\left( 2p^{e+a}-p^{a}+\frac{p^{e}-1}{\theta },p^{f}-1\right)\text{,}
\end{equation*}%
which is the assertion (2).

Finally, since $p^{a+e}k_{1} \equiv p^{a}-\frac{p^{e}-1}{\theta } \pmod{p^{f}+1}$, it follows that
\begin{equation*}
\left( k_{1},p^{f}+1\right)=\left( p^{a+e}k_{1},p^{f}+1\right)=\left( p^{a}-\frac{p^{e}-1}{\theta },p^{f}+1\right)\text{,}
\end{equation*}
which is the assertion (3).
\end{proof}

\bigskip

Now, we are in position to prove Theorem \ref{FTPI}.

\bigskip

\begin{proof}[Proof of Theorem \ref{FTPI}]
It follows from Lemmas \ref{tris} and \ref{Uredu} that either the assertions (1), (2) or (3) hold, or $G$ acts flag-transitively and point-$2$-transitively on the $2$-$(p^{f}+1,k_{1},\lambda_{1})$ design $\mathcal{D}_{1}$, $(k_1,p)=1$, $k_{0}>2$ and $p\mid \frac{p^{f-a}v_0}{k_0}$. Hence, we assume that the latter case occurs.

Since $G$ acts flag-transitively on the $2$-$(p^{f}+1,k_{1},\lambda_{1})$ design, then $k_{1}\mid \left\vert G_{B}\right\vert $, where $B$ is any block of $\mathcal{D}$. Hence $k_{1}\mid \left\vert M \right\vert \cdot (2,p^{f}-1)f$ with $M$ a maximal subgroup of $X$ containing $X_{B}$. Furthermore, $M$ is of the groups listed in Table \ref{MaxDickson}. We are going to show that the previous divisibility conditions leads to no cases. 

If $M\cong A_{4}$, $S_4$ or $A_5$, then $f=1,2$ by Table \ref{MaxDickson}, which is a contradiction because $f\geq 4$ by Corollary \ref{fdosta}.

Assume that $M \cong D_{\frac{2(p^{f}+1)}{(2,p^{f}-1)}}$. Then $k_{1} \mid 2(p^{f}+1)f$, and hence 
\[
k_{1}\mid 2\left( p^{a}-\frac{p^{e}-1}{\theta }\right) \frac{f}{p^{a}} 
\]%
by Lemma \ref{Boom}(1),(3). Then $\frac{2}{3}p^{f}<k_{1}<2f\leq 2p^{f/2}$ by
Lemma \ref{SigmaMax1}(4), and so $p^{f/2}<3$. Thus $p^{f}<9$ with $f\geq 4$
by Corollary \ref{fdosta}, a contradiction.

Assume that $M\cong D_{\frac{2(p^{f}-1)}{(2,p^{f}-1)}}$. Then $k_{1} \mid 2(p^{f}-1)f$, and hence
\[
k_{1}\mid 2\left( \frac{p^{e}-1}{\theta }+p^{a}\left( 2p^{e}-1\right)
\right) \frac{f}{p^{a}} 
\]%
by Lemma \ref{Boom}(1)--(2). 
Then $\frac{2}{3}p^{f}<k_{1}<4p^{e}f+2f\leq
4p^{f/2+e}+2p^{f/2}$ by Lemma \ref{SigmaMax1}(4), 
then $p^{f/2-e}<9$ and so $p\leq p^{f/2-a}\leq p^{f/2-e}<9$ by Corollary \ref%
{fdosta} since $e\leq a$ by Proposition \ref{v0prime}. Hence, we obtain either $e=a=f/2-1$
and $p=2,3,5,7$ or $e=f/2-2$ and $p=2$.\\
In the former case, one has $f=2a+2$
and hence $p^{a}\mid 2a+2$, so $p=2$ and $\left( a,f\right)
=(1,4),(3,8)$. If $(a,f)=(1,4)$ then $k_0=2^e=2^a=2$, a contradiction. If $(a,f)=(3,8)$ then $v_0=2^3\theta+1$ with $\theta\mid k_0-1=2^3-1=7$. Hence, $v_0=9$ or $57$, which contradicts Proposition \ref{v0prime}.
Thus $e=f/2-2$ and $p=2$. If $a=f/2-1$, from $2^a\mid f=2a+2$ we get $(a,f)=(1,4),(3,8)$. The former gives $e=0$, a contradiction. The latter gives $e=2$, $k_0=4$ and then $\theta\mid 3$. Hence $v_0=8\theta+1=25$ or $9$, which is impossible by Proposition \ref{v0prime}. If $a=e=f/2-2$, from $2^a\mid 2a+4$ we get $(a,f)=(1,6),(2,8)$. The former gives $k_0=2$, a contradiction. The latter gives $k_0=4$, $\theta\mid 3$ and hence $v_0=2^4\theta+1=5$ or $13$, and by Lemma \ref{SigmaMax1} $k_1=241$ or $209$, respectively. However, in any case we have that $k_1\nmid \lvert\PGaL(2,2^8)\rvert$, a contradiction.

Assume that $M\cong \left( Z_{p}\right) ^{f}:Z_{\frac{p^{f}-1}{(2,p-1)}}$. Then $k_{1} \mid p^{f}(p^{f}-1)f$, and hence $%
k_{1}\mid (p^{f}-1)f$ by Lemma \ref{Boom}(1). Then $k_{1}\mid 2(p^{f}-1)f$, and no cases occur by the
previous argument.

Assume that either $M$ is is isomorphic to one of the groups $PSL(2,p^{f/t})$ or $ 
PGL(2,p^{f/t})$ with $t\mid f$ and $t$ prime, or $PGL(2,p^{f/t})$ with $t=2$ and $p$ odd. Then $k_{1} \mid p^{f/t}(p^{2f/t}-1)(2,p^{f}-1)f$ and hence $k_{1}\mid
(p^{2f/t}-1)(2,p^{f}-1)\frac{f}{p^{a}}$ by Lemma \ref{Boom}(1). Then 
\[
\frac{2}{3}p^{f}<k_{1}\leq
2(p^{2f/t}-1)\frac{f}{p^{a}} 
\]%
and so $p^{f}<3(p^{2f/t}-1)\frac{f}{p^{a}}$. Then $p^{f/2+a}<3p^{2f/t}$ and
hence either $t=2,3$ as $t$ is prime. If $t=2$, then $%
k_{1}\mid (p^{f}-1)(2,p^{f}-1)\frac{f}{p^{a}}$ and hence $k_{1}\mid
2(p^{f}-1)f$, and we have seen that this case cannot occur. Finally, if $t=3$ then $p^{f/3+1}\leq p^{f/3+a}<3f$ and hence $p=2$ and $%
f=3,6,9,12$ and hence $p=2$ and $f=6,9,12$ as $3\mid f$ and $f\geq 4$. Then $k_{1}\mid (2^{2f/3}-1)\frac{f}{2^{a}}$ and so $f=6$ or $12$. Then $k_{1}\mid (2^{2f/3}-1)\frac{f}{2^{a}}$ and so $\frac{2}{3}2^f<k_1\leq (2^{2f/3}-1)\frac{f}{2^{a}}$.  This gives $(f,a)=(6,1)$, hence $e=1$ and $k_0=2$, a contradiction. This completes the proof.
\end{proof}

\section{Completion of the proof of Theorem \ref{main}}\label{FT-PP-Geometry-Section}
In this section we settle the cases (3) of Theorem \ref{FT-PP-reduction}, thus completing the proof of Theorem \ref{main}, and provide some reductions for cases (4) and (5).

Before proceeding, define any $2$-design as in (3)--(5) of Theorem \ref{FT-PP-reduction} being of type I--III, respectively. 

\bigskip
Recall that $X \cong \PSL(2,q)$, $q=p^f$ and let $\theta =\left\vert G:X\right\vert $, where $\theta \mid f$. Let $B$ be any block of $\mathcal{D}$ and $\alpha ,x$ be any two distinct
points of $\mathcal{D}$ such that $\alpha \notin B$ and $x\in B$. Finally,
let $\theta =\left\vert G:X\right\vert $,  $\theta _{1}=\left\vert
G_{B}:X_{B}\right\vert $, $\theta _{2}=\left\vert G_{\alpha ,B}:X_{\alpha
,B}\right\vert $ and $\theta _{3}(x)=\left\vert G_{\alpha ,x}:X_{\alpha
,x}\right\vert $. Since $G$ is flag-transitive on $\mathcal{D}$, then $G$ is
both block-transitive on $\mathcal{D}$. Thus $\theta _{2}$, $\theta _{1}$
and do not depend on the flag $(\alpha ,B)$ or $B$, respectively. On the contrary $%
\theta _{3}(x)$ depends on $x$, once $\alpha $ is fixed. Since $X\unlhd G$ and $G$ acts point-primitively on $\mathcal{D}$, it follows that
$\left\vert G:G_{\alpha }\right\vert =\left\vert
X:X_{\alpha }\right\vert $, hence $\left\vert G_{\alpha }:X_{\alpha }\right\vert =\left\vert
G:X\right\vert =\theta $. Moreover, $G_{x,B}\leq G_{B}X\leq G$ and $G_{\alpha ,x}X\leq G$ imply $\theta _{i}\mid \theta $ for each $i=1,2$, $\theta _{3}(x)\mid \theta $ and $\theta _{2}\mid \theta _{1}$. In particular, 
\begin{equation}\label{abroad}
\left\vert x^{G_{\alpha
}}\right\vert=\frac{\left\vert
G_{\alpha }\right\vert }{\left\vert X_{\alpha }\right\vert }\cdot \frac{%
\left\vert X_{\alpha ,x}\right\vert }{\left\vert G_{\alpha ,x}\right\vert }%
\cdot \left\vert x^{X_{\alpha }}\right\vert=\frac{\theta}{\theta_{3}(x)}\left\vert x^{X_{\alpha }}\right\vert\text{.}    
\end{equation}

\bigskip
The symbols of $\alpha$, $B$, $\theta, \theta_{1}, \theta _{2}$ and $\theta _{3}(x)$ will have these fixed meaning throughout this section. Moreover, unless differently specified, $\theta _{3}(x)$ will simply be denoted by $\theta _{3}$.

\bigskip

$\PGL(3,q)$ has a unique conjugacy class of subgroups isomorphic
to $\PSL(2,q)$ by \cite[Table 8.3]{BHRD} (note that $\PSL(2,q) \cong \Omega (3,q)$ by \cite[Corollary 7.14]{Hir} is reducible
and not maximal in $\PGL(3,q)$ when $q$ is even. Further, the irreducible conics do not arise from
polarities in this case), and each of these groups is the stabilizer in $%
\PGL(3,q)$ of a suitable regular hyperoval of $\PG(2,q)$. The converse is also
true as a consequence of \cite[Theorem 7.4]{Hir}. Therefore, we may choose $X$ to be the copy of$\ PSL(2,2^{f})$ preserving the conic 
\begin{equation*}
\mathcal{C}:Y_{0}Y_{2}-Y_{1}^{2}=0\text{,}
\end{equation*}%
namely $X$ consists of the collineations of the plane represented, up to a
nonzero scalar, by the matrices 
\begin{equation*}
\left( 
\begin{array}{ccc}
\begin{array}{ccc}
a_{11}^{2} & a_{11}a_{13} & a_{13}^{2} \\ 
0 & a_{11}a_{33}+a_{13}a_{31} & 0 \\ 
a_{31}^{2} & a_{31}a_{33} & a_{33}^{2}%
\end{array}
\end{array}%
\right) 
\end{equation*}%
with coefficients in $\mathbb{F}_{p^{f}}$ such that $%
a_{11}a_{33}+a_{13}a_{31}\neq 0$. Then $G=X:\left\langle \varsigma
^{f/\theta }\right\rangle $, where $\varsigma
:(Y_{0}:Y_{1}:Y_{2})\longrightarrow (Y_{0}^{p}:Y_{1}^{p}:Y_{2}^{p})$ is the
collineation of $\PG(2,q)$ induced by the Frobenius automorphism of $%
\mathbb{F}_{q}$.

Assume that $q$ is even. Then the nucleus of $\mathcal{C}$ is $Y_{\infty}=(0:1:0)$ by \cite[Corollary 7.12]{Hir}, and in this case we denote by $\ \mathcal{H}$ the (regular) hyperoval $\mathcal{C}\cup \left\{
Y_{\infty}\right\}$. Moreover, a line with coordinates $%
[1:1:x]$ is external to $\mathcal{C}$ if and only if $T_{\mathbb{F}_{2^{f}}/%
\mathbb{F}_{2}}(x)=1$ (e.g. see \cite[Section 1.4(ii.b)]{Hir}), where 
\begin{equation*}
T_{\mathbb{F}_{2^{f}}/\mathbb{F}_{2}}(x)=\sum_{j=0}^{f-1}x^{2^{i}}
\end{equation*}%
denotes the absolute trace of $x$.

As the actions of $G$ on the set $\mathcal{E}$ of the external lines to $\mathcal{H}$ and on
the conjugacy class of the subgroups of $G$ isomorphic to $G_{\alpha }\cong
D_{2(2^{f}+1)}:Z_{\theta }$ are equivalent, we may identify the points of $\mathcal{D%
}$ with the set $\mathcal{E}$. Therefore, a block of $\mathcal{D}$ is a suitable $k$-subset of $\mathcal{E}$. 
\bigskip

\begin{lemma}\label{I.1}
If $\mathcal{D}$ is a $2$-design as in case (3) of Theorem \ref{FT-PP-reduction}, then the following hold:
\end{lemma}

\begin{enumerate}
\item 
$\left\vert G_{\alpha }\right\vert =2(q+1)\theta$, $\left\vert
G_{\alpha ,B}\right\vert =\frac{2(q+1)\theta }{3(q+1)\frac{\lambda }{2}}=%
\frac{4\theta }{3\lambda }$ and $\left\vert G_{B}\right\vert =\frac{q+1%
}{3}\frac{4\theta }{3\lambda }$; 

\item Let $x$ be any point of $\mathcal{D}$ distinct from $\alpha$, then $\theta_{3}(x) \mid \frac{\theta}{3}$ and $\left\vert B\cap x^{G_{\alpha }}\right\vert =2\frac{\theta }{%
3\theta _{3}(x)}$;

\item $\left\vert G_{\alpha ,x}\right\vert =2\theta_{3}(x)$.
\end{enumerate}

\begin{proof}
Since $\left\vert G_{\alpha }:X_{\alpha }\right\vert =\left\vert
G:X\right\vert =\theta $ and $X_{\alpha}\cong \D_{2(2^f+1)}$, it follows that $\left\vert G_{\alpha }\right\vert=2(2^f+1)\theta $. Now, since $r=3(2^f+1)\frac{%
\lambda }{2}$ and $k=\frac{2^f+1}{3}$, we obtain
\[
\left\vert G_{\alpha ,B}\right\vert =\frac{2(2^f+1)\theta }{3(2^f+1)\frac{%
\lambda }{2}}=\frac{4\theta }{3\lambda }\text{ and }\left\vert
G_{B}\right\vert =\frac{2^f+1}{3}\cdot \frac{4\theta }{3\lambda} \text{.} 
\]%
Let $x$ be any point of $\mathcal{D}$ distinct from $\alpha$. Since $r\left\vert B\cap x^{G_{\alpha }}\right\vert  =\left\vert x^{G_{\alpha
}}\right\vert \lambda$ by Lemma \ref{lem:basic-params}(3), and $\left\vert x^{G_{\alpha
}}\right\vert=\frac{\theta}{\theta_{3}(x)}(2^f+1)$ by (\ref{abroad}) and Table \ref{tab:subdegrees}, one obtains
\begin{equation*}
3\left\vert B\cap x^{G_{\alpha }}\right\vert  =2\frac{\theta }{\theta _{3}(x)}\text{,}    
\end{equation*}
and hence $\left\vert B\cap x^{G_{\alpha }}\right\vert =2\frac{\theta }{%
3\theta _{3}}$ with $3\theta_{3}(x) \mid \theta$. Again by Table \ref{tab:subdegrees}, it results that $\left\vert X_{\alpha}\right\vert=2$. Therefore, $\left\vert G_{\alpha ,x}\right\vert =2\theta_{3}(x)$.
\end{proof}

\begin{proposition}\label{Ifin}    
If $\mathcal{D}$ is a $2$-design as in case (3) of Theorem \ref{FT-PP-reduction}, then  $G=\PGaL(2,2^{3})$ and one of the following holds:

\begin{enumerate}
\item $\mathcal{D}$ is one of the two $2$-$(28,3,2)$ designs as in Lines 19, 20 of Table \ref{tab:alldesigns};
\item  $\mathcal{D}$ is one of the $2$-$(28,3,4)$ design as in Line 21 of Table \ref{tab:alldesigns}.
\end{enumerate}
\end{proposition}

\begin{proof}
The collineation $\varsigma^{f/\theta }$ fixes the points $Y_{\infty}=(0:1:0)$, $O=(0:0:1)$ and $P=(1:0:1)$ on the line $Y_{\infty}O$ tangent to $\mathcal{C}$. Moreover, $\varsigma^{f/\theta }$ fixes the line $\alpha=[1:1:1]$, which is external to $\mathcal{C}$ since $T_{\mathbb{F}_{2^{f}}/\mathbb{F}_{2}}(1)=1$, being $f$ odd. On the other hand, $Fix(\sigma ^{f/\theta })$ fixes a subplane of $\PG(2,2^f)$ isomorphic to $\PG(2,2^{f/\theta})$, hence $\left \vert Fix(\sigma ^{f/\theta })\cap Y_{\infty}O \right\vert=2^{f/\theta}+1$. Moreover, $\left \vert Fix(\varsigma ^{f/\theta })\cap \mathcal{H} \right\vert$ is the sub-hyperoval of $\mathcal{H}$ consisting of the points of $\mathcal{H}$ with projective coordinates, up to a nonzero scalar, in $\mathbb{F}_{2^{f/\theta}}$. Then there are $2^{f/\theta-1}(2^{f/\theta}-1)$ lines of $Fix(\varsigma ^{f/\theta })$ which are external to $Fix(\varsigma ^{f/\theta })\cap \mathcal{H}$. These lines of $Fix(\varsigma ^{f/\theta })$ are the intersections of lines of $\PG(2,2^{f})$ external to $\mathcal{H}$ since $f$ is odd. 
If $\theta <f$, then $2^{f/\theta-1}(2^{f/\theta}-1)>1$, then $\varsigma ^{f/\theta
}\leq G_{\alpha ,x}$ for some suitable line $x$ external to $\mathcal{H}$, hence $\theta \mid 2\theta _{3}$ by Lemma \ref{I.1}(3), and so $\theta
\mid \theta _{3}$ since $\theta $ is odd, whereas $\theta _{3}\mid \frac{\theta }{%
3}$. Thus $\theta =f$, and so $G=\PGaL(2,2^{f})$.

Let $3^{u}$ be the maximum power of $3$ dividing $f$. If $f\neq 3^{u}$, then 
$\varsigma ^{f/3^{u}}\neq 1$ and hence repeating the previous argument with $%
\varsigma ^{f/3^{u}}$ in the role of $\varsigma ^{f/\theta }$ we see that $%
3^{u}\mid 2\theta _{3}$ and hence $3^{u}\mid \theta _{3}$ since $\theta _{3}$
is odd, and so $3^{u+1}\mid f $ since $3\theta _{3}\mid f $ by by Lemma \ref{I.1}(2), a
contradiction. Thus $f=3^{u}$ and $G=\PGaL(2,2^{3^{u}})$.

As the actions of $G$ on the set $\mathcal{E}$ of the external lines to $\mathcal{H}$ and on
the conjugacy class of the subgroups of $G$ isomorphic to $G_{\alpha }\cong
D_{2(2^{f}+1)}:Z_{\theta }$ are equivalent, we may identify the points of $\mathcal{D%
}$ with the set $\mathcal{E}$.  `

The $2^{3^u-1}$ lines of $\mathcal{E}$ incident with $P=(0:1:1)$ are $\ell_{c}=[1:1:c^2+c+1]$ with $c\in \mathbb{F}_{2^{3^u}}$ since $T_{\mathbb{F}_{2^{3^{u}}}/%
\mathbb{F}_{2}}(c^2+c+1)=T_{\mathbb{F}_{2^{3^{u}}}/%
\mathbb{F}_{2}}(1)=1$ (note that $\ell_{c}=\ell_{c+1}$).
Assume $\ell_{c_{1}}, \ell_{c_{2}}$, with $c\neq 0,1$ are two distinct lines $\mathcal{E}$ incident with $P=(1:1:0)$ that lie in the same $X_{\alpha}$-orbit. However, $\ell_{c_{1}}, \ell_{c_{2}}$ and $\alpha$ are three distinct lines concurrent in $P$, whereas they should lie in line-hyperoval of $\PG(2,2^{3^{u}})$ by the dual of \cite[Lemma 3.2]{OKP} ($\PG(2,2^{3^{u}})$ is self-dual). Thus distinct lines $\ell_{c}$'s, with $c\neq 0,1$ are representatives of the $2^{3^u-1}$ orbits under $X_{\alpha} \cong D_{2(2^{3^u}-1)}$ on $\mathcal{E}$ by Table \ref{tab:subdegrees}. 

Now, $G_{\alpha}=\left\langle \gamma \right\rangle: \left( \left\langle \tau \right\rangle \times \left\langle \varsigma \right\rangle \right)$, where $\left\langle \gamma \right\rangle \cong Z_{2^{3^u}-1}$, $\left\langle \tau \right\rangle \cong Z_{2}$ and $\left\langle \varsigma \right\rangle \cong Z_{3^{u}}$, with $\tau$ an involutory $\left(P,\ell_{\infty}\right)$-elation of $\PG(2,2^{3^{u}})$. Let $\ell=\ell_{c}$ with $c\neq 0,1$ and $\eta \in G_{\alpha,\ell}$, then $\eta=\tau^{i}\varsigma^{3^{j}}\gamma^{h}$ for some $i=0,1$ and $j,h \geq 0$. Then $P \in \ell$, $\ell \neq \alpha$ and $\ell^\eta=\ell$. Hence, $\ell^{\tau^{i}\varsigma^{3^{j}}}=\ell^{\gamma^{-h}}$. Since both $\tau$ and $\varsigma$ fix $P$, then $P\in \ell^{\tau^{i}\varsigma^{3^{j}}}$, then $P \in \ell^{\gamma^{-h}}$, and so $h=0$ since distinct lines $\ell_{c}$'s, with $c\neq 0,1$ are representatives of the distinct orbits under $X_{\alpha}$. Thus $\eta=\tau^{i}\varsigma^{3^{j}}$ for some $i=0,1$ and $0 \leq j \leq u$, and hence $\left\langle \tau \right\rangle \unlhd G_{\alpha, \ell} \leq  \left\langle \tau \right\rangle \times \left\langle \varsigma \right\rangle $. Suppose that $G_{\alpha, \ell_{c_{0}}} = \left\langle \tau \right\rangle \times \left\langle \varsigma \right\rangle$ for some $\ell_{c_{0}}=[1:1:c_{0}^2+c_{0}+1]$. Then $(c_{0}^2+c_{0}+1)^{2}=c_{0}^2+c_{0}+1$ with $T_{\mathbb{F}_{2^{3^{u}}}/\mathbb{F}_{2}}(c_{0}^2+c_{0}+1)=1$, and hence $c_{0}^2+c_{0}+1=1$. Thus $c_{0}=0,1$, and so $\ell_{c_{0}}=\alpha$. Therefore, for each $\ell_{c}=[1:1:c^2+c+1]$, with $\ell_{c}\neq \alpha$, the group $G_{\alpha, \ell_{c}}=\left\langle \tau \right\rangle \times \left\langle \varsigma^{3^{s_{c}}} \right\rangle$ for some $1\leq s_{c} \leq u$.

It is worth noting that, $s$ is the minimal positive integer such that $c^2+c+1$ lies in the subfield $\mathbb{F}_{2^{3^s}}$ of $\mathbb{F}_{2^{3^u}}$. Since $\left\vert \ell_{c}^{G_{\alpha }}\right\vert=(2^{3^u}+1)3^{u-s_{c}}$ and $\theta _{3}(\ell_{c})=3^{u-s_{c}}$, it follows that $$\left\vert B\cap \ell_{c}^{G_{\alpha }}\right\vert =2\frac{\theta }{3\theta _{3}(\ell_{c})}=2 \cdot 3^{u-1-(u-s_{c})}=2 \cdot 3^{s_{c}-1}$$ by Lemma \ref{I.1}(2). For each $1\leq s \leq u$, let $n_{s}\geq 0$ be the number of $G_{\alpha}$-orbits of the form $\ell_{c}^{G_{\alpha }}$, with $c \neq 0,1$, and length $(2^{3^u}+1)3^{u-s_{c}}$. Then $n_{1}+\cdots + n_{u}=2^{3^{u}-1}$ and
\begin{equation}\label{fond}
\frac{2^{3^{u}}+1}{3}=k=1+2\sum_{s=1}^{u}n_{i}\cdot 3^{s-1}\text{.}    
\end{equation}
Since the polynomial $Z^{2}+Z+m \in \mathbb{F}_{2^{3}}[Z]$ is completely reducible in $\mathbb{F}_{2^{3^{u}}}$ if and only if it so in $\mathbb{F}_{2^{3}}$ (otherwise the roots lie a quadratic extension of $\mathbb{F}_{2^{3}}$), it follows that $c^{2}+c+1 \in \mathbb{F}_{2^{3}}$, with $c \in \mathbb{F}_{2^{3^{u}}}$, if and only if $c \in \mathbb{F}_{2^{3}}$. Hence, the number of such elements $c^{2}+c+1 \in \mathbb{F}_{2^{3}} \setminus \{0,1 \}$, is $3$, and consequently $1 \leq n_{1}\leq 3$. 

If $u>1$, both $\frac{2^{3^{u}}+1}{3}$ and $\sum_{i=2}^{u}n_{i}\cdot 3^{i-1}$ are divisible $3^{2}$, and hence (\ref{fond}) implies $1+2n_{1} \equiv 0 \pmod{3^{2}}$, which is not the case since $1 \leq n_{1}\leq 3$. Thus $u=1$. Therefore, $q=8$, then $(v,k,r,\lambda ,b)=(28,3,27\lambda /2,\lambda
,\ 126\lambda )$ with $\lambda =2,4$ and $G= P\Gamma
L(2,8)$. Note that $G\cong $ $%
^{2}G_{2}(3)$ is the smallest Ree group, and $G$ has a unique permutation
representation of degree $28$, namely the one on the points of $\mathcal{R}%
(3)$, the Ree unital of order $3$. In particular, $G$ acts $2$-transitvely
on the points of $\mathcal{R}(3)$. We may identify the point set of $%
\mathcal{D}$ with that $\mathcal{R}(3)$. The $2$-transitivity implies that
any $G$-orbit of length $3$ leads to flag-transitive $2$-designs with $v=28$
and $k=3$. By \cite{At}, there are exactly $2$ conjugacy $G$-classes of
subgroups of order $3$, namely one represented by subgroups of order $3$ of 
$PSL(2,8)$ and one generated by $3$-elements induced by automorphisms of $%
\mathbb{F}_{2^{3}}$.  In the former case orbits are contained in the blocks
of $\mathcal{R}(3)$, and hence $\lambda =4$, in the latter triangles of $\mathcal{R}(3)$ in the latter $%
G_{B}\cong Z_{3}$ and hence $\lambda =2$.
\end{proof}

\bigskip

\begin{proof}[Proof of Theorem \ref{main}.]
The proof Theorem \ref{main} is the immediate consequence of the combination of Theorem \ref{FT-PP-reduction} and Proposition \ref{Ifin} or by Theorem \ref{FTPI} according as when $G$ acts point-primitively or point-imprimitively on $\mathcal{D}$, respectively.  
\end{proof}

\section{Appendix}\label{Appendix}
In this final section, for each of the $2$-design and corresponding group $G$ recorded in Table \ref{tab:alldesigns} we provide a base block of $\mathcal{D}$ and generators of $G$. In this way a potential reader can quickly construct the design by using the package \texttt{Design} of \texttt{GAP}. We generated an hypertext in the table so that by clicking on the line of Table \ref{tab:alldesigns} a reader is redirected here to the corresponding base block of $\mathcal{D}$ and generators of $G$. 

\begin{enumerate}
    \item\label{533} $(v,k,\lambda)=(5,3,3)$, $G=\PSL(2,5)$\\
    \texttt{g1:=(1,2,3,4,5);}\\
    \texttt{g2:=(3,4,5);}\\
    \texttt{B:=[1,2,3];}\\

    \item\label{632} $(v,k,\lambda)=(6,3,2)$, $G=\PSL(2,2^2)$ \\
    \texttt{g1:=(1,6,4)(2,5,3);}\\
    \texttt{g2:=(1,5,4)(2,6,3);}\\
    \texttt{g3:=(1,2,4)(3,6,5);}\\
    \texttt{B:=[1,2,3];}\\

    \item\label{634_1} $(v,k,\lambda)=(6,3,4)$, $G=\PGaL(2,2^2)$ \\
    \texttt{g1:=(1,6,5,4);}\\
    \texttt{g2:=(1,5,6)(2,3,4);}\\
    \texttt{B:=[1,2,3];}\\

     \item\label{634_2}$(v,k,\lambda)=(6,3,4)$, $G=\PSL(2,3^2)$\\
    \texttt{g1:=(1,2,3,4,5);}\\
    \texttt{g2:=(4,5,6);}\\
    \texttt{B:=[1,2,3];}\\

    \item\label{646_1}$(v,k,\lambda)=(6,4,6)$, $G=\PGaL(2,2^2)$, \\
    \texttt{g1:=(1,6,5,4);}\\
    \texttt{g2:=(1,5,6)(2,3,4);}\\
    \texttt{B:=[1,2,3,4];}\\

    \item\label{646_2}$(v,k,\lambda)=(6,4,6)$,  $G=\PSL(2,3^2)$ \\
    \texttt{g1:=(1,2,3,4,5);}\\
    \texttt{g2:=(4,5,6);}\\
    \texttt{B:=[1,2,3,4];}\\

    \item\label{731} $(v,k,\lambda)=(7,3,1)$, $G=\PSL(2,7)$ \\
    \texttt{g1:=(1,3,2)(4,5,7);}\\
    \texttt{g2:(1,2,4)(3,7,6);}\\
    \texttt{B:=[1,2,7];}\\

 \item\label{742} $(v,k,\lambda)=(7,4,2)$, $G=\PSL(2,7)$\\
    \texttt{g1:=(1,3,2)(4,5,7);}\\
    \texttt{g2:(1,2,4)(3,7,6);}\\
    \texttt{B:=[1,2,3];}\\   

    \item\label{734}$(v,k,\lambda)=(7,3,4)$, $G=\PSL(2,7)$\\
    \texttt{g1:= (1,3,2)(5,7,6);}\\
    \texttt{g2:=(1,5,3)(2,4,7);}\\
    \texttt{B:=[1,2,3];}\\

    \item\label{1042}$(v,k,\lambda)=(10,4,2)$, $G=\PGaL(2,2^4)$\\
    \texttt{g1:=(1,2)(3,9,7,5)(4,10,8,6);}\\
    \texttt{g2:=(1,9,8)(2,4,7)(3,10,5);}\\
    \texttt{B:=[1,2,4,8];}\\

     \item\label{1133}$(v,k,\lambda)=(11,3,3)$, $G=\PSL(2,11)$\\
    \texttt{g1:= (2,10,8,6,4)(3,11,9,7,5);}\\
    \texttt{g2:=(1,8,7)(2,11,5)(3,10,6);}\\
    \texttt{B:=[1,2,4];}\\

    \item\label{1136}$(v,k,\lambda)=(11,3,6)$, $G=\PSL(2,11)$\\
    \texttt{g1:= (2,10,8,6,4)(3,11,9,7,5);}\\
    \texttt{g2:=(1,8,7)(2,11,5)(3,10,6);}\\
    \texttt{B:=[1,2,3];}\\

    \item\label{1146}$(v,k,\lambda)=(11,4,6)$, $G=\PSL(2,11)$\\
    \texttt{g1:= (2,10,8,6,4)(3,11,9,7,5);}\\
    \texttt{g2:=(1,8,7)(2,11,5)(3,10,6);}\\
    \texttt{B:=[1,2,3,8];}\\

    \item \label{1152}$(v,k,\lambda)=(11,5,2)$ $G=\PSL(2,11)$\\
    \texttt{g1:= (2,10,8,6,4)(3,11,9,7,5);}\\
    \texttt{g2:=(1,8,7)(2,11,5)(3,10,6);}\\
    \texttt{B:=[1,2,3,8,11];}\\

    \item \label{11512}$(v,k,\lambda)=(11,5,12)$, $G=\PSL(2,11)$\\
    \texttt{g1:= (2,10,8,6,4)(3,11,9,7,5);}\\
    \texttt{g2:=(1,8,7)(2,11,5)(3,10,6);}\\
    \texttt{B:=[1,2,3,4,6];}\\

    \item\label{1163}$(v,k,\lambda)=(11,6,3)$,$G=\PSL(2,11)$, \\
    \texttt{g1:= (2,10,8,6,4)(3,11,9,7,5);}\\
    \texttt{g2:=(1,8,7)(2,11,5)(3,10,6);}\\
    \texttt{B:=[1,2,3,4,6,9];}\\

    \item\label{11615}$(v,k,\lambda)=(11,6,15)$, $G=\PSL(2,11)$\\
    \texttt{g1:= (2,10,8,6,4)(3,11,9,7,5);}\\
    \texttt{g2:=(1,8,7)(2,11,5)(3,10,6);}\\
    \texttt{B:=[1,2,3,4,5,11];}\\


    \item\label{1584_2}$(v,k,\lambda)=(15,8,4)$, $G=\PSL(2,3^2)$\\
    \texttt{g1:=(1,12)(2,3)(4,11)(7,14)(9,15)(10,13);}\\
    \texttt{g2:=(1,10)(2,11,5,6)(3,14,12,8)(4,13,9,7);}\\
    \texttt{B:=[1,2,3,4,10,11,12,13];}\\

    \item\label{2832_1} $(v,k,\lambda)=(28,3,2)$, $G=\PGaL(2,2^3)$\\
    {\scriptsize \texttt{g1:=(1,2,5,3,7,6,4)(8,18,13,26,24,21,15)(9,12,28,27,23,19,16)(10,25,22,17,11,20,14);}\\
    \texttt{g2:=(1,26,11)(2,14,12)(3,19,13)(5,18,9)(6,22,8)(7,23,10)(15,24,17)(16,20,25)(21,27,28);}\\
    \texttt{g3:=(1,3,2)(5,7,6)(8,9,10)(11,13,12)(14,26,19)(15,27,20)(16,17,21)(18,23,22)(24,28,25);}}\\
    \texttt{B:=[1,2,16];}\\

\item\label{2832_2} $(v,k,\lambda)=(28,3,2)$, $G=\PGaL(2,2^3)$\\
    {\scriptsize \texttt{g1:=(1,2,5,3,7,6,4)(8,18,13,26,24,21,15)(9,12,28,27,23,19,16)(10,25,22,17,11,20,14);}\\
    \texttt{g2:=(1,26,11)(2,14,12)(3,19,13)(5,18,9)(6,22,8)(7,23,10)(15,24,17)(16,20,25)(21,27,28);}\\
    \texttt{g3:=(1,3,2)(5,7,6)(8,9,10)(11,13,12)(14,26,19)(15,27,20)(16,17,21)(18,23,22)(24,28,25);}}\\
    \texttt{B:=[1,2,14];}\\

    \item\label{2834} $(v,k,\lambda)=(28,3,4)$, $G=\PGaL(2,2^3)$\\
    {\scriptsize \texttt{g1:=(1,2,5,3,7,6,4)(8,18,13,26,24,21,15)(9,12,28,27,23,19,16)(10,25,22,17,11,20,14);}\\
    \texttt{g2:=(1,26,11)(2,14,12)(3,19,13)(5,18,9)(6,22,8)(7,23,10)(15,24,17)(16,20,25)(21,27,28);}\\
    \texttt{g3:=(1,3,2)(5,7,6)(8,9,10)(11,13,12)(14,26,19)(15,27,20)(16,17,21)(18,23,22)(24,28,25);}}\\
    \texttt{B:=[1,2,3];}\\

    \item\label{2865} $(v,k,\lambda)=(28,6,5)$, $G=\PGaL(2,2^3)$\\
    {\scriptsize \texttt{g1:=(1,2,5,3,7,6,4)(8,18,13,26,24,21,15)(9,12,28,27,23,19,16)(10,25,22,17,11,20,14);}\\
    \texttt{g2:=(1,26,11)(2,14,12)(3,19,13)(5,18,9)(6,22,8)(7,23,10)(15,24,17)(16,20,25)(21,27,28);}\\
    \texttt{g3:=(1,3,2)(5,7,6)(8,9,10)(11,13,12)(14,26,19)(15,27,20)(16,17,21)(18,23,22)(24,28,25);}}\\
    \texttt{B:=[1,2,3,16,17,21];}\\

    \item\label{28610_1} $(v,k,\lambda)=(28,6,10)$, $G=\PGaL(2,2^3)$\\
    {\scriptsize \texttt{g1:=(1,2,5,3,7,6,4)(8,18,13,26,24,21,15)(9,12,28,27,23,19,16)(10,25,22,17,11,20,14);}\\
    \texttt{g2:=(1,26,11)(2,14,12)(3,19,13)(5,18,9)(6,22,8)(7,23,10)(15,24,17)(16,20,25)(21,27,28);}\\
    \texttt{g3:=(1,3,2)(5,7,6)(8,9,10)(11,13,12)(14,26,19)(15,27,20)(16,17,21)(18,23,22)(24,28,25);}}\\
    \texttt{B:=[1,2,3,5,14,20];}\\

\item\label{28610_2} $(v,k,\lambda)=(28,6,10)$, $G=\PGaL(2,2^3)$\\
    {\scriptsize \texttt{g1:=(1,2,5,3,7,6,4)(8,18,13,26,24,21,15)(9,12,28,27,23,19,16)(10,25,22,17,11,20,14);}\\
    \texttt{g2:=(1,26,11)(2,14,12)(3,19,13)(5,18,9)(6,22,8)(7,23,10)(15,24,17)(16,20,25)(21,27,28);}\\
    \texttt{g3:=(1,3,2)(5,7,6)(8,9,10)(11,13,12)(14,26,19)(15,27,20)(16,17,21)(18,23,22)(24,28,25);}}\\
    \texttt{B:=[1,2,3,4,5,6];}\\

\item\label{28610_3} $(v,k,\lambda)=(28,6,10)$, $G=\PGaL(2,2^3)$\\
    {\scriptsize \texttt{g1:=(1,2,5,3,7,6,4)(8,18,13,26,24,21,15)(9,12,28,27,23,19,16)(10,25,22,17,11,20,14);}\\
    \texttt{g2:=(1,26,11)(2,14,12)(3,19,13)(5,18,9)(6,22,8)(7,23,10)(15,24,17)(16,20,25)(21,27,28);}\\
    \texttt{g3:=(1,3,2)(5,7,6)(8,9,10)(11,13,12)(14,26,19)(15,27,20)(16,17,21)(18,23,22)(24,28,25);}}\\
    \texttt{B:=[1,2,3,5,8,18];}\\

\item\label{28610_4} $(v,k,\lambda)=(28,6,10)$, $G=\PGaL(2,2^3)$\\
    {\scriptsize \texttt{g1:=(1,2,5,3,7,6,4)(8,18,13,26,24,21,15)(9,12,28,27,23,19,16)(10,25,22,17,11,20,14);}\\
    \texttt{g2:=(1,26,11)(2,14,12)(3,19,13)(5,18,9)(6,22,8)(7,23,10)(15,24,17)(16,20,25)(21,27,28);}\\
    \texttt{g3:=(1,3,2)(5,7,6)(8,9,10)(11,13,12)(14,26,19)(15,27,20)(16,17,21)(18,23,22)(24,28,25);}}\\
    \texttt{B:=[1,2,3,18,22,23];}\\

\item\label{28610_5} $(v,k,\lambda)=(28,6,10)$, $G=\PGaL(2,2^3)$\\
    {\scriptsize \texttt{g1:=(1,2,5,3,7,6,4)(8,18,13,26,24,21,15)(9,12,28,27,23,19,16)(10,25,22,17,11,20,14);}\\
    \texttt{g2:=(1,26,11)(2,14,12)(3,19,13)(5,18,9)(6,22,8)(7,23,10)(15,24,17)(16,20,25)(21,27,28);}\\
    \texttt{g3:=(1,3,2)(5,7,6)(8,9,10)(11,13,12)(14,26,19)(15,27,20)(16,17,21)(18,23,22)(24,28,25);}}\\
    \texttt{B:=[1,2,4,10,12,18];}\\

\item\label{2872} $(v,k,\lambda)=(28,7,2)$, $G=\PSL(2,2^3)$\\
    {\scriptsize 
    \texttt{g1:=(1,26,14,16,8,3,23)(2,12,9,24,10,4,28)(5,15,21,19,6,25,13)(7,27,22,11,17,18,20);}\\
    \texttt{g2:=(1,16,9,20,15,10,5)(2,17,7,21,14,12,6)(3,18,8,19,13,11,4)(22,26,27,24,28,23,25);}\\
    \texttt{g3:=(1,5,12,6,10,15,8)(2,14,4,27,9,24,11)(3,21,7,17,13,23,22)(16,18,25,19,20,26,28);}}\\
    \texttt{B:=[1,2,8,9,19,21,26];}\\ 

\item\label{2876} $(v,k,\lambda)=(28,7,6)$, $G=\PGaL(2,2^3)$\\
    {\scriptsize 
    \texttt{g1:=(1,5,10,15,20,9,16)(2,6,12,14,21,7,17)(3,4,11,13,19,8,18)(22,25,23,28,24,27,26);}\\
    \texttt{g2:=(1,19,13)(2,5,16)(3,12,7)(4,26,15)(6,24,9)(8,21,25)(10,23,17)(11,27,14)(18,20,22);}\\
    \texttt{g3:=(1,2,3)(4,10,21)(5,12,19)(6,11,20)(7,13,16)(8,15,17)(9,14,18)(22,24,27)(23,25,26);}}\\
    \texttt{B:=[1,2,4,12,20,24,27];}\\

    \item\label{2898_1} $(v,k,\lambda)=(28,9,8)$, $G=\PGaL(2,2^3)$\\
    {\scriptsize \texttt{g1:=(1,2,5,3,7,6,4)(8,18,13,26,24,21,15)(9,12,28,27,23,19,16)(10,25,22,17,11,20,14);}\\
    \texttt{g2:=(1,26,11)(2,14,12)(3,19,13)(5,18,9)(6,22,8)(7,23,10)(15,24,17)(16,20,25)(21,27,28);}\\
    \texttt{g3:=(1,3,2)(5,7,6)(8,9,10)(11,13,12)(14,26,19)(15,27,20)(16,17,21)(18,23,22)(24,28,25);}}\\
    \texttt{B:=[1,2,3,5,13,14,15,20,22];}\\

    \item\label{2898_2} $(v,k,\lambda)=(28,9,8)$, $G=\PGaL(2,2^3)$\\
    {\scriptsize \texttt{g1:=(1,2,5,3,7,6,4)(8,18,13,26,24,21,15)(9,12,28,27,23,19,16)(10,25,22,17,11,20,14);}\\
    \texttt{g2:=(1,26,11)(2,14,12)(3,19,13)(5,18,9)(6,22,8)(7,23,10)(15,24,17)(16,20,25)(21,27,28);}\\
    \texttt{g3:=(1,3,2)(5,7,6)(8,9,10)(11,13,12)(14,26,19)(15,27,20)(16,17,21)(18,23,22)(24,28,25);}}\\
    \texttt{B:=[1,2,3,11,12,13,14,19,26];}\\

    \item\label{28916} $(v,k,\lambda)=(28,9,16)$, $G=\PGaL(2,2^3)$\\
    {\scriptsize \texttt{g1:=(1,2,5,3,7,6,4)(8,18,13,26,24,21,15)(9,12,28,27,23,19,16)(10,25,22,17,11,20,14);}\\
    \texttt{g2:=(1,26,11)(2,14,12)(3,19,13)(5,18,9)(6,22,8)(7,23,10)(15,24,17)(16,20,25)(21,27,28);}\\
    \texttt{g3:=(1,3,2)(5,7,6)(8,9,10)(11,13,12)(14,26,19)(15,27,20)(16,17,21)(18,23,22)(24,28,25);}}\\
    \texttt{B:=[1,2,3,5,8,9,15,23,24];}\\

    \item\label{281211_1} $(v,k,\lambda)=(28,12,11)$, $G=\PGaL(2,2^3)$\\
    {\scriptsize \texttt{g1:=(1,2,5,3,7,6,4)(8,18,13,26,24,21,15)(9,12,28,27,23,19,16)(10,25,22,17,11,20,14);}\\
    \texttt{g2:=(1,26,11)(2,14,12)(3,19,13)(5,18,9)(6,22,8)(7,23,10)(15,24,17)(16,20,25)(21,27,28);}\\
    \texttt{g3:=(1,3,2)(5,7,6)(8,9,10)(11,13,12)(14,26,19)(15,27,20)(16,17,21)(18,23,22)(24,28,25);}\\}
    \texttt{B:=[1,2,3,4,5,6,8,10,15,16,22,28];}

    \item\label{281211_2} $(v,k,\lambda)=(28,12,11)$, $G=\PGaL(2,2^3)$\\
    {\scriptsize \texttt{g1:=(1,2,5,3,7,6,4)(8,18,13,26,24,21,15)(9,12,28,27,23,19,16)(10,25,22,17,11,20,14);}\\
    \texttt{g2:=(1,26,11)(2,14,12)(3,19,13)(5,18,9)(6,22,8)(7,23,10)(15,24,17)(16,20,25)(21,27,28);}\\
    \texttt{g3:=(1,3,2)(5,7,6)(8,9,10)(11,13,12)(14,26,19)(15,27,20)(16,17,21)(18,23,22)(24,28,25);}}\\
    \texttt{B:=[1,2,3,5,9,12,13,15,17,25,27,28];}\\

    \item\label{281834} $(v,k,\lambda)=(28,18,34)$, $G=\PGaL(2,2^3)$\\
    {\scriptsize \texttt{g1:=(1,2,5,3,7,6,4)(8,18,13,26,24,21,15)(9,12,28,27,23,19,16)(10,25,22,17,11,20,14);}\\
    \texttt{g2:=(1,26,11)(2,14,12)(3,19,13)(5,18,9)(6,22,8)(7,23,10)(15,24,17)(16,20,25)(21,27,28);}\\
    \texttt{g3:=(1,3,2)(5,7,6)(8,9,10)(11,13,12)(14,26,19)(15,27,20)(16,17,21)(18,23,22)(24,28,25);}}\\
    \texttt{B:=[1,2,3,4,5,6,8,9,10,11,12,14,15,16,17,20,22,28];}\\

    \item\label{3662}$(v,k,\lambda)=(36,6,2)$, $G=\PSL(2,2^3)$\\
    {\scriptsize \texttt{g1:=(1,6,4,2,7,5,3)(8,28,23,14,10,26,19)(9,11,25,16,17,27,21)(12,15,24,18,20,13,22)\\
      (29,34,32,30,35,33,31);}\\
    \texttt{g2:=(1,29,10)(2,25,12)(3,13,5)(4,16,9)(6,18,11)(7,30,8)(14,15,33)(17,34,28)(19,21,31)\\(
      20,22,27)(23,24,35)(26,32,36);}\\}
    \texttt{B:=[1,2,14,15,22,23];}\\

    \item\label{3666_1} $(v,k,\lambda)=(36,6,6)$, $G=\PGaL(2,2^3)$\\
    {\scriptsize \texttt{g1:=(1,7,6,5,4,3,2)(9,15,14,13,12,11,10)(16,21,36,34,31,27,
22)(17,26,20,35,32,28,23)\\(18,30,25,19,33,29,24);}\\
    \texttt{g2:=(1,8,9)(2,13,17)(3,10,19)(4,15,20)(5,11,16)(6,12,21)(7,14,
18)(22,24,28)(23,35,29)\\(25,31,30)(26,34,27)(32,33,36);}\\
    \texttt{g3:=(2,5,3)(4,6,7)(10,13,11)(12,14,15)(16,19,17)(18,20,21)(22,
24,28)(23,34,30)(25,35,27)\\(26,31,29)(32,36,33);}\\}
    \texttt{B:=[1,2,3,12,31,33];}\\

     \item\label{3666_2} $(v,k,\lambda)=(36,6,6)$, $G=\PGaL(2,2^3)$\\
     {\scriptsize \texttt{g1:=(1,7,6,5,4,3,2)(9,15,14,13,12,11,10)(16,21,36,34,31,27,
22)(17,26,20,35,32,28,23)\\(18,30,25,19,33,29,24);}\\
    \texttt{g2:=(1,8,9)(2,13,17)(3,10,19)(4,15,20)(5,11,16)(6,12,21)(7,14,
18)(22,24,28)(23,35,29)\\(25,31,30)(26,34,27)(32,33,36);}\\
    \texttt{g3:=(2,5,3)(4,6,7)(10,13,11)(12,14,15)(16,19,17)(18,20,21)(22,
24,28)(23,34,30)(25,35,27)\\(26,31,29)(32,36,33);}\\}
    \texttt{B:=[1,2,3,14,34,36];}\\

     \item\label{3666_3} $(v,k,\lambda)=(36,6,6)$, $G=\PGaL(2,2^3)$\\
      {\scriptsize \texttt{g1:=(1,7,6,5,4,3,2)(9,15,14,13,12,11,10)(16,21,36,34,31,27,
22)(17,26,20,35,32,28,23)\\(18,30,25,19,33,29,24);}\\
    \texttt{g2:=(1,8,9)(2,13,17)(3,10,19)(4,15,20)(5,11,16)(6,12,21)(7,14,
18)(22,24,28)(23,35,29)\\(25,31,30)(26,34,27)(32,33,36);}\\
    \texttt{g3:=(2,5,3)(4,6,7)(10,13,11)(12,14,15)(16,19,17)(18,20,21)(22,
24,28)(23,34,30)(25,35,27)\\(26,31,29)(32,36,33);}\\}
    \texttt{B:=[1,2,10,12,21,33];}\\

    \item\label{3688} $(v,k,\lambda)=(36,8,8)$,$G=\PGaL(2,3^2)$\\
    {\scriptsize \texttt{g1:=(1,15,14,12,3,9,7,6)(2,16,13,11,4,10,8,5)(17,18,36,33)(19,26,31,20,28,25,35,29)\\(21,22,24,34,23,30,27,32);}\\
    \texttt{g2:=(1,33,28)(2,36,30)(3,9,27)(4,5,29)(6,21,16)(7,26,8)(10,25,
12)(11,32,18)(13,23,14)\\(15,31,17)(19,24,35)(20,34,22);}\\
    \texttt{g3:=(1,2)(3,4)(5,9)(6,10)(7,13)(8,14)(11,15)(12,16)(17,18)(19,
22)(20,24)(21,25)(23,26)\\(27,29)(28,30)(31,32)(33,36)(34,35);}\\}
    \texttt{B:=[1,2,3,4,6,10,12,16];}\\

    \item\label{12084_1} $(v,k,\lambda)=(120,8,4)$, $G=\PSL(2,2^4):Z_2$\\
    {\tiny \texttt{g1:=(3,15,28,4)(5,35,78,49)(6,21,68,53)(7,97,93,8)(9,104,45,27)(10,16,57,96)(11,79,47,64)(12,63,59,72)\\
    (13,84,107,48)(14,24,65,111)(18,43,92,109)(19,30,42,37)(20,62,90,86)(22,50,105,95)(23,41,120,98)\\
    (25,60,100,69)(26,118,70,52)(29,51,102,117)(31,39,112,83)(32,103,110,116)(33,67,73,54)(34,94,66,99)\\
    (36,85,77,55)(38,74,81,61)(44,113,115,108)(46,89,82,75)(56,106,71,101)(58,87,91,80)(76,114)(88,119);}\\
    \texttt{g2:=(1,90,8,72,3,16,60,69,68,50,64,15,101,76,78)(2,25,56,94,54,119,6,112,111,74,10,34,65,23,96)\\(4,17,107,41,27,43,13,114,97,67,102,82,87,18,100)(5,84,39,19,99,45,37,115,89,22,35,71,93,51,85)\\
    (7,21,42,40,38,26,113,59,109,98,47,110,118,62,88)(9,116,83,52,73,92,29,33,70,117,58,49,63,57,79)\\
    (11,31,30,104,105,120,12,44,81,46,20,106,80,86,75)(14,103,61,77,28,32,53,108,95,36,66,91,24,55,48);}\\}
    \texttt{B:=[1,2,11,12,13,25,99,106];}\\

 \item\label{12084_2} $(v,k,\lambda)=(120,8,4)$, $G=\PGaL(2,2^4)$\\
    {\tiny \texttt{g1:=(1,57,54,50,46,9,43,39,35,32,7,25,21,17,14)(2,58,53,49,45,10,44,40,36,31,8,26,22,18,13)\\
    (3,59,55,52,48,12,42,37,34,29,5,27,23,20,15)(4,60,56,51,47,11,41,38,33,30,6,28,24,19,16)\\
    (61,119,87,70,113,88,101,71,120,93,73,69,95,102,81)(62,116,117,105,114,76,97,98,75,94,85,90,65,66,84)\\
    (63,107,103,110,82,89,77,72,104,106,83,111,109,74,78)(64,91,96,112,100,80,99,86,108,67,92,68,115,79,118);}\\
    \texttt{g2:=(1,24,79)(2,52,72)(3,32,69)(4,44,84)(5,51,74)(6,23,86)(7,108,77)(8,103,67)\\(9,31,81)(10,43,90)(11,114,70)(12,119,65)(13,93,80)(14,59,88)(15,58,55)(16,102,62)\\
    (17,111,64)(18,56,45)(19,110,73)(20,53,83)(21,109,66)(22,75,54)(25,97,89)(26,47,76)\\
    (27,98,61)(28,46,33)(29,94,68)(30,91,48)(34,57,39)(35,92,38)(36,50,87)(37,115,85)\\
    (40,118,63)(41,101,78)(42,82,60)(49,112,71)(95,113,107)(96,120,99)(100,105,116)(104,117,106);}\\
    \texttt{g3:=(1,4,2,3)(5,9,6,10)(7,11,8,12)(13,37,25,19)(14,38,26,20)(15,39,28,18)(16,40,27,17)(21,41,49,29)\\
(22,42,50,30)(23,43,51,31)(24,44,52,32)(33,45,55,57)(34,46,56,58)(35,47,53,59)(36,48,54,60)(61,64,62,63)\\
(65,77,70,67)(66,78,71,68)(69,79,84,72)(73,80,85,89)(74,81,86,90)(75,82,87,91)(76,83,88,92)(93,115,97,110)\\
(94,109,101,112)(95,99,105,106)(96,116,104,113)(98,111,102,118)(100,117,107,120)(103,119,108,114);}\\}
    \texttt{B:=[1,2,11,12,17,40,47,59];}\\

    \item\label{12088} $(v,k,\lambda)=(120,8,8)$, $G=\PGaL(2,2^4)$\\
    {\tiny \texttt{g1:=(1,57,54,50,46,9,43,39,35,32,7,25,21,17,14)(2,58,53,49,45,10,44,40,36,31,8,26,22,18,13)\\
    (3,59,55,52,48,12,42,37,34,29,5,27,23,20,15)(4,60,56,51,47,11,41,38,33,30,6,28,24,19,16)\\
    (61,119,87,70,113,88,101,71,120,93,73,69,95,102,81)(62,116,117,105,114,76,97,98,75,94,85,90,65,66,84)\\
    (63,107,103,110,82,89,77,72,104,106,83,111,109,74,78)(64,91,96,112,100,80,99,86,108,67,92,68,115,79,118);}\\
    \texttt{g2:=(1,24,79)(2,52,72)(3,32,69)(4,44,84)(5,51,74)(6,23,86)(7,108,77)(8,103,67)(9,31,81)(10,43,90)\\(11,114,70)(12,119,65)(13,93,80)(14,59,88)(15,58,55)(16,102,62)(17,111,64)(18,56,45)(19,110,73)\\(20,53,83)(21,109,66)(22,75,54)(25,97,89)(26,47,76)(27,98,61)(28,46,33)(29,94,68)(30,91,48)\\(34,57,39)(35,92,38)(36,50,87)(37,115,85)(40,118,63)(41,101,78)(42,82,60)(49,112,71)(95,113,107)\\(96,120,99)(100,105,116)(104,117,106);}\\
    \texttt{g3:=(1,4,2,3)(5,9,6,10)(7,11,8,12)(13,37,25,19)(14,38,26,20)(15,39,28,18)(16,40,27,17)(21,41,49,29)\\
(22,42,50,30)(23,43,51,31)(24,44,52,32)(33,45,55,57)(34,46,56,58)(35,47,53,59)(36,48,54,60)\\(61,64,62,63)
(65,77,70,67)(66,78,71,68)(69,79,84,72)(73,80,85,89)(74,81,86,90)(75,82,87,91)\\(76,83,88,92)
(93,115,97,110)(94,109,101,112)(95,99,105,106)(96,116,104,113)
(98,111,102,118)\\(100,117,107,120)(103,119,108,114);}\\}
    \texttt{B:=[1,2,11,12,13,25,99,106];}\\

    \item\label{49641} $(v,k,\lambda)=(496,4,1)$, $G=\PGaL(2,2^5)$\\
    {\tiny
\texttt{g1:=(1,6,16,39,76,42,88,64,129,143,93,75,149,101,11,27,57,119,124,135,154,111,33,70,139,83,55,107,23,50,96)\\
(2,7,17,36,77,41,90,65,127,141,95,72,146,102,13,29,58,120,125,134,152,113,35,67,138,85,51,108,24,46,97)\\
(3,9,20,40,80,44,87,62,128,142,92,73,147,103,14,26,56,116,122,132,153,115,32,69,140,84,53,110,25,47,98)\\
(4,8,18,37,78,45,89,63,130,144,91,71,148,105,15,30,60,118,121,131,155,114,31,68,137,82,54,106,21,48,99)\\
(5,10,19,38,79,43,86,61,126,145,94,74,150,104,12,28,59,117,123,133,151,112,34,66,136,81,52,109,22,49,100)\\
(156,266,158,160,268,395,477,185,308,405,220,300,426,452,343,454,260,318,399,159,267,238,328,427,393,447,\\
488,157,254,310,362)(161,275,381,350,438,491,345,245,338,378,462,359,228,272,355,460,444,441,494,230,168,\\
282,407,237,217,249,232,278,181,302,428)(162,277,219,211,291,391,202,331,429,483,166,281,406,475,178,298,\\
394,234,284,411,235,214,195,321,352,458,492,487,361,464,437)(163,274,403,481,471,377,439,448,257,223,206,\\
212,290,218,299,425,486,224,306,430,484,496,204,335,373,386,248,264,322,402,469)(164,276,404,482,188,244,\\
339,221,304,392,369,468,435,433,354,242,239,240,333,357,440,390,463,192,319,400,191,174,292,346,344)(165,\\
280,199,326,241,332,431,446,367,366,253,311,380,456,351,226,270,397,348,455,213,187,313,409,365,258,255,\\
314,413,227,271)(167,236,285,410,485,375,364,465,175,293,419,374,459,349,353,372,225,269,396,169,287,415,\\
207,261,316,246,337,417,208,341,251)(170,183,309,189,312,408,387,445,222,305,196,325,424,388,262,233,283,\\
389,451,489,436,231,279,215,296,422,479,434,358,371,186)(171,288,414,263,324,356,368,467,490,363,453,474,\\
473,177,297,423,243,334,265,323,179,200,247,256,315,412,259,317,398,480,347)(172,286,382,457,376,461,442,\\
193,320,173,294,418,432,493,198,184,307,250,342,201,330,205,336,370,194,229,273,401,379,384,466)(176,190,
\\209,340,420,478,472,210,289,416,449,495,383,182,303,252,216,295,421,197,327,385,450,476,470,180,301,360,\\
443,203,329);}\\
\texttt{g2:=(1,308,225)(2,426,190)(3,477,178)(4,267,232)(5,395,205)(6,380,264)(7,417,213)(8,162,146)(9,228,108)\\
(10, 94, 73)(11,275,227)(12, 41,170)(13,166,155)(14,373,250)(15,239, 55)(16,194, 68)(17,440,256)\\
(18,476,244)(19,169,137)(20,153, 33)(21,175,145)(22,183,113)(23,334,202)(24,382,192)(25, 91,212)\\
(26,398,197)(27,203, 52)(28, 77,235)(29,438,204)(30,449,223)(31,324,229)(32,487,158)(34,135,259)\\
(35,109,214)(36,309,207)(37,313,237)(38,435,248)(39,371,220)(40,129,245)(42,132,249)(43,342,221)\\
(44,266,240)(45,112, 60)(46,319,172)(47,273,258)(48,173, 75)(49,411,226)(50, 59,180)(51,450,157)\\
(53,390,236)(54,127,125)(56,317,222)(57,291,171)(58,198, 88)(61,279,217)(62,420,179)(63,124,164)\\
(64,201, 95)(65,253,103)(66,299,246)(67,375,186)(69,251,105)(70,385,231)(71, 84,218)(72,403,181)\\
(74,432,160)(76,143,102)(78,305,241)(79,338,196)(80,210,117)(81,216,122)(82,331,233)(83,151,247)\\
(85,409,167)(86,392,184)(87,296,200)(89,425,182)(90,448,260)(92,429,262)(93,310,189)(96,351,161)\\
(97,168,115)(98,130,208)(99,320,177)(100,272,265)(101,374,156)(104,413,195)(106,389,211)(107,352,242)\\
(110,337,191)(111,325,209)(114,345,159)(116,344,185)(118,367,215)(119,276,234)(120,288,252)(121,378,187)\\
(123,399,255)(126,402,188)(128,307,163)(131,468,176)(133,285,206)(134,142,199)(136,397,238)(138,412,174)\\
(139,486,261)(140,348,193)(141,289,243)(144,300,263)(147,451,219)(148,297,254)(149,322,165)(150,347,230)\\
(152,427,224)(154,485,257)(269,447,445)(270,328,376)(271,333,470)(274,302,339)(277,349,491)(278,353,483)\\
(280,473,372)(281,292,405)(282,359,286)(283,404,354)(284,494,383)(287,360,465)(290,312,358)(293,490,433)\\
(294,415,301)(295,475,457)(298,455,484)(303,416,444)(304,326,366)(306,474,406)(311,459,356)(314,384,318)\\
(315,430,377)(316,461,441)(321,495,355)(323,428,488)(327,469,393)(329,346,456)(330,489,496)(332,479,421)\\
(335,369,350)(336,391,418)(340,454,386)(341,458,394)(343,419,422)(357,437,362)(361,472,466)(363,396,482)\\
(364,370,443)(365,492,471)(368,439,464)(379,410,460)(381,480,452)(387,446,478)(388,493,481)(400,408,462)\\
(401,424,414)(407,434,463)(423,436,442)(431,467,453);}\\
\texttt{g3:=(1,2,4,5,3)(6,17,78,126,56)(7,18,79,128,57)(8,19,80,129,58)(9,16,77,130,59)(10,20,76,127,60)\\
(11,24,82,133,62)(12,25,83,134,63)(13,21,81,132,64)(14,23,85,131,61)(15,22,84,135,65)(26,96,46,106,66)\\
(27,97,48,109,69)(28,98,50,108,68)(29,99,49,110,70)(30,100,47,107,67)(31,104,53,111,72)(32,101,51,114,74)\\
(33,102,54,112,73)(34,103,55,113,71)(35,105,52,115,75)(36,89,150,140,119)(37,86,147,139,120)\\
(38,87,149,138,118)(39,90,148,136,116)(40,88,146,137,117)(41,91,151,142,124)(42,95,155,145,122)\\
(43,92,154,141,121)(44,93,152,144,123)(45,94,153,143,125)(156,157,159,160,158)(161,172,211,246,197)\\
(162,169,210,245,198)(163,171,213,244,196)(164,170,212,247,199)(165,174,215,248,200)(166,175,216,249,201)\\
(167,176,217,250,202)(168,173,214,251,203)(177,226,191,231,204)(178,225,190,232,205)(179,227,192,233,206)\\
(180,228,194,235,208)(181,229,195,236,209)(182,230,193,234,207)(183,218,259,253,239)(184,219,261,252,237)\\
(185,220,260,254,238)(186,223,265,258,242)(187,221,262,257,243)(188,222,264,256,241)(189,224,263,255,240)\\
(266,310,427,300,399)(267,395,477,308,426)(269,340,428,384,437)(270,292,422,469,480)(271,400,489,306,356)\\
(272,273,352,375,449)(274,414,365,433,479)(275,382,331,285,420)(276,309,425,347,348)(277,415,495,460,466)\\
(278,370,284,316,295)(279,373,334,409,468)(280,346,434,481,368)(281,419,327,381,376)(282,418,458,465,303)\\
(283,290,315,431,304)(286,391,341,360,444)(287,416,359,336,394)(288,313,392,451,486)(289,378,342,429,485)\\
(291,417,476,338,307)(293,421,302,401,492)(294,321,410,472,491)(296,322,412,367,435)(297,397,344,371,448)\\
(298,396,478,350,442)(299,398,351,319,389)(301,355,379,361,349)(305,402,317,380,440)(311,333,408,496,474)\\
(312,430,467,326,404)(314,357,445,386,323)(318,362,447,454,488)(320,411,337,385,438)(324,413,390,325,403)\\
(328,405,343,393,452)(329,407,493,464,372)(330,406,459,470,462)(332,339,424,471,490)(335,423,455,482,387)\\
(345,432,487,374,450)(353,443,494,461,475)(354,358,377,453,366)(363,446,369,436,484)(364,383,441,457,483)\\
(388,439,473,456,463);\\}
}
 \texttt{B:=[1,2,116,329];}\\

 \item\label{4962214} $(v,k,\lambda)=(496,22,14)$, $G=\PGaL(2,2^5)$\\
    {\tiny
\texttt{g1:=(1,6,16,39,76,42,88,64,129,143,93,75,149,101,11,27,57,119,124,135,154,111,33,70,139,83,55,107,23,50,96)\\
(2,7,17,36,77,41,90,65,127,141,95,72,146,102,13,29,58,120,125,134,152,113,35,67,138,85,51,108,24,46,97)\\
(3,9,20,40,80,44,87,62,128,142,92,73,147,103,14,26,56,116,122,132,153,115,32,69,140,84,53,110,25,47,98)\\
(4,8,18,37,78,45,89,63,130,144,91,71,148,105,15,30,60,118,121,131,155,114,31,68,137,82,54,106,21,48,99)\\
(5,10,19,38,79,43,86,61,126,145,94,74,150,104,12,28,59,117,123,133,151,112,34,66,136,81,52,109,22,49,100)\\
(156,266,158,160,268,395,477,185,308,405,220,300,426,452,343,454,260,318,399,159,267,238,328,427,393,447,\\
488,157,254,310,362)(161,275,381,350,438,491,345,245,338,378,462,359,228,272,355,460,444,441,494,230,168,\\
282,407,237,217,249,232,278,181,302,428)(162,277,219,211,291,391,202,331,429,483,166,281,406,475,178,298,\\
394,234,284,411,235,214,195,321,352,458,492,487,361,464,437)(163,274,403,481,471,377,439,448,257,223,206,\\
212,290,218,299,425,486,224,306,430,484,496,204,335,373,386,248,264,322,402,469)(164,276,404,482,188,244,\\
339,221,304,392,369,468,435,433,354,242,239,240,333,357,440,390,463,192,319,400,191,174,292,346,344)(165,\\
280,199,326,241,332,431,446,367,366,253,311,380,456,351,226,270,397,348,455,213,187,313,409,365,258,255,\\
314,413,227,271)(167,236,285,410,485,375,364,465,175,293,419,374,459,349,353,372,225,269,396,169,287,415,\\
207,261,316,246,337,417,208,341,251)(170,183,309,189,312,408,387,445,222,305,196,325,424,388,262,233,283,\\
389,451,489,436,231,279,215,296,422,479,434,358,371,186)(171,288,414,263,324,356,368,467,490,363,453,474,\\
473,177,297,423,243,334,265,323,179,200,247,256,315,412,259,317,398,480,347)(172,286,382,457,376,461,442,\\
193,320,173,294,418,432,493,198,184,307,250,342,201,330,205,336,370,194,229,273,401,379,384,466)(176,190,
\\209,340,420,478,472,210,289,416,449,495,383,182,303,252,216,295,421,197,327,385,450,476,470,180,301,360,\\
443,203,329);}\\
\texttt{g2:=(1,308,225)(2,426,190)(3,477,178)(4,267,232)(5,395,205)(6,380,264)(7,417,213)(8,162,146)(9,228,108)\\
(10,94,73)(11,275,227)(12,41,170)(13,166,155)(14,373,250)(15,239,55)(16,194,68)(17,440,256)\\
(18,476,244)(19,169,137)(20,153,33)(21,175,145)(22,183,113)(23,334,202)(24,382,192)(25,91,212)\\
(26,398,197)(27,203,52)(28,77,235)(29,438,204)(30,449,223)(31,324,229)(32,487,158)(34,135,259)\\
(35,109,214)(36,309,207)(37,313,237)(38,435,248)(39,371,220)(40,129,245)(42,132,249)(43,342,221)\\
(44,266,240)(45,112,60)(46,319,172)(47,273,258)(48,173,75)(49,411,226)(50,59,180)(51,450,157)\\
(53,390,236)(54,127,125)(56,317,222)(57,291,171)(58,198,88)(61,279,217)(62,420,179)(63,124,164)\\
(64,201,95)(65,253,103)(66,299,246)(67,375,186)(69,251,105)(70,385,231)(71,84,218)(72,403,181)\\
(74,432,160)(76,143,102)(78,305,241)(79,338,196)(80,210,117)(81,216,122)(82,331,233)(83,151,247)\\
(85,409,167)(86,392,184)(87,296,200)(89,425,182)(90,448,260)(92,429,262)(93,310,189)(96,351,161)\\
(97,168,115)(98,130,208)(99,320,177)(100,272,265)(101,374,156)(104,413,195)(106,389,211)(107,352,242)\\
(110,337,191)(111,325,209)(114,345,159)(116,344,185)(118,367,215)(119,276,234)(120,288,252)(121,378,187)\\
(123,399,255)(126,402,188)(128,307,163)(131,468,176)(133,285,206)(134,142,199)(136,397,238)(138,412,174)\\
(139,486,261)(140,348,193)(141,289,243)(144,300,263)(147,451,219)(148,297,254)(149,322,165)(150,347,230)\\
(152,427,224)(154,485,257)(269,447,445)(270,328,376)(271,333,470)(274,302,339)(277,349,491)(278,353,483)\\
(280,473,372)(281,292,405)(282,359,286)(283,404,354)(284,494,383)(287,360,465)(290,312,358)(293,490,433)\\
(294,415,301)(295,475,457)(298,455,484)(303,416,444)(304,326,366)(306,474,406)(311,459,356)(314,384,318)\\
(315,430,377)(316,461,441)(321,495,355)(323,428,488)(327,469,393)(329,346,456)(330,489,496)(332,479,421)\\
(335,369,350)(336,391,418)(340,454,386)(341,458,394)(343,419,422)(357,437,362)(361,472,466)(363,396,482)\\
(364,370,443)(365,492,471)(368,439,464)(379,410,460)(381,480,452)(387,446,478)(388,493,481)(400,408,462)\\
(401,424,414)(407,434,463)(423,436,442)(431,467,453);}\\
\texttt{g3:=(1,2,4,5,3)(6,17,78,126,56)(7,18,79,128,57)(8,19,80,129,58)(9,16,77,130,59)(10,20,76,127,60)\\
(11,24,82,133,62)(12,25,83,134,63)(13,21,81,132,64)(14,23,85,131,61)(15,22,84,135,65)(26,96,46,106,66)\\
(27,97,48,109,69)(28,98,50,108,68)(29,99,49,110,70)(30,100,47,107,67)(31,104,53,111,72)(32,101,51,114,74)\\
(33,102,54,112,73)(34,103,55,113,71)(35,105,52,115,75)(36,89,150,140,119)(37,86,147,139,120)\\
(38,87,149,138,118)(39,90,148,136,116)(40,88,146,137,117)(41,91,151,142,124)(42,95,155,145,122)\\
(43,92,154,141,121)(44,93,152,144,123)(45,94,153,143,125)(156,157,159,160,158)(161,172,211,246,197)\\
(162,169,210,245,198)(163,171,213,244,196)(164,170,212,247,199)(165,174,215,248,200)(166,175,216,249,201)\\
(167,176,217,250,202)(168,173,214,251,203)(177,226,191,231,204)(178,225,190,232,205)(179,227,192,233,206)\\
(180,228,194,235,208)(181,229,195,236,209)(182,230,193,234,207)(183,218,259,253,239)(184,219,261,252,237)\\
(185,220,260,254,238)(186,223,265,258,242)(187,221,262,257,243)(188,222,264,256,241)(189,224,263,255,240)\\
(266,310,427,300,399)(267,395,477,308,426)(269,340,428,384,437)(270,292,422,469,480)(271,400,489,306,356)\\
(272,273,352,375,449)(274,414,365,433,479)(275,382,331,285,420)(276,309,425,347,348)(277,415,495,460,466)\\
(278,370,284,316,295)(279,373,334,409,468)(280,346,434,481,368)(281,419,327,381,376)(282,418,458,465,303)\\
(283,290,315,431,304)(286,391,341,360,444)(287,416,359,336,394)(288,313,392,451,486)(289,378,342,429,485)\\
(291,417,476,338,307)(293,421,302,401,492)(294,321,410,472,491)(296,322,412,367,435)(297,397,344,371,448)\\
(298,396,478,350,442)(299,398,351,319,389)(301,355,379,361,349)(305,402,317,380,440)(311,333,408,496,474)\\
(312,430,467,326,404)(314,357,445,386,323)(318,362,447,454,488)(320,411,337,385,438)(324,413,390,325,403)\\
(328,405,343,393,452)(329,407,493,464,372)(330,406,459,470,462)(332,339,424,471,490)(335,423,455,482,387)\\
(345,432,487,374,450)(353,443,494,461,475)(354,358,377,453,366)(363,446,369,436,484)(364,383,441,457,483)\\
(388,439,473,456,463);\\}
}
\texttt{B:=[1,2,3,7,29,73,121,170,175,177,187,223,237,259,267,284,308,311,315,\\422,427,486];}\\

\item\label{2016324_1} $(v,k,\lambda)=(2016,32,4)$, $G=\PSL(2,2^6):Z_{2}$\\
    The group $G$ here can be obtained from the copy $\PGaL(2,64)$ generated by as \ref{2016324_2}, say \texttt{H}. Indeed,\\ \texttt{G:=IntermediateSubgroups(H,DerivedSubgroup(H)).subgroups[1];} \\ 
     {\tiny    
\texttt{B:=[1,2,11,108,142,168,209,220,389,427,452,472,520,563,897,924,1052,1074,1117,1123,1158,1416,1421,1456,\\1459,1521,1595,1639,1670,1841,1854,1897];}}\\
 
 \item\label{2016324_2} $(v,k,\lambda)=(2016,32,4)$, $G=\PGaL(2,2^6)$\\
    {\tiny
\texttt{g1:=(1,1946,752,93,1353,1155,733,1157,36,1908,323,1010,526,1614,1239,513,2000,475,312,495,959,1188,\\
1089,640,1988,1532,1849,1502,805,1277,73,753,591,260,413,257,882,68,925,1991,1425,328,1482,1005,\\
1269,683,790,1917,656,1659,274,989,9,561,1710,594,1712,1900,818,1631,1386,782,1803)\\
(2,1881,1422,1860,1652,983,418,1644,550,531,1950,537,1887,915,1764,795,1037,367,1148,1551,1021,1558,\\
804,1913,58,844,266,123,427,169,440,2013,677,1083,1998,300,1680,1313,1295,681,1303,1586,541,1703,\\
1862,684,155,1429,249,1543,336,1905,888,1783,854,823,1404,5,1672,953,220,1281,601)\\
(3,1382,1379,1845,1103,577,1501,45,1468,1341,857,1287,1046,1584,229,227,179,112,1232,1686,1621,1084,\\
1894,128,487,1208,1178,401,803,1583,1378,1936,1756,682,25,4,1727,1349,1907,1825,1125,1585,350,911,\\
902,294,1288,1118,1113,1066,984,6,651,1222,1685,1542,998,1667,1612,630,584,578,460)\\
(7,1335,1853,241,608,791,1604,757,1640,1357,434,1331,528,1743,1310,1951,481,907,1632,1419,811,1837,\\ 
38,1944,499,1132,1506,1036,202,650,125,851,1534,1088,306,1602,1135,686,1996,430,308,1296,117,1994,\\ 
1471,414,1889,46,398,1464,943,1278,203,1243,1846,54,908,153,824,1732,1391,1091,727)\\
(8,1527,364,1726,788,239,1371,1857,1512,237,1638,1330,235,1327,522,1972,374,848,1888,699,1262,516,\\
1137,1266,1006,1600,900,1226,879,422,876,1034,77,774,395,1059,617,773,1130,309,778,1597,1129,431,\\
741,415,1345,108,1822,1078,668,1714,246,1467,136,1570,1961,671,1167,2004,815,151,1383)\\(10,469,1717,1365,1760,1836,1082,1665,713,765,1457,574,654,1121,20,1830,940,934,664,158,1767,1115,\\
1070,1955,1663,626,663,1920,567,130,96,1977,1823,606,600,1500,191,1882,1575,1948,1171,1240,744,233,\\
139,412,84,703,1733,1453,1992,39,369,1759,1205,1248,1960,897,754,491,1355,1833,996)\\
(11,1692,2001,226,635,1186,291,850,1730,540,534,1328,1134, 792,1180,332,1898,760,988,1360,1556,1832,\\
279,1637,396,1724,1282,222,174,1962,517,913,1801,1090,500,1605,1108,1545,1306,443,1535,1595,244,\\
796,1516,26,1694,1591,307,1483,721,403,115,1309,1816,842,834,299,1397,1279,680,1057,1959)\\
(12,159,1061,1935,881,109,14,1691,511,1107,105,1329,178,1928,557,1974,337,816,1366,1385,1746,560,\\
948,1176,1569,1687,262,1324,878,701,1410,1450,366,505,527,1446,1432,1270,892,35,904,1596,1418,980,\\
871,411,467,1610,1820,1029,462,215,74,1668,783,424,1368,580,1696,1159,874,409,1257)\\
(13,1068,1112,293,1618,1190,1854,856,48,806,1528,282,952,602,586,1938,777,2011,236,1234,1651,1127,\\
930,1491,549,1896,304,1042,939,1264,1433,368,225,670,1564,1608,747,1217,1350,447,1826,1124,265,\\
1177,27,454,776,1780,400,1487,1681,238,1578,649,728,1289,633,347,1565,935,910,1428,1265)\\
(15,150,1758,2015,1511,341,861,1223,864,56,642,994,378,992,1883,1748,1321,1851,726,1653,1162,1655,\\
756,553,738,555,1343,716,325,830,1430,1394,1117,867,1025,480,1945,887,1033,655,740,315,1141,314,\\
628,453,385,1739,1980,1514,406,1509,1317,305,421,1493,734, 94,1338,349,1567,1011,232)\\
(16,272,1796,264,1562,168,849,31,839,1765,1400,1136,196,1085,1268,722,1251,1332,390,69,969,1589,695,\\
399,1447,392,1695,1008,938,921,1237,781,632,767,961,884,197,1411,1051,618,838,242,1399,193,1629,452,\\
593,749,510,1352,281,1827,564,1587,1304,1494,1481,135,1531,1856,1519,1214,1572)\\
(17,1484,1071,690,808,1271,1381,917,1790,295,1549,652,1144,1546,598,243,709,1636,99,1688,1187,50,784,\\
1997,439,1949,429,1160,937,44,1431,1544,371,471,1478,1630,1466,924,2003,30,1235,1609,250,846,154,451,\\
819,1776,1138,812,1290,210,1978,1181,446,1174,1590,259,1693,1351,397,1706,1123)\\
(18,394,1505,972,1300,1307,1267,1489,1492,1800,280,944,674,1435,1713,344,370,1902,1890,245,1778,832,\\
843,780,1402,1367,1058,529,607,1721,285,1987,1843,283,1643,1648,1628,1552,536,544,509,1272,1817,932,\\
936,126,37,847,1040,1067,1530,1579,1139,1784,1937,634,171,218,263,1970,450,599,1840)\\
(19,547,1116,1286,993,1074,1762,862,889,905,872,1231,995,957,603,1062,34,625,1892,268,801,1664,1376,\\
1656,551,569,582,558,745,1147,1314,1921,41,1098,318,182,1470,521,1715,1922,1916,438,317,81,175,363,\\
1203,855,1824,1105,1850,1476,951,1318,1163,356,1325,1660,1965,252,941,562,863)\\
(21,723,1049,1009,302,1815,968,659,1999,1577,102,1228,1745,2005,502,615,877,103,637,1690,1522,702,\\
1771,1452,1520,143,1044,1976,189,1657,1347,1979,166,137,1336,1868,278,987,465,1253,1388,1624,1805,\\
976,1197,490,1792,1215,1498,483,979,361,1738,762,1865,380,162,1754,1018,1507,769,696,425)\\
(22,183,1097,1031,1274,1810,1043,785,1989,1369,298,1230,810,1969,79,407,1847,1133,468,1689,974,786,\\
1798,1370,672,1166,973,1967,116,685,1348,1454,207,342,1075,1957,157,1786,675,1244,1424,1623,288,981,\\
1322,489,623,53,1533,1958,287,1606,1741,1273,1794,444,284,1953,1291,570,708,868,1456)\\
(23,669,1020,408,163,1941,1210,1973,711,704,895,1229,110,1095,679,1000,1939,1901,883,240,840,339,\\
1401,717,1539,472,466,572,743,1927,1263,1662,331,1099,1364,565,1499,542,1701,1880,476,1622,1445,208,\\
160,131,1947,775,1802,504,1503,1407,1128,32,1032,219,1170,463,258,1521,798,1891,1374)\\
(24,359,977,75,459,121,2008,1581,971,419,1592,660,568,1354,405,1829,1156,1763,1893,1573,653,62,61,\\
1871,1566,1903,1926,57,493,144,82,814,290,827,1054,1392,950,1807, 95,1861,1246,1165,1633,1315,1639,\\
488,70,1875,525,946,532,1438,1439,1761,1150,885,1540,763,891,1361,858,1835,1016)\\
(28,1744,1859,1705,497,1645,729,1742,514,538,929,986,381,1812,1087,918,899,1143,1417,1557,1145,836,\\
120,482,1932,1598,311,170,809,416,100,1755,441,432,
1588,1779,1560,1276,391,1574,1775,301,1479,\\
1465,198,345,1752,59,1559,1723,247,152,1218,1785,789,835,195,1398,821,990,1056,1182,1173)\\
(29,354,822,692,1390,1441,1055,442,213,1736,212,384,1221,1100,1642,970,1674,1477,535,382,610,1259,\\
1841,1004,1986,248,1793,657,807,149,1380,1340,1844,1104,1047,597,172,97,90,2016,1298,1102,1620,214,\\
1707,91,484,1302,1294,1195,1925,1627,2007,321,1260,508,1242, 829,1146,1393,1910,1773,1027)\\
(33,127,1874,72,714,437,1458,1389,1791,1119,975,596,764,1774,286,1475,1834,1015,113,255,1734,627,\\
383,251,1981,1275,1924,914,880,322,1209,751,1109,1106,997,1828,1158,1729,1709,234,98,1661,787,2002,\\
587,563,576,140,492,1460,1206,916,1154,1421,1474,1571,1149,927,1247,673,1616,1842,223)\\
(40,296,1731,1161,852,1838,1101,859,636,955,1858,758, 51,1358,1626,1831,47,1650,506,1697,886,548,\\
1412,710,402,1142,277,1233,496,923,1434,1423,1753,507,982,1700,216,417,1515,267,474,1673,292,1526,\\
746,605,119,1258,1797,1126,1311,211,1911,1316,724, 60,1333,688,206,967,1867,648,1261)\\
(42,1954,539,725,423,1334,1437,1818,373,706,1436,1192,134,86,1072,1918,1619,1611,271,1227,221,455,\\
106,1199,1196,458,622,1683,942,1185,114,1305,1720,1114,1795,1495,742,1718,303,65,698,693,1448,63,184,\\
604,922,1995,841,1216,893,960,1415,647,1193,978,1870,620,614,1671,1444,1646,1708)\\
(43,1725,1568,1704,386,1877,1442,1254,503,964,365,1236,486,1971,1413,794,1769,1140,1517,1699,1022,611,\\
148,1601,771,1599,165,426,748,420,351,1064,676,1122,1983,1768,1919,1249,873, 88,1409,667,1537,579,\\
1885,104,2006,55,1884,1749,335,1524,1396,1469,766,556,1168,1864,735,188,1339,180,1821)\\
(49,2012,1869,141,1654,1990,1319,132,554,1486,1198,533,501,1872,1245,1929,273,320,346,1702,1770,1553,\\
1518,865,1984,1405,313,697,813,122,167,761,1007,1131,1863,896,1017,1320,1909,1576,552,1250,1963,1964,\\
616,689,933,107,662,962,1002,1038,573,1787,866,1547,1406,825,316,991,1363,592,138)\\
(52,355,1594,730,1003,209,410,1819,719,393,1462,694,678,1152,1384,644,595,949,1855,1561,641,1615,1536,\\
1757,338,477,1485,194,613,523,333,1377,1848,377,324,1323,1789,1788,327,1548,1912,1280,1635,256,1225,\\
687,85,78,1876,1219,470,1886,869,1735,1212,837,1541,1030,793,1811,1086,739,926)\\
(64,1895,177,1039,464,1617,1207,828,1496,478,1508,1194,1026,1175,1497,449, 66,1372,353,1504,164,1679,\\
161,715,387,1459,1403,253,947,954,901,1252,515,343,1555,1065,1678,1649,1213,1750,631,546,147,1669,945,\\
457,736,581,768,1151,310,1048,1451,1050,1809,720,1914,1463,230,270,1153,456,1933)\\
(67,1342,376,1120,1045,205,1523,1956,894,853,590,833,658,999,192,629,485,2014,1777,1711,1480,1666,\\
1513,1799,1641,575,545,1899,276,1982,92,1930,1346,231,352,1077,1563,1813,1096,1184,1923,1241,770,129,\\
269,494,173,645,1866,1722,473,1852,1698,1747,1169,1312,920,1301,755,665,1356,712,737)\\
(71,1737,348,1052,1344,190,156,1079,919,707,700,1238,1449,963,1073,519,1255,1172,1580,646,1110,1670,\\
643,1675,1684,589,624,619,750,1510,931,1968,588,358,898,1041,1416,375,1080,1035,326,1625,1473,275,181,\\
133,1582,448,1224,2010,224,428,1751,1766,1183,1204,1200,1488,1772,1740,1525,802,1719)\\
(76,1952,1658,1293,1593,228,559,1993,89,1202,1443,1292,1806,985,906,1878,1179,1682,319,1094,820,1613,\\
1387,860,362,461,340,1395,817,1189,1024,435,1804,912,1966,1326,691,379,1111,1677,404,966,389,111,\\
1879,524,928,1550,1426,201,585,1001,875,609,870,101,1408,1060,289,1028,146,186,1728)\\
(80,732,1529,1337,831,826,1012,1285,1808,1023,261,772,705,759,1455,1359,1081,1076,639,1716,1440,388,\\
530,965,800,1676,903,2009,666,621,498,1063,1554,956,1985,612,330,1634,779,145,124,518,1839,583,1934,\\
571,187,185,118,1781,1299,1297,1942,87,1873,512,1931,1256,1201,360,1915,176,1904)\\
(83,909,1420,433,661,1093,1092,1943,479,1607,1538,1164,217,445,357,731,1427,436,334,1069,1013,204,142,\\
1308,1940,1284,254,1472,890,1283,1414,297,638,958,1906,200,199,1211,1975,1782,845,1490,799,1897,1375,\\
566,797,1053,1373,372,329,1019,1362,1647,718,1014,1461,543,1814,520,1191,1220,1603);}\\
\texttt{g2:=(1,76,364)(2,971,1321)(3,517,1661)(4,216,218)(5,368,1347)(6,1584,135)(7,784,1792)(8,1626,1658)\\
(9,433,497)(10,489,1964)(11,244,1290)(12,1021,199)(13,391,461)(14,425,1762)(15,184,1511)(16,48,983)\\(17,864,1759)(18,943,1239)(19,403,342)(20,1471,1175)(21,879,901)(22,546,1993)(23,1785,1380)\\
(24,591,1035)(25,1869,462)(26,726,1259)(27,1162,382)(28,1131,451)(29,625,33)(30,569,1079)\\
(31,404,1513)(32,533,1357)(34,127,171)(35,1425,496)(36,1912,656)(37,867,261)(38,1488,1056)\\
(39,1789,1667)(40,1908,708)(41,1053,1896)(42,800,1538)(43,775,663)(44,1832,1790)(45,1225,96)\\
(46,1427,574)(47,335,1778)(49,74,1966)(50,470,1335)
(51,258,838)(52,1804,949)(53,734,1913)\\
(54,1798,1022)(55,230,576)(56,1402,552)(57,1588,1957)(58,1198,1359)(59,1826,1677)(60,588,1149)\\
(61,1758,1566)(62,179,585)(63,1254,155)(64,173,1906)(65,1508,153)(66,2007,394)(67,1462,931)\\
(68,993,820)(69,1013,632)(70,1737,1988)(71,1861,549)(72,732,1186)(73,615,468)(75,603,1339)(77,426,713)\\(78,1432,284)(79,1295,994)(80,293,1534)(81,1955,124)(82,780,1250)(83,1123,814)(84,1373,220)\\
(85,1192,563)(86,1714,1630)(87,1525,1326)(88,629,270)(89,1280,631)(90,1514,680)(91,274,303)\\
(92,823,1755)(93,1612,1554)(94,128,727)(95,1281,1509)(97,1837,1429)(98,788,499)(99,782,1396)\\
(100,654,1974)(101,1108,716)(102,109,252)(103,1564,1970)(104,1257,1707)(105,1011,392)(106,452,1048)\\
(107,1544,1537)(108,544,1611)(110,1833,253)(111,1009,1558)(112,444,610)(113,1939,1210)(114,1246,953)\\
(115,487,681)(116,1816,1116)(117,649,207)(118,899,1822)(119,321,645)(120,1174,2012)(121,1051,752)\\
(122,466,228)(123,1209,1332)(125,701,1542)(126,1482,308)(129,1663,897)(130,1831,174)(131,558,396)\\
(132,1227,826)(133,1249,1164)(134,1071,1137)(136,686,1473)(137,1261,227)(138,304,212)(139,813,1646)\\
(140,1549,1950)(141,1161,1163)(142,694,307)(143,1989,494)(144,650,389)(145,1469,2009)(146,1317,276)\\
(147,1354,484)(148,2013,1936)(149,441,1401)(150,710,917)(151,1893,571)(152,1573,973)(154,947,1774)\\
(156,1026,1655)(157,2014,1114)(158,1614,1866)(159,1695,861)(160,763,505)(161,898,385)(162,914,192)\\
(163,959,200)(164,337,416)(165,1807,1932)(166,289,1220)(167,1059,756)(168,194,556)(169,1623,356)\\
(170,1915,1709)(172,1592,1228)(175,613,1386)(176,1437,697)(177,1285,1959)(178,180,970)(181,213,427)\\
(182,473,488)(183,1917,548)(185,977,940)(186,412,236)
(187,1674,1815)(188,1421,1750)(189,1546,1420)\\(190,1723,1697)(191,1293,930)(193,483,634)(195,1696,1497)(196,1444,636)(197,616,430)(198,1262,688)\\
(201,1878,1270)(202,1585,642)(203,764,435)(204,1415,1020)(205,1591,1870)(206,1967,1942)(208,705,1925)\\
(209,1753,1862)(210,1398,1400)(211,1392,1616)(214,1440,302)(215,1846,1987)(217,322,232)(219,1813,757)\\
(221,1364,1127)(222,1553,749)(223,1751,1958)(224,1818,871)(225,1827,1404)(226,1327,1610)(229,1977,1255)\\
(231,1405,746)(233,1624,787)(234,660,507)(235,388,493)(237,1151,594)(238,794,1825)(239,1986,597)\\
(240,1567,1475)(241,1371,673)(242,809,812)(243,1036,1044)(245,531,600)(246,750,939)(247,1689,691)\\
(248,1960,534)(249,881,526)(250,565,1344)(251,1207,1590)(254,1496,605)(255,1236,1345)(256,1336,1556)\\
(257,592,1498)(259,266,1904)(260,1218,1306)(262,1477,469)
(263,352,1418)(264,1730,1864)(265,1062,1956)\\
(267,1403,357)(268,1368,1034)(269,350,475)(271,1766,1410)(272,551,744)(273,1794,1219)(275,759,1189)\\
(277,1455,1075)(278,667,1568)(279,465,1788)(280,1510,1451)(281,1948,438)(282,1120,1466)(283,1796,524)\\
(285,482,638)(286,767,1061)(287,1765,1146)(288,1779,817)(290,1603,724)(291,1202,740)(292,908,306)\\
(294,445,508)(295,658,522)(296,886,1651)(297,1370,774)(298,1499,945)(299,410,334)(300,846,831)\\
(301,1233,1467)(305,816,758)(309,1539,1738)(310,1672,1803)(311,1708,460)(312,540,1576)(313,751,857)\\
(314,448,564)(315,1112,1157)(316,1090,925)(317,911,1968)
(318,609,550)(319,902,1095)(320,343,1895)\\
(323,1823,519)(324,1474,1072)(325,1570,491)(326,777,1581)(327,1167,537)(328,776,1256)(329,1383,1468)\\
(330,798,995)(331,1632,865)(332,1562,1872)(333,709,1506)
(336,1972,1635)(338,807,714)(339,1644,348)\\
(340,1978,1121)(341,1389,890)(344,479,1173)(345,596,633)(346,1377,785)(347,513,1023)(349,924,567)\\
(351,1299,980)(353,936,442)(354,1399,1230)(355,1289,1000)(358,1916,1706)(359,1563,1314)(360,515,796)\\
(361,1954,584)(362,1629,1342)(363,501,1375)(365,1183,966)(366,1698,1530)(367,1911,678)(369,1180,1101)\\
(370,1721,601)(371,1531,683)(372,1491,1460)(373,1267,900)(374,1456,671)(375,1800,1483)(376,1834,1363)\\
(377,1810,892)(378,1969,420)(379,1609,822)(380,1067,1854)(381,1494,987)(383,1107,390)(384,647,1443)\\
(386,1213,894)(387,1216,1088)(393,679,1574)(395,1522,1891)(397,557,670)(398,685,1681)(399,1365,910)\\
(400,595,1973)(401,1181,1504)(402,1947,684)(405,1434,1490)(406,1532,1927)(407,1662,598)(408,627,1223)\\
(409,1194,1536)(411,1391,1341)(413,1193,1874)(414,1559,1276)
(415,1577,464)(417,1027,446)(418,920,789)\\
(419,1140,1182)(421,1178,1892)(422,532,972)(423,1097,858)(424,2005,779)(428,1729,810)(429,459,641)\\
(431,1899,1128)(432,640,635)(434,530,1142)(436,1786,695)(437,554,1685)(439,1197,1975)(440,611,682)\\
(443,1324,815)(447,1937,1068)(449,1776,1349)(450,1502,1242)(453,1130,1654)(454,1057,1448)(455,946,570)\\
(456,772,884)(457,1238,1648)(458,1414,1839)(463,1692,607)(467,1875,492)(471,1300,1074)(472,498,1318)\\
(474,1486,1747)(476,988,1715)(477,657,851)(478,1680,842)(480,706,1812)(481,1360,1019)(485,1850,1428)\\
(486,1126,728)(490,1430,1640)(495,1879,562)(500,510,1749)(502,653,1694)(503,582,535)(504,1606,1050)\\
(506,761,1231)(509,651,1941)(511,1935,1229)(512,1678,1251)(514,1222,1933)(516,952,1449)(518,737,1237)\\
(520,1518,801)(521,1304,528)(523,1814,1033)(525,606,1212)
(527,602,1369)(529,1185,1561)(536,1076,1454)\\
(538,1619,1117)(539,976,1204)(541,1266,1702)(542,1214,1540)(543,1719,560)(545,965,1003)(547,841,698)\\
(553,1301,1517)(555,1669,979)(559,1388,799)(561,1985,1501)(566,1886,830)(568,575,578)(572,955,1136)\\
(573,747,954)(577,1144,612)(579,1248,1701)(580,1883,1858)(581,1900,1860)(583,1356,1152)(586,896,1195)\\
(587,1240,1274)(589,1453,790)(590,999,1110)(593,1524,1735)(599,1924,1682)(604,795,1919)(608,913,1150)\\
(614,1422,2008)(617,1615,951)(618,1358,942)(619,1909,1999)(620,1187,783)(621,985,1897)(622,990,1379)\\
(623,1997,1032)(624,1793,933)(626,1047,1801)(628,1599,1982)(630,1063,1949)(637,1030,1526)(639,1600,729)\\
(643,1268,1855)(644,1124,1727)(646,1535,1322)(648,2015,1431)(652,1086,1633)(655,1070,1888)(659,1479,863)\\
(661,1586,1329)(662,1871,962)(664,866,1847)(665,876,1099)
(666,1507,1385)(668,1551,1929)(669,1080,854)\\
(672,1308,921)(674,1094,1352)(675,1607,778)(676,1015,1226)(677,872,974)(687,1770,1705)(689,1617,1618)\\
(690,1643,1031)(692,1277,1040)(693,1296,889)(696,950,1004)(699,802,1243)(700,1819,1799)(702,1103,1196)\\
(703,888,1058)(704,1621,928)(707,1084,1656)(711,719,1752)(712,1054,1845)(715,1928,793)(717,762,771)\\
(718,1206,2011)(720,1100,827)(721,844,1459)(722,1278,1787)(723,991,963)(725,736,849)(730,932,873)\\
(731,1930,1492)(733,1642,893)(735,1145,850)(738,1572,1200)(739,1671,1890)(741,1817,1361)(742,1069,840)\\
(743,1902,1595)(745,1951,1519)(748,1578,1208)(753,1851,1170)(754,1376,903)(755,1748,1593)(760,1029,1077)\\
(765,1745,769)(766,818,1688)(768,1527,1946)(770,1991,1894)(773,957,1811)(781,848,1221)(786,1756,1424)\\
(791,1598,1582)(792,1726,1433)(797,1668,1625)(803,1265,1594)
(804,1125,1309)(805,880,1791)\\
(806,1652,1805)(808,941,1272)(811,1722,1263)(819,1910,1258)(821,1842,1012)(824,1217,1337)(825,832,1338)\\
(828,1767,895)(829,1351,1395)(833,937,2004)(834,1628,1052)(835,1159,1082)(836,1115,1843)(837,1732,1844)\\(839,1784,1452)(843,1156,1406)(845,1829,2003)(847,1550,975)(852,1806,1898)(853,1119,877)(855,1907,992)\\
(856,874,1235)(859,1638,1728)(860,2010,1122)(862,1039,1311)(868,1060,1914)(869,1664,1631)(870,1495,1294)\\(875,1650,1282)(878,916,1016)(882,883,1093)(885,1328,1944)(887,1976,1419)(891,1190,1153)(904,1741,1042)\\(905,1325,1372)(906,1113,1903)(907,967,1934)(909,1983,1547)(912,1981,1996)(915,1269,1691)(918,1441,1346)\\
(919,923,2002)(922,1516,1565)(926,1411,1580)(927,1979,1575)
(929,1868,1856)(934,1620,1820)\\
(935,1298,1739)(938,1177,1830)(944,1478,1291)(948,1863,1024)(956,1172,1675)(958,1943,1773)\\
(960,1205,1316)(961,1458,1998)(964,1782,1963)(968,1613,1303)(969,1302,1961)(978,1523,1166)\\
(981,1754,1849)(982,1252,1965)(984,1684,1992)(986,1543,1887)(989,1821,1743)(996,1096,1005)\\
(997,1795,1089)(998,1330,1647)(1001,1634,1191)(1002,1725,1381)(1006,1288,1447)(1007,1378,1569)\\(1008,1232,1078)(1010,1481,1271)(1014,1500,1881)(1017,1188,1520)(1018,1279,1926)(1025,1840,1931)\\
(1028,1521,1990)(1037,1880,1867)(1038,1104,1073)(1041,1731,1560)(1043,1118,2016)(1045,1148,1853)\\(1046,1307,1901)(1049,1286,1465)(1055,1515,1184)(1064,1487,1700)(1065,1287,1087)(1066,1412,1253)\\(1081,1971,1450)(1083,1627,1740)(1085,1165,1710)(1091,1763,1387)(1092,1923,1555)(1098,1938,1461)\\(1102,1922,1768)(1105,1139,1828)(1106,1545,1529)(1109,1340,1541)(1111,1836,1138)(1129,1366,1203)\\
(1132,1703,1711)(1133,2001,1670)(1134,1382,1247)(1135,1176,1713)(1141,1771,1809)(1143,1394,1147)\\(1154,1905,1160)(1155,1310,1885)(1158,1824,1393)(1168,1838,1367)(1169,1260,1742)(1171,1571,1953)\\(1179,1436,1409)(1199,1224,1657)(1201,1848,1980)(1211,1353,1417)(1215,1241,1273)(1234,1407,1334)\\(1244,1602,1601)(1245,1323,1557)(1264,1484,1480)(1275,1775,1780)(1283,1442,1772)(1284,1552,1312)\\
(1292,1921,1783)(1297,1984,1313)(1305,1746,1636)(1315,2006,1476)(1319,1882,1683)(1320,1734,1666)\\(1331,1362,1637)(1333,2000,1446)(1343,1659,1350)(1348,1781,1489)(1355,1744,1435)(1374,1699,1690)\\(1384,1877,1505)(1390,1472,1686)(1397,1777,1720)(1408,1463,1548)(1413,1533,1493)(1416,1503,1693)\\(1423,1920,1718)(1426,1704,1918)(1438,1649,1608)(1439,1622,1596)(1445,1736,1873)(1457,1769,1597)\\
(1464,1995,1660)(1470,1945,1645)(1485,1733,1835)(1512,1952,1712)(1528,1857,1962)(1579,1797,1583)\\(1587,1802,1761)(1589,1641,1676)(1604,1841,1764)(1605,1859,1884)(1639,1716,1717)(1653,1876,1808)\\(1665,1940,1994)(1673,1757,1865)(1679,1889,1760)(1687,1724,1852);}\\
\texttt{g3:=(1,714,163,1061,928,78)(2,780,747,1311,1464,1812)(3,623,815,829,334,1481)(4,207,1467,1841,638,16)
(5,1778,293,1731,1132,1744)(6,1424,671,1986,958,632)(7,1145,300,1800,1264,1831)(8,657,433,1332,1046,1741)
(9,1209,1901,904,1658,477)(10,21,1768,800,448,575)(11,1790)(12,524,644,1917,673,23)(13,1753,308,1775,684,
1713)(14,1426,1876,591,1734,110)(15,61,662,1895,1325,698)(17,1279,1381,1959,30,1591)(18,633,1867,1846,59,367)
(19,134,1514,1829,1872,449)(20,659,1254,261,326,1747)(22,773,1642,890,1214,25)(24,1320,253,1762,1918,1653)
(26,1123,222,1181,1816,1630)(27,724,1278,1087,677,544)(28,844,1300,1565,648,1944)(29,204,1695,1113,785,237)
(31,45,468,1345,1925,797)(32,1696,1292,739,1157,234)(33,1020,159,906,1377,594)(34,221,305,1763,1245,581)
(35,912,393,1269,1924,1000)(36,1791,1210,1935,1878,52)(37,777,474,1471,1479,155)(38,1755,531,1970,728,967)
(39,1754,676,1873,1719,1342)(40,499,441,1586,674,1265)(41,942,94,1573,1929,66)(42,314,568,1964,901,1922)
(43,779,1183,1356,1453,162)(44,244)(46,381,2013,1489,939,119)(47,1534,1560,1703,1435,368)(48,216,117,301,1764,
285)(49,631,175,1795,1567,1439)(50,1306,1160,1595,397,1282)(51,125,1779,537,450,1491)(53,422,212,254,1519,1379)
(54,899,1083,1492,910,1261)(55,498,224,755,1733,1520)(56,891,141,1508,1147,614)(57,1702,353,1098,1444,642)
(58,1648,265,1316,1419,1182)(60,811,100,1303,932,776)(62,273,768,1824,742,1011)(63,867,1581,1486,1175,863)
(64,551,1196,1493,1893,933)(65,1394,2008,552,1153,1163)(67,158,1624,771,1299,1525)(68,1474,1927,467,228,793)
(69,1584,786,374,597,1943)(70,1363,828,252,1436,1980)(71,231,1760,1049,1919,1931)(72,565,1029,1111,256,782)
(73,255,208,1668,404,1225)(74,89,613,1089,1475,1622)(75,132,1026,1314,1683,555)(76,338,959,286,704,1432)
(77,214,1164,1589,229,1273)(79,741,1195,357,1304,460)(80,1041,1711,1355,1347,335)(81,539,628,1639,991,1213)
(82,573,1039,1660,706,1321)(83,749,401,116,1714,1259)(84,1452,794,1455,1473,1169)(85,753,916,1521,409,1879)
(86,726,814,1553,1504,1116)(87,931,1866,233,762,351)(88,903,1224,1301,940,1999)(90,845,1562,1349,342,1137)
(91,217,1587,1378,868,617)(92,600,103,486,1934,1488)(93,1729,1128,262,289,1735)(95,1406,1617,1965,693,1117)
(96,615,1236,705,1625,1513)(97,479,1827,1583,685,1262)(98,1401,1329,1094,1594,1908)(99,1186,250,760,1138,1556)
(101,641,312,1774,476,1746)(102,1409,1201,1737,376,1767)(104,621,1580,770,1171,1522)\\(105,585,1219,1946,1158,840)
(106,421,946,167,720,801)(107,515,1916,1720,349,653)\\(108,382,1603,722,1287,974)(109,1179,1848,1988,1834,895)
(111,1384,513,2002,219,1159)\\(112,1370,848,1390,1308,1008)(113,1229,1270,1804,1912,805)(114,1338,1438,761,387,1470)\\
(115,1478,792,819,1360,812)(118,1740,1930,1365,137,1022)(120,1680,1272,454,710,608)\\(121,554,945,1203,1495,830)
(122,715,993,1072,406,1875)(123,1307,1042,1626,851,497)\\(124,1204,1982,1717,723,386)(126,1780,402,1135,391,1862)
(127,1364,816,1387,1323,1900)\\(128,973,668,2007,1053,618)(129,765,278,1469,956,1110)(130,1792,735,831,375,1480)\\
(131,783,1202,1086,925,251)(133,1641,1977,1215,426,1440)(135,1756,1456,239,1773,1013)\\(136,1260,1373,839,1825,1075)
(138,1496,745,622,734,532)(139,1771,2006,176,898,473)\\(140,463,874,1806,333,640)(142,1531,1382,407,1327,822)
(143,1769,187,1772,737,1992)\\(144,1770,310,1105,1334,1883)(145,1675,999,897,769,1568)(146,1212,257,1421,798,1410)\\
(147,363,725,992,858,1654)(148,1781,750,645,574,987)(149,661,1352,803,708,788)\\
(150,1150,2012,478,941,1448)
(151,1793,1906,272,1103,1533)(152,1551,394,1896,1358,727)\\
(153,1143,1422,1402,1190,1161)(154,988,1636,1108,429,540)
(156,1563,1882,979,1339,826)\\(157,1570,1004,909,1251,911)(160,948,1408,869,413,627)(161,182,1646,1141,1315,697)\\
(164,318,1185,1343,1392,865)(165,1297,1510,485,1920,490)(166,1524,1554,428,1899,469)\\(168,1501,1958,876,1620,799)
(169,1267,935,1434,430,514)(170,1644,263,930,1526,1310)\\(171,1127,758,1091,1557,1860)(172,1782,884,1084,1166,1888)
(173,744,702,1413,571,624)\\(174,784,850,846,1483,1484)(177,569,1870,864,1926,1002)
(178,1001,1855,475,587,1539)\\
(179,1794,1527,1393,1069,392)(180,1634,643,269,1240,1738)(181,1312,1830,1253,1601,639)\\(183,1371,1380,199,1796,1727)
(184,1758,1871,320,1372,547)(185,619,629,754,1657,1517)\\(186,1757,989,1842,669,366)(188,330,1766,276,606,1498)
(189,1749,1063,646,833,1960)\\(190,1045,1115,1805,1168,1337)(191,637,1885,1915,348,352)(192,663,2005,365,583,589)\\
(193,1208,1967,699,1441,1014)(194,561,223,1099,1820,559)(195,1783,1890,1112,507,1357)\\(196,1585,1957,1266,442,1461)
(197,1222,1623,900,1298,1490)(198,1429,344,225,1797,398)\\(200,767,1621,284,364,807)(201,949,656,880,1939,511)
(202,1588,541,936,1938,636)\\(203,1752,249,847,400,1333)(205,1759,1018,1725,1285,1079)(206,1243,918,1881,634,1578)\\
(209,1482,1275,1947,1176,461)(210,1309,690,1962,446,1637)(211,1632,416,550,529,1350)\\(213,1284,938,1066,1244,1600)
(215,1677,678,1239,596,1880)(218,649,1233,1604,482,983)\\(220,1784,238,277,686,1742)(226,1609,307,371,1134,598)
(227,1798,1822,1627,731,399)\\(230,1318,1818,1748,1835,1017)(232,1761,1007,1914,951,303)(235,1104,1092,281,584,1969)\\
(236,1673,1743,821,854,1530)(240,1596,1966,1280,1386,997)(241,1785,1021,283,1124,548)\\(242,487,1291,1726,1910,1019)
(243,1605,259,834,1466,1180)(245,1564,1126,907,1056,823)\\(246,321,1427,1494,1936,1348)(247,336,1040,2011,955,1391)
(248,1420,510,630,1322,2004)\\(258,701,817,327,1502,1460)(260,1154,743,557,875,470)(264,577,1133,522,1047,1975)\\
(266,1628,282,1101,1036,1859)(267,528,311,1295,1817,586)(268,647,1033,1165,313,1050)\\(270,1376,1199,553,977,1909)
(271,1317,525,592,1750,995)(274,322,1802,1569,1060,1615)\\(275,1799,996,1979,1699,1081)(279,1688,1545,1693,1724,1978)
(280,1428,1423,1640,1932,1313)\\(287,1383,1146,329,1765,1468)(288,1167,1242,372,838,1542)(290,1787,1065,872,1611,1162)\\
(291,652,1090,1590,396,1187)(292,414,986,601,1937,1651)(294,1786,1006,2016,1607,969\\)(295,1801,439,1730,1144,2001)
(296,908,1559,1543,370,1068)(297,781,1686,672,1078,610)\\(298,1129,1294,566,1051,1685)(299,1271,517,1997,443,1351)
(302,1249,1676,1255,1241,713)\\(304,605,1951,990,1913,972)(306,1276,1887,1721,447,886)(309,970,1191,1268,350,1274)\\
(315,660,533,1497,1921,1671)(316,1207,558,620,740,1633)(317,1114,325,950,1547,1678)\\(319,377,1712,1109,883,1107)
(323,1119,1701,1385,860,1003)(324,1532,492,466,1928,820)\\(328,1149,1263,1974,1613,730)(331,337,609,1561,1659,1616)
(332,1776,403,1071,680,2003)\\(339,1418,1293,1030,1155,1709)(340,1788,818,1106,1503,1687)(341,1540,1869,546,957,893)\\
(343,905,841,1511,1566,962)(345,1037,1987,1826,1911,943)(346,1151,625,1415,861,1903)\\(347,923,757,835,1558,1505)
(354,1647,1399,998,1230,1330)(355,1425,1981,679,1691,1682)\\(356,373,1739,405,616,1933)(358,1777,567,502,1537,360)
(359,1319,1669,603,1227,1655)\\(361,420,388,1073,1923,1948)(362,1789,1849,1015,1445,560)(369,1044,1884,1904,1052,1120)\\
(378,493,1038,1555,438,1954)(379,694,683,1247,1662,1610)(380,1064,530,1582,920,703)\\(383,1095,892,1952,1541,882)
(384,1814,1085,857,1606,1034)(385,488,813,1809,1476,1437)\\(389,687,1803,1828,1407,580)(390,902,1810,1512,1340,1093)
(395,1707,1897,1411,1894,1953)\\(408,505,1024,1548,1631,1874)(410,1010,1661,1032,1324,1395)(411,691,1635,1277,1206,572)\\
(412,1865,1821,1012,1080,1698)(415,535,1414,1572,1845,1847)(417,1331,1398,804,1643,1528)\\(418,1058,1854,1700,434,1598)
(419,1198,457,1062,960,887)(423,453,1354,501,736,855)\\(424,966,1152,790,914,775)(425,1704,1808,1035,1666,1833)
(427,1552,1177,1412,1602,729)\\(431,1477,1220,593,1612,810)(432,1950,607,856,852,1506)(435,719,526,975,1973,1446)\\
(436,1447,1118,675,1961,508)(437,1499,1366,870,1886,752)(440,536,952,1515,1994,929)\\(444,774,1221,1472,921,1232)
(445,695,1288,1043,778,1302)(451,721,471,1516,1431,1328)\\(452,1178,570,1059,1100,520)(455,756,827,1518,1048,1892)
(456,1656,978,655,1592,1963)\\(458,738,1054,866,464,582)(459,1576,947,1715,1305,716)(462,1993,1811,733,1389,542)\\
(465,766,1529,1416,1852,1121)(472,871,1593,1485,1188,576)(480,971,1250,954,862,1995)\\(481,809,681,944,1433,1258)
(483,748,1942,802,1346,1500)(484,1375,849,1907,1097,1130)\\(489,879,1102,1538,564,578)(491,696,1983,512,1968,2014)
(494,1457,1815,1442,518,1684)\\(495,563,1170,878,1028,1536)(496,1996,1574,915,599,1289)(500,1949,1535,937,534,1546)\\
(503,1839,1200,712,664,1577)(504,1368,1443,523,1710,751)(506,1088,1645,953,832,1608)\\(509,602,859,1732,1417,1998)
(516,1055,1940,1856,682,1454)(519,853,1248,1507,1122,965)\\(521,1708,994,1361,896,1403)(527,1728,837,1353,1458,1941)
(538,1281,843,1618,982,1296)\\(543,1629,1667,981,1226,1736)(545,1823,877,579,666,2010)(549,746,1889,1465,795,1840)\\
(556,1985,1751,665,934,1388)(562,1192,1851,1016,1990,1194)(588,1722,654,968,873,1256)\\(590,1663,1745,667,2009,1172)
(595,2000,764,711,881,1550)(604,1025,1861,1405,1679,1286)\\(611,1076,1449,1184,1665,1868)(612,1670,658,626,1197,1864)
(635,1549,1692,1235,1898,709)\\(650,1705,1672,1139,1234,1858)(651,1369,1638,1844,1211,961)(670,1650,1335,1218,1905,1902)\\
(688,1837,1173,1404,1579,1681)(689,1252,889,1619,1509,1156)(692,718,1136,1341,1689,1972)\\(700,1096,1575,1690,1971,759)
(707,1523,1070,1228,964,772)(717,980,1326,1462,1614,787)\\(732,1344,1077,1836,1336,1396)(763,1863,1463,1664,1193,1223)
(789,888,1067,1487,1142,791)\\(796,1706,1397,924,1694,1544)(806,1838,824,1723,1148,1843)(808,1057,917,913,1174,842)\\
(825,1649,1231,1216,1945,1246)(836,1652,1367,1217,1697,1853)(885,1131,1459,1074,922,2015)\\(894,1955,976,1599,1716,963)
(919,1813,1082,1009,1877,1023)(926,1991,1571,1891,1257,985)\\(927,1374,1450,1189,1819,1005)(984,1989,1597,1674,1283,1237)
(1027,1362,1400,1125,1031,1857)\\(1140,1359,1238,1956,1205,1976)
(1290,1832)(1430,1807,1984,1451,1850,1718);}}\\
\texttt{B:=[1,2,11,108,142,168,209,220,389,427,452,472,520,
563,897,924,1052,1074,\\1117,1123,1158,1416,1421,
1456,1459,1521,1595,1639,1670,1841,1854,1897];}\\

\item\label{20163212_1} $(v,k,\lambda)=(2016,32,1)$, $G=\PGaL(2,2^6)$\\
The generators of $G$ are the same as in case (\ref{2016324_2}). \\
\texttt{ B:=[1,2,7,9,12,182,349,480,492,502,556,597,731,753,
794,813,844,892,959,\\1131,1328,1349,1362,1397,1468,
1475,1621,1627,1660,1685,1771,1846];}\\

\item\label{20163212_2} $(v,k,\lambda)=(2016,32,1)$, $G=\PGaL(2,2^6)$\\
The generators of $G$ are the same as in case (\ref{2016324_2}). \\
\texttt{ B:=[1,2,15,154,179,489,496,558,568,600,685,714,753,
827,1205,1217,1243,\\1256,1352,1373,1390,1402,1443,
1527,1607,1719,1750,1751,1770,1789,1945,1961];}\\
\end{enumerate}

\textbf{Acknowledgements.}
The first author's work is supported by the National Natural Science Foundation of China (No:12101120).

All of the authors thank the Italian National Group for Algebraic and Geometric Structures and their Applications (GNSAGA–INdAM) for its support to their research.

The work was completed when the fifth author was visiting the first author at Foshan University in Foshan and Prof. Shenglin Zhou at the South China University of Technology in Guangzhou. The author is grateful to both colleagues for financing and supporting his visit.

\end{document}